\documentclass[11pt]{amsart}
\usepackage{enumitem}
\usepackage{color}
\usepackage{amssymb,verbatim}
\usepackage{mathrsfs}
\usepackage{dsfont}
\usepackage{tikz-cd} 
\usepackage[colorlinks,citecolor=blue,urlcolor=black,linkcolor=blue]{hyperref}

\newtheorem*{MAIN}{Main Theorem}
\newtheorem{theorem}{Theorem}[section]
\newtheorem*{theorem*}{Theorem}
\newtheorem*{question*}{Question}
\newtheorem*{definition*}{Definition}
\newtheorem{prop}[theorem]{Proposition}

\newtheorem{claim}{Claim}[theorem]
\newtheorem{subclaim}{Subclaim}[claim]

\newtheorem{lemma}[theorem]{Lemma}
\newtheorem{cor}[theorem]{Corollary}

\newtheorem{question}{Question}

\makeatother

\theoremstyle{definition}
\newtheorem{definition}[theorem]{Definition}
\newtheorem{notation}[theorem]{Notation}
\newtheorem{conv}[theorem]{Convention}

\newtheorem*{setup}{Setup}

\theoremstyle{remark}
\newtheorem{remark}[theorem]{Remark}

\def\s{\subseteq}
\def\sq{\sqsubseteq}
\def\forces{\Vdash}
\def\br{\blacktriangleright}

\newcommand{\one}{\mathop{1\hskip-3pt {\rm l}}} 
\newcommand{\Ult}{\mathrm{Ult}}
\renewcommand{\restriction}
{\mathbin\upharpoonright}    	
\renewcommand{\mid}{\mathrel{|}\allowbreak}

\newcommand{\diagonal}{\bigtriangleup}

\newcommand{\cat}{{}^{\curvearrowright}}
\DeclareMathOperator{\stem}{stem}

\newcommand\old[1]{}

\DeclareMathOperator{\range}{range}
\DeclareMathOperator{\supp}{supp}

\DeclareMathOperator{\crit}{crit}
\DeclareMathOperator{\Succ}{Succ}

\DeclareMathOperator{\eval}{eval}

\DeclareMathOperator{\dom}{dom}

\DeclareMathOperator{\gp}{GP}

\DeclareMathOperator{\otp}{otp}

\DeclareMathOperator{\cf}{cf}

\DeclareMathOperator{\id}{id}
\DeclareMathOperator{\RK}{RK}
\DeclareMathOperator{\rk}{\leq_{RK}}
\DeclareMathOperator{\image}{''}

\newcommand{\GCH}{\mathrm{GCH}}
\newcommand{\Ord}{\mathrm{Ord}}

\title[The directedness of the Rudin-Keisler order]{The directedness of the Rudin-Keisler order at measurable cardinals }
\author[Hayut]{Yair Hayut}
\address[Hayut]{Einstein Institute of Mathematics, Hebrew University of Jerusalem, Givat-Ram, 91904, Jerusalem, Israel, \url{https://math.huji.ac.il/~yairhayut/}.}
\email{yair.hayut@mail.huji.ac.il}
\author[Poveda]{Alejandro Poveda}
\address[Poveda]{Universitat de Barcelona, Departament de Matemàtiques i Informàtica, Barcelona 08007, Catalonia, Spain \\
\url{https://alejandropovedaruzafa.com/}.}
\email{alejandro.poveda@ub.edu}

\subjclass[2020]{03E35, 03E55}
\keywords{Rudin-Keisler order, Ultrafilters, Gluing property.}
\thanks{ The first author acknowledges the support of {the Israel Science   Foundation through the grant 1967/21}. The second author was partially supported by project number PID2023-147428NB-I00 funded by the the Spanish Government. }

\begin{document}
\maketitle

\begin{abstract}
The manuscript is concerned with the Rudin-Keisler order of ultrafilters on measurable cardinals. The main theorem proved  reads as follows: Given regular cardinals $\lambda\leq \kappa$, the following  theories are equiconsistent modulo ZFC:
\begin{enumerate}
    \item $\kappa$ is a measurable cardinal with $o(\kappa)=\lambda^+$ (resp. $o(\kappa)=\kappa$).
    \item The Rudin-Keisler order restricted to the set of $\kappa$-complete (non-principal) ultrafilters on $\kappa$ is $\lambda^+$-directed (resp. $\kappa^+$-directed).
\end{enumerate} 
The theorem reported here is proved after bridging the directedness of the RK-order  with the $\lambda$-Gluing Property introduced by the authors in \cite{HP}. Our result provides what seems to be  the first example of a compactness-type property at the level of  measurable cardinals whose consistency strength is much lower than the existence of a strong cardinal. As part of our analysis  we also answer a question of  Gitik by showing that in the model of \cite{ChangingCofinalities} the $\aleph_0$-Gluing Property fails and as a consequence that  the Rudin-Keisler order  is not even $\aleph_1$-directed. 
\end{abstract}


\section{Introduction}

The present manuscript is concerned with the {Rudin-Keisler order} on $\kappa$-complete non-principal ultrafilters on (an uncountable) regular cardinal. A set $\mathcal{F}\s \mathcal{P}(\kappa)$ is called  a \emph{filter} if $\mathcal{F}$ is stable under finite intersections and supersets, $\kappa\in \mathcal{F}$, and $\varnothing\notin \mathcal{F}$. A filter $\mathcal{U}$ is  an \emph{ultrafilter} if it is maximal with respect to inclusion. An ultrafilter $\mathcal{U}$ is called \emph{non-principal} (or \emph{free}) if there is no $\alpha<\kappa$ such that $\mathcal{U}=\{A\s \kappa\mid \alpha\in A\}$. Ultrafilters were independently introduced by Riesz \cite{Riesz} and Ulam \cite{Ulam} and their existence (in fact, the existence of a non-principal one) is  granted by the Axiom of Choice. Ultrafilters have played a pivotal role in the contemporary development of mathematics. A non-exhaustive, yet illustrative, list of areas where  they have surfaced includes Set Theory \cite{MalliarisShelahII,GoldbergUA}, Model Theory \cite{ MalliarisShelah,GoldbringBook}, Topology \cite{CNBook, HandbookTopology}, and Ramsey Theory \cite{TodorcevicWalks}. All the  ultrafilters considered in this introduction  will be assumed non-principal.

\smallskip

Interesting ultrafilters posses additional combinatorial properties, such as \emph{completeness}.  Given a regular cardinal $\lambda\leq \kappa$, an ultrafilter $\mathcal{U}$ is \emph{$\lambda$-complete} if the intersection of any ${<}\lambda$-sized family of members of the ultrafilter  belongs to $\mathcal{U}$. Certainly, every ultrafilter is $\aleph_0$-complete. Surprisingly, however, the standard foundation of mathematics, ZFC,\footnote{That is, the Zermelo–Fraenkel axioms together with the Axiom of Choice.} cannot prove the existence of an uncountable cardinal $\kappa$ possessing the above-mentioned property. 

An uncountable cardinal $\kappa$ is called \emph{measurable} if it carries a $\kappa$-complete (non-principal) ultrafilter. This type of cardinals were introduced by Ulam (30's)  in connection to classical  \emph{Lebesgue's Measure Problem} (\cite[\S1]{Kan}). Measurable cardinals belong to the so-called \emph{large cardinal hierarchy}. Large cardinals are a  type of 'infinite'   satisfying certain  properties making them appear 'very large',\footnote{For instance, a cardinal $\kappa$ is  \emph{weakly inaccessible} if it is a regular limit cardinal  (i.e., $\lambda^+<\kappa$ for all $\lambda<\kappa$). This notion is due to Felix Hausdorff (1908) (cf \cite[\S1]{Kan}).} and whose existence cannot be established by ZFC. Measurable cardinals have had such an influence on  the development of modern pure mathematics   that Kanamori describes them in his   monograph \cite[\S2]{Kan} as \emph{the most important concept of
all large cardinal theory}.

\smallskip

There is a  long-standing research project aimed at classifying ultrafilters. One of the most successful tools in this endeavor has been the  \emph{Rudin-Keisler} order  $\leq_{\mathrm{RK}}$. Given ultrafilters $\mathcal{U}_0$  and $\mathcal{U}_1$ on $\kappa$ one writes  $\mathcal{U}_0\leq_{\RK} \mathcal{U}_1$  if and only if there is a function $\pi\colon \kappa\rightarrow\kappa$ such that $X\in \mathcal{U}_0 \Leftrightarrow \pi^{-1}(X)\in \mathcal{U}_1.$ Independently introduced by M. E. Rudin  and J. Keisler, the Rudin-Keisler order was first systematically studied by  A. Blass in his Ph.D. dissertation \cite{BlassPhD}, and later by many others. A central feature of the Rudin-Keisler order is that it provides a framework for understanding the intrinsic structure of an ultrafilter through its relative position within the order. For instance, \emph{selective} (a.k.a. \emph{Ramsey}) ultrafilters are  those that are $\leq_{\mathrm{RK}}$-minimal.  The literature delving on the structure of the  Rudin-Keisler order   is fairly extensive  \cite{RudinTAMS,  ComfortNegrepontis, Ketonen,  KunenTAMS, BlassTAMS, KanamoriUltrafilters,Gitik1988Ordering, KR, GoldbergUA}.

\smallskip

In this manuscript we will focus on the directedness of $\leq_{\mathrm{RK}}$ on $\mathfrak{U}_\kappa$, the set of $\kappa$-complete ultrafilters on $\kappa$: For a regular uncountable cardinal $\lambda$,   $\langle \mathfrak{U}_\kappa,\leq_{\RK}\rangle$ is  \emph{$\lambda$-directed} if  given any cardinal $\mu<\lambda$ and a collection of ultrafilters $\{\mathcal{U}_\alpha\mid \alpha<\mu\}\s \mathfrak{U}_\kappa$ there is yet another ultrafilter    $\mathcal{U}\in \mathfrak{U}_\kappa$ such that $\mathcal{U}_\alpha\leq_{\mathrm{RK}} \mathcal{U}.$  Note that every $\mathcal{U}\in \mathfrak{U}_\kappa$ has at most $2^\kappa$-many predecessors in $\leq_{\mathrm{RK}}$, hence $\leq_{\mathrm{RK}}$ is at most $(2^\kappa)^+$-directed. A classical result of Kat\v{e}tov, shows that $\leq_{\mathrm{RK}}$ restricted to uniform ultrafilters on the integers is $(2^{\aleph_0})^+$-directed. Other results along this line can be found in Comfort-Negrepontis book \cite{CNBook}. 
The directedness of  $\leq_{\mathrm{RK}}$ should be imagined as an indicative of    'canonicity' or 'regularity' within the structure of the order.

\smallskip

When it comes to measurable cardinals, classical theorems of Kunen \cite[Theorem~2.3]{Ketonen} and, independently,  Comfort–Negrepontis \cite[Theorem~4.3]{ComfortNegrepontis} deduce the $(2^\kappa)^+$-directedness of $\leq_{\mathrm{RK}}$ from instances of \emph{compactness}; namely, if $\kappa$ is \emph{$\kappa$-compact}   then $\langle \mathfrak{U}_\kappa, \leq_{\mathrm{RK}}\rangle$  is $(2^\kappa)^+$-directed.\footnote{Recall that an uncountable cardinal $\kappa$ is called $\kappa$-compact if every $\kappa$-complete filter on $\kappa$ extends to a $\kappa$-complete ultrafilter on $\kappa$.}  This is not surprising, for the directedness of $\leq_{\mathrm{RK}}$ can itself be viewed as a form of compactness: rather than extending an arbitrary $\kappa$-complete filter, one extends a filter on the product space $\prod_{\alpha < 2^\kappa} \kappa$ having the property that each projection is already a $\kappa$-complete ultrafilter (see \cite[Theorem~2.3]{Ketonen}).

\smallskip

Compactness is the mathematical phenomenon whereby the local beha\-vior of a 
structure (broadly understood) determines its global properties. 
A major line of research in set theory (e.g., \cite{Shelahcompactness, Sil, MagSheGroups, BagariaMagidor, PartIII, GoldbgerPov}) has long been devoted to classifying the 
various instances of compactness that arise throughout mathematics. Modulo some 
striking exceptions, 
compactness typically entails the existence of large cardinals. 
For this reason, the program has devoted extensive efforts at  
classifying these phenomena according to their \emph{large-cardinal 
strength}—that is, according to the position of the  relevant large cardinal  in  
the large-cardinal hierarchy.

In light of its status as both 'regularity property' and 'compactness-type property', the directedness of~$\leq_{\mathrm{RK}}$ emerges as a property of independent research interest. As such, a question that this project aims to elucidate is: 
 \begin{question*}
Let $\lambda$ be a regular uncountable cardinal. What is the consistency strength of the theory  $``\mathrm{ZFC}+\langle \mathfrak{U}_\kappa,\leq_{\RK}\rangle$ is  {$\lambda$-directed}''?
 \end{question*}
 Phrased in different terms: Is the directedness of  $\langle \mathfrak{U}_\kappa,\leq_{\RK}\rangle$ an intrinsic feature of $\kappa$-compact cardinals or are there other, consistencywise weaker, large cardinals entailing such property?
Until the present moment no large cardinal with consistency strength strictly below $\kappa$-compactness was known to imply the $\aleph_1$-directedness of $\langle \mathfrak{U}_\kappa,\leq_{\RK}\rangle$. Hence, a natural conjecture is that $\kappa$-compactness (or large cardinals with similar strength) are necessary.

 Our main result, however, refutes this conjecture:

\begin{MAIN}
    Let $\kappa$ be a measurable cardinal and $\lambda\leq  \kappa$ be a cardinal. Then, the following theories are equiconsistent modulo $\mathrm{ZFC}$:
    \begin{enumerate}
        \item[$(\aleph)$]  $``\langle\mathfrak{U}_\kappa,\leq_{\RK}\rangle$ is $\lambda^+$-directed" (resp. $\kappa^+$-directed).
        \item[$(\beth)$]  $``\text{$\kappa$ is measurable with }o(\kappa)=\lambda^+$" (resp. with $o(\kappa)=\kappa$).
    \end{enumerate}
\end{MAIN}
The result above shows that high degrees of directedness can be obtained from large cardinal assumptions far weaker than the existence of a strong cardinal, whose consistency strength is itself much lower than that of a $\kappa$-compact cardinal (see e.g., \cite{Hayutpartial, Gitikcompact}). Thereby, our result demonstrates that relatively weak large cardinals suffice to produce models in which a measurable cardinal exhibits strong compactness-like properties. To the best of the author's knowledge, this is the first example of a compactness-like property, at the level of measurable cardinals, with this property.

\smallskip

Our \textbf{Main Theorem} is obtained after bridging the directedness of the Rudin-Keisler order and a variation of the \emph{Gluing Property} ($\gp$) \cite{HP} called \emph{$\gp$ via ultrafilters}. More precisely, we prove the following theorem:
\begin{theorem*}\label{thm:main1}
    Let $\kappa$ be a measurable cardinal and $\lambda\leq 2^\kappa$ be a cardinal. Then, the following statements are equivalent:
    \begin{enumerate}
        \item $\kappa$ has the ${<}\lambda^+$-$\mathrm{GP}$ via ultrafilters.
        \item $\kappa$ has the $\lambda$-$\mathrm{GP}$ via ultrafilters.
        \item $\langle\mathfrak{U}_\kappa,\leq_{\RK}\rangle$ is $\lambda^+$-directed. 
    \end{enumerate}
    Moreover, if $\lambda\leq \kappa$  then (1) and (2) above are, respectively, equivalent to
      \begin{enumerate}
        \item[(1')] $\kappa$ has the ${<}\lambda^+$-$\mathrm{GP}$,
        \item[(2')] $\kappa$ has the $\lambda$-$\mathrm{GP}$.
     
    \end{enumerate}
\end{theorem*}
Once the above fact is established, we develop the set-theoretic  machinery needed to produce the consistency of the $\lambda$-$\mathrm{GP}$ via ultrafilters, starting from a model where $(\beth)$ holds. Our strategy has an intellectual debt with Gitik \cite{ChangingCofinalities, Gitik1988Ordering} and Ben-Neria–Unger \cite{BenUng}. The basic idea is to define a non-stationary support iteration of Prikry-type forcings that ``glue'' any $\lambda$-collection of $\kappa$-complete ultrafilters (non-necessarily normal)  via a $\kappa$-complete one. To materialize this plan two key technologies had to be developed: (1) The \emph{Gluing Poset} (\S\ref{sec: the gluing iteration}); (2) The \emph{Coding Lemma} (\S\ref{sec:codinglemma}). The Gluing Poset is an elaboration of a forcing due to Gitik \cite{ChangingCofinalities} and we argue  that it is necessary to obtain the desired result (\S\ref{sec: no gluing in Gitiks}). The Coding Lemma allows for a complete description  of \textbf{all} $\kappa$-complete ultrafilters in a generic extension via a non-stationary support iteration of Prikry-type forcings over the core model up to $o(\kappa)=\kappa^{++}$. To the best of our knowledge, this is the only result in the literature that provides a complete description of all $\kappa$-complete ultrafilters in a Prikry-type extension of this kind, and is therefore of independent interest\footnote{Similar charaterization of \emph{normal} measures appeared in various works in the literature. See, for example, \cite{FriedmanMagidor, GitikKaplan-nonstationary2022}.} 

Our proof follows the lead of \cite[\S6]{HP}, alas we have to account for a number of technical obstacles after allowing uncountable degrees of gluing (e.g., Lemmas~\ref{lem;maximal-stem}, \ref{lemma:imagesofmu} and \ref{lem:mu-remains-regular}).  Results akin to the Coding Lemma have appeared in works  of Kunen \cite{KunenMeasures}, Ben-Neria \cite{BenNeriaMeasures}, Gitik–Kaplan \cite{GitikKaplan-nonstationary2022}, Benhamou–Goldberg \cite{BenhamouGoldberg} and Ben-Neria–Kaplan \cite{BenNeriaKaplan}.

\smallskip

We conclude the paper  addressing a question of Gitik on the gluing property. After a series of seminar lectures  delivered by the first author in 2023  at the Hebrew University of Jerusalem, Gitik conjectured that the $\gp$ via ultrafilters  should hold in his model of \cite{ChangingCofinalities}. We prove the following result which shows that our Gluing Poset from \S\ref{sec: the gluing iteration} is indeed necessary to get $\gp$:
\begin{theorem*}
    The $\aleph_0$-$\gp$ fails in Gitik's model from  \cite{ChangingCofinalities}. In particular, in Gitik's model the Rudin-Keisler order is not $\aleph_1$-directed.
\end{theorem*}

The structure of the paper is as follows.  In \S\ref{sec:prelimminaries} we provide pertinent preliminaries and set notations for later use. In \S\ref{sec: gluing} we revisit the Gluing Property introduced in \cite{HP}, introduce a few new variants  and explore their interplay. In \S\ref{sec: SectionRK} we prove that the directedness of $\leq_{\mathrm{RK}}$ is equivalent to the $\gp$ via ultrafilters. Sections~\ref{sec: the gluing iteration} and \ref{sec:codinglemma} provide the technical core of the paper, which will be later utilized to establish our \textbf{Main Theorem} in \S\ref{sec: A model for gluing}. In \S\ref{sec: no gluing in Gitiks} we provide a negative answer to Gitik's conjecture.  \S\ref{sec: open problems} garners a few open problems that are deemed of interest. Our notations and terminologies follow (we hope) the standard vernacular of set theory.

 \section{Preliminaries}\label{sec:prelimminaries}

 \subsection{Ultrafilters} 
Let $U_0, U_1$ be ultrafilters. We  write $U_0\leq_{\mathrm{RK}}U_1$  if there is a function $\pi\colon \bigcup U_1\rightarrow\bigcup U_0$ such that for every set $X\s \bigcup U_0$, $$\text{$X\in U_0$ if and only if $\pi^{-1}(X)\in U_1.$}$$
Equivalently, there is a function $\pi$ as before such that $U_0=\pi_*(U_1)$ where $$\pi_*(U_1):=\{X\s \kappa\mid \pi^{-1}(X)\in U_1\}.$$
In the above circumstances we will say that  \emph{$U_0$ is Rudin-Keisler below $U_1$}.

We will abbreviate $``U_0\rk U_1$ and $U_1\rk U_0$" by $U_0\equiv_{\mathrm{RK}} U_1$. Similarly, $U_0<_{\mathrm{RK}} U_1$ will stand for $U_0\rk U_1$ and $\neg (U_1\rk U_0)$. Also, $U_0\equiv_{\mathrm{RK}} U_1$

The relation $\leq_{\RK}$ defines a preorder on the class of all ultrafilters –– that is, a reflexive and transitive binary relation. This latter, in turn, induces an order on the class of all ultrafilters, modulo $\equiv_{\RK}$. In practice we shall blur the  distinction between this two settings and will talk about the Rudin-Keisler order of ultrafilters. In this paper we shall be preoccupied with the Rudin-Keisler order $\leq_{\RK}$  in the context of  $\kappa$-complete non-principal ultrafilters over an uncountable cardinal $\kappa$. Moving forward we will denote this set by $\mathfrak{U}_\kappa.$  Since there are only $2^\kappa$-many functions $f\colon \kappa\rightarrow \kappa$ no member of $\mathfrak{U}_\kappa$ can have more than $2^\kappa$-many $\rk$-predecessors.
For a regular cardinal $\lambda\leq (2^\kappa)^+$ we say that  $\langle \mathfrak{U}_\kappa,\leq_{\RK}\rangle$ is \emph{$\lambda$-directed} if given $\mu<\lambda$ and ultrafilters $\{ U_\alpha\mid \alpha<\mu\}\s \mathfrak{U}_\kappa$ there is  another ultrafilter  $U\in \mathfrak{U}_\kappa$ such that $U_\alpha\leq_{\mathrm{RK}} U.$

Given a $\kappa$-complete ultrafilter $U$ on $\kappa$ we will denote by $j_U$ the ultrapower embedding $j_U\colon V\rightarrow M_U$ induced by $U$. Conversely, knowing the embedding $j_U$ and the ordinal $[\id]_U<j_U(\kappa)$, one can recover the ultrafilter $U$ via the equation $U=\{X\s \kappa\mid [\id]_U\in j_U(X)\}$.  For the general theory of ultrapowers  we refer the reader to \cite{SteelIteratedUlt} or to \cite[Ch. 1, \S5]{Kan}. 

\subsection{Forcing}\label{sectionPrikrytype}
The main forcing technology utilized in this manuscript are forcings of Prikry-type.
Here we adopt the perspective on Prikry-type posets  proposed by  Gitik \cite[\S6]{Gitik-handbook}. Specifically,
\begin{definition}\label{GitikPrikrydef}
    A triple $\langle \mathbb{P},\leq,\leq^*\rangle$ is called a \emph{Prikry-type forcing} if 
    \begin{enumerate}
        \item  $\leq^*\subseteq \leq$,
        \item $\leq^*$ has the \emph{Prikry property}; to wit, for every $p\in {P}$ and a statement $\varphi$ in the language of forcing of $\langle \mathbb{P},\leq\rangle$ there is  $q\leq^* p$ such that $$\text{$q\forces_{\mathbb{P}}\varphi$ or $q\forces_{\mathbb{P}}\neg\varphi.$}$$
    \end{enumerate}
\end{definition}
We will typically refer to $\mathbb{P}$ as being of Prikry-type when the ordering $\leq^*$ is clear from the context. Note that Gitik's definition is sufficiently general to include any poset by simply setting $\leq^* := \leq$. However, genuine Prikry-type forcings will satisfy $\leq^* \subsetneq \leq$, and typically the direct extension order $\leq^*$ will exhibit better closure properties than the regular order $\leq.$

\smallskip

In \cite{ChangingCofinalities}, Gitik introduced Easton-supported iterations of Prikry-type forcings. In this paper, we consider a modification of this construction due to Ben-Neria and Unger \cite{BenUng}, where the iteration has non-stationary support. This modification allows for finer control over the types of ultrafilters available in the generic extension (see e.g., \cite{HP, GitikKaplan-nonstationary2022, BenNeriaKaplan}).
\begin{definition}[Non-stationary support iteration, \cite{BenUng}]\label{Gitikiteration}
Let $\varrho$ be an ordinal. We define an iteration $\langle\mathbb{P}_\alpha, \dot{\mathbb{Q}}_\beta\mid \beta<\alpha\leq \varrho\rangle$ recursively as follows:  For  $\alpha\leq \varrho$ members of $\mathbb{P}_\alpha$ are sequences $p=\langle \dot{p}_\beta\mid \beta\in \supp(p)\rangle$ such that:
\begin{enumerate}[label=(\alph*)]
    \item $\supp(p)\s \alpha$ has \emph{non-stationary support}: namely, $\supp(p)\cap \beta$ is \emph{non-stationary} in $\beta$ for all inaccessible cardinals $\beta\leq \alpha$.
    \item for every $\beta\in \supp(p)$, $p\restriction\beta:=\langle \dot{p}_\beta\mid \beta\in \supp(p)\cap\beta\rangle\in \mathbb{P}_\beta$      $$p\restriction\beta\forces_{\mathbb{P}_\beta}``\dot{p}_\beta\in \dot{\mathbb{Q}}_\beta$$ and, moreover, $$\one\Vdash_{\mathbb{P}_\beta} |\dot{\mathbb{Q}}_\beta|\leq 2^\beta\,\wedge\,\text{$\langle \dot{\mathbb{Q}}_\beta,\leq_\beta, \leq_\beta^*\rangle$ is Prikry-type and $\leq^*$ is $\beta$-closed''}.\footnote{Recall that $\beta$-closure stands for thee following property: For every $\leq^*_\beta$-decreasing sequence $\langle r_\gamma\mid \gamma<\bar{\gamma}\rangle$ of members of $\mathbb{Q}_\beta$ (with $\bar{\gamma}<\beta$) there is $r\in\mathbb{Q}_\beta$ such that $r\leq^*_\beta r_\gamma.$ }$$
\end{enumerate}
Let $p=\langle \dot{p}_\beta\mid \beta\in \supp(p)\rangle$ and $q=\langle \dot{q}_\beta\mid \beta\in \supp(q)\rangle$ be elements of $\mathbb{P}_\alpha$.

We write $p\leq_\alpha q$  if and only if the following hold:
\begin{enumerate}
    \item $\supp(p)\supseteq \supp(q)$,
    \item $p\restriction\beta\forces_{\mathbb{P}_\beta}\dot{p}_\beta\leq_\beta \dot{q}_\beta$, for every $\beta\in \supp(q)$;
    \item there is $b\subseteq \supp(q)$, finite, such that for all $\beta\in \supp(q)\setminus b$
    $$p\restriction\beta\forces_{\mathbb{P}_\beta} \dot{p}_\beta\leq^*_\beta \dot{q}_\beta.$$
\end{enumerate}
The Prikry ordering $\leq^*_\alpha$ is defined by the case in which $b$ above is empty. 
\end{definition}

Two coments are in order. About (a): While members of $p\in \mathbb{P}_\varrho$ are names for the various posets involved note that $\supp(p)$ is assumed to be a set in the ground model. About (b): The assumptions that $|\dot{\mathbb{Q}}_\beta|\leq 2^\beta$ and $\langle \dot{\mathbb{Q}}_\beta,\leq^*_\beta\rangle$ is forced
to be $\beta$-closed are technical requirements permitting \emph{diagonalizations}. This is key to proving the Prikry property and its variants.

A fundamental theorem regarding the aforementioned iterations is the Prikry property. The result for Easton-supported iterations was proved by Gitik \cite[\S1]{ChangingCofinalities}; the corresponding one for non-stationary-supported iterations is due to Ben-Neria and Unger \cite[\S2]{BenUng}. Variants of the Prikry property relevant to this paper are presented in the following lemmas. (In both lemmas the iteration $\mathbb{P}_\varrho$ is assumed to have non-stationary support.) 

\begin{lemma}[\cite{BenUng}]\label{lemma:bnu}
Suppose that $\varrho$ is a limit ordinal and $e\colon \varrho\rightarrow V$ is a function such that for each $\alpha<\varrho$, $e(\alpha)$ is a $\mathbb{P}_{\alpha+1}$-name for a $\leq^*$-dense subset of $\mathbb{P}_\varrho\setminus (\alpha+1)$. Then, for each $p\in \mathbb{P}$ and $\nu<\varrho$ there is $p^*\leq^* p$ such that $p^*\restriction\nu=p\restriction\nu$ and $p^*\restriction \alpha+1\forces_{\mathbb{P}_{\alpha+1}}p^*\restriction\alpha+1\in e(\alpha).$ \qed
\end{lemma}
\begin{lemma}[\cite{BenUng}]\label{lemma: reducingdensesets}
   Let $p\in \mathbb{P}_\varrho$, $\nu<\varrho$ and $D\s \mathbb{P}_\varrho$  a  dense open set. Then, there is $p^*\leq^*p$ and $\alpha<\varrho$ such that $p^*\restriction\nu+1=p\restriction\nu+1$ and $$\{r\leq p^*\restriction\alpha\mid r^\smallfrown p^*\setminus \alpha\in D\}$$ is $\leq$-dense below $p^*\restriction \alpha.$\qed
\end{lemma}
An immediate corollary is the ``reduction of names" lemma:
\begin{cor}\label{cor: reducingnames}
    Let $\sigma$ be a $\mathbb{P}_\varrho$-name and $p\in \mathbb{P}_\varrho$ such that $p\forces_{\mathbb{P}_\varrho}\sigma\in \check{V}$. Then, there is $p^*\leq^* p$, $\alpha<\varrho$ and a $\mathbb{P}_\alpha$-name $\tau$ such that $p^*\forces_{\mathbb{P}_\varrho}\sigma=\tau.$

    Moreover, for each $\nu<\varrho$ one can find $\alpha\geq \nu$ and $p^*\leq^* p$ as above so that $p^*\restriction\nu+1= p\restriction\nu+1.$
\end{cor}
\begin{proof}
    Let $D_\sigma:=\{q\leq p\mid \exists x\, (q\forces_{\mathbb{P}_\varrho}\sigma=\check{x})\}.$ Since this is dense open below $p$ there is $p^*\leq^*p$ and $\alpha<\varrho$ with $\{r\leq p^*\restriction\alpha\mid r^\smallfrown p^*\setminus\alpha\in D_\sigma\}$ being dense below $p^*\restriction\alpha.$ Let $A$ be a maximal antichain (in $\mathbb{P}_\alpha$) below $p^*\restriction\alpha$ included in the previous set. For each $r\in A$ there is $x_r$ such that $r^\smallfrown p^*\setminus\alpha\forces_{\mathbb{P}_\varrho}\sigma=\check{x}_r$. Let $\tau$ be the $\mathbb{P}_\alpha$-name resulting from mixing the names $\check{x}_r$ according to the maximal antichain $A$. It follows that, for each $r\in A$, $r^\smallfrown p^*\restriction\alpha\forces_{\mathbb{P}_\varrho}\check{x}_r=\tau$. From this it is easy to infer that $p^*\forces_{\mathbb{P}_\varrho}\sigma=\tau.$
    \end{proof}
 \begin{cor}   \label{cor: meeting-dense-open-bounded}
       Suppose that  $p\in \mathbb{P}_\varrho$ is a condition $\nu\leq \bar\varrho<\varrho$ are ordinals and $\sigma$ is a $\mathbb{P}_{\bar\varrho}$-name such that $p\restriction\bar\varrho\forces_{\mathbb{P}_{\bar\varrho}}\sigma\in \check{V}.$ Then, there is $p^*\leq^* p$ and a $\mathbb{P}_\alpha$-name $\tau$, with $\alpha<\bar\varrho$, such that $p^*\restriction\nu=p\restriction\nu$ and  $p^*\restriction\bar\varrho\forces_{\mathbb{P}_{\bar\varrho}}\sigma=\tau.$ \qed
 \end{cor}     
        The proof  is identical to the argument given in Corollary~\ref{cor: reducingnames}, with the proviso that
$D_\sigma:=\{q\leq p\mid \exists x\, (q\restriction\bar\varrho\forces_{\mathbb{P}_{\bar\varrho}}\sigma=\check{x})\}.$    This variant of the ``name reduction lemma" will be useful later in this paper  (see Lemma~\ref{lemma:densityconditions}).

\smallskip

As in \cite{HP} in this paper we will also need to ``predict" the ultrafilters that will arise in the target generic extension. This will be achieved using \emph{Woodin's fast function forcing}; or, more precisely, using its non-stationary support version consider in \cite[\S6]{HP}:
\begin{definition}[Fast Function Forcing]\label{def: fast function forcing}
    For an inaccessible  $\kappa$  we denote by $\mathbb{S}_\kappa$ the poset consisting on partial functions $s \colon \kappa \to H(\kappa)$ such that  
\begin{enumerate}\label{fastfunctionforcing}
    \item $\dom(s)\s \mathrm{Inacc}$,
    \item $\dom (s) \cap \beta\in \mathrm{NS}_\beta$ for all $\beta\in\mathrm{Inacc}\cap (\kappa+1)$,
    \item  $s``\alpha\s \alpha$ for all $\alpha\in\dom(s)$,
    \item and $s(\alpha)\in H(\alpha^+)$ for all $\alpha\in\dom s$.
\end{enumerate}
The order of $\mathbb{S}_\kappa$ is defined naturally as $s\leq t$ iff $s\supseteq t.$
\end{definition}
\begin{prop}[Some properties of $\mathbb{S}_\kappa$, {\cite[Proposition~6.2]{HP}}]\label{PropertiesofS}
\hfill
\begin{enumerate}
    \item $\mathbb{S}_\kappa$ adds a generic fast function $\ell\colon \kappa\rightarrow \kappa$.\footnote{Specifically a function $\ell\colon\kappa\rightarrow\kappa$ for which the following holds: For all $x\in H(\kappa^+)^V$ there is an  embedding $j\colon V[S]\rightarrow M$ with $\crit(j)=\kappa$, ${}^\kappa M\s M$ and $j(\ell)(\kappa)=x$.}
    \item $\mathbb{S}_\kappa$ is $\kappa^{++}$-cc and preserves $\kappa^+$.
     \item Forcing with $\mathbb{S}_\kappa$ preserves the $\mathrm{GCH}$.
    \item For every  condition $s=\{\langle \alpha, s(\alpha)\rangle\}\in \mathbb{S}_\kappa$ 
    $$\mathbb{S}_\kappa/s\simeq \mathbb{S}_\alpha\times \mathbb{S}_\kappa\setminus \alpha,$$ where $\mathbb{S}_\alpha$ is $\alpha^{++}$-cc and $\mathbb{S}\setminus \alpha$ is $\theta_\alpha$-closed ($\theta_\alpha:=\min (\mathrm{Inacc}\setminus \alpha^+)$).
    \item  If $\kappa$ is measurable then forcing with $\mathbb{S}_\kappa$ preserves this property. $\qed$
\end{enumerate}
\end{prop}
The basic idea is to use the fast function  $\ell\colon \kappa\rightarrow \kappa$ to predict the ``codes"  for the measures to be glued (see Definition~\ref{def: codes}). The fast function $\ell$ will also induce a natural ``order" $o^\ell(\cdot)$ (Definition~\ref{ellorder}) which will be used to define a variant of Gitik's forcing from \cite{ChangingCofinalities} yielding the gluing property. 

    \subsection{Inner model theory} As part of our analysis we will also need a few basic concepts from Inner Model theory. The first one is a natural order between ultrafilters introduced by Mitchell: 
\begin{definition}[Mitchell order]\label{MitchellIncreasing}
    For normal ultrafilters ${U}$ and ${W}$ over the same cardinal  one writes $ {U}\lhd {W}$ iff $ {U}\in \Ult(V, {W})$. 
\end{definition}
It is a theorem of Mitchell \cite[\S2]{MitChap} that $\lhd$ is well-founded, hence it makes perfect sense to speak about the rank of a normal measure $ {U}$ in $\lhd$.
\begin{definition}
    The \emph{Mitchell order of a normal measure} $ {U}$ (in symbols, $o( {U})$) is its rank in $\lhd$; namely, $o( {U}):=\sup\{o( {W})+1\mid  {W}\lhd  {U}\}.$

    Similarly, the \emph{Mitchell order of a cardinal $\kappa$} is $$o(\kappa):=\sup\{o( {U})+1\mid \text{$ {U}$ is a normal measure over $\kappa$}\}.$$
\end{definition}
 A cardinal $\kappa$ is not measurable if and only if $o(\kappa)=0$. If $o(\kappa)=1$ then $\kappa$ is measurable but any normal measure $ {U}$ over it concentrates on $\{\alpha<\kappa\mid \text{$\alpha$ is not measurable}\}.$ Solovay observed (see \cite[\S2]{MitChap}) that $o(\kappa)\leq (2^\kappa)^+$ and thus, under GCH, $o(\kappa)\leq \kappa^{++}$.

\smallskip

A related notion, also due to Mitchel, are \emph{coherent sequence of measures}.
\begin{definition}
A \emph{coherent sequence of measures} is a function ${\mathcal{U}}$ with:
\begin{enumerate}
    \item $\dom({\mathcal{U}})=\{(\alpha,\beta)\mid \alpha<\mathrm{len}({\mathcal{U}})\;\text{and}\;\beta<o^{{\mathcal{U}}}(\alpha)\}$, where $\mathrm{len}({\mathcal{U}})$ is a cardinal and $o^{{\mathcal{U}}}$ is a map between cardinals $\alpha<\mathrm{len}({\mathcal{U}})$ to ordinals;
    \item if $(\alpha,\beta)\in \dom({\mathcal{U}})$ then $\mathcal{U}(\kappa,\beta)$ is a normal measure on $\kappa$;
    \item  for all $(\alpha,\beta)\in\dom({\mathcal{U}})$,  $j_{\mathcal{U}(\alpha,\beta)}({\mathcal{U}})\restriction\alpha+1={\mathcal{U}}\restriction(\alpha,\beta)$ where
    $${\mathcal{U}}\restriction(\alpha,\beta):=\{\mathcal{U}(\gamma,\bar{\beta})\mid (\gamma< \alpha\,\wedge\,\bar{\beta}<o^{{\mathcal{U}}}(\gamma))\; \vee\; (\gamma=\alpha\,\wedge\,\bar{\beta}<\beta)\},$$
    and 
    $j_{\mathcal{U}(\alpha,\beta)}({\mathcal{U}})\restriction\alpha+1:=j_{\mathcal{U}(\alpha,\beta)}({\mathcal{U}})\restriction(\alpha+1,0)$.
\end{enumerate}
\end{definition}

\begin{definition}\label{Iteration}
A system of elementary embeddings $\langle \iota_{\alpha, \beta} \mid \alpha \leq \beta \leq \delta\rangle$ between transitive models is called a \emph{linear iteration} if the following hold:
\begin{enumerate}
    \item $\iota_{\alpha,\beta}\colon \mathcal{M}_{\alpha} \to \mathcal{M}_{\beta}$ has $\crit(\iota_{\alpha,\beta})=\mu_\alpha$,
    \item for all $\alpha \leq \beta \leq \gamma \leq \delta$, $\iota_{\beta,\gamma}\circ \iota_{\alpha,\beta} = \iota_{\alpha,\gamma}$,
    \item for all $\alpha \leq \delta$, $\iota_{\alpha,\alpha} = \id$, and
    \item for every limit ordinal $\gamma \leq \delta$, $$\text{$\langle \mathcal{M}_\gamma, \langle \iota_{\alpha,\gamma} \mid \alpha < \gamma\rangle\rangle$ is the direct limit of $\langle \iota_{\alpha,\beta} \mid \alpha \leq \beta < \gamma\rangle$.}$$
\end{enumerate}

A linear iteration is a \emph{normal iteration, using internal measures} if:
\begin{enumerate}\addtocounter{enumi}{4}
    \item  
    $\iota_{\alpha, \alpha+1}$ is an ultrapower embedding by a measure $\mathcal{U}_\alpha \in \mathcal{M}_\alpha$ on $\mu_\alpha$,   
    \item $\langle \mu_\alpha \mid \alpha < \delta\rangle$ is a strictly increasing sequence.\end{enumerate}
\end{definition}
\begin{conv}
For each $\alpha\leq \delta$, $\iota_\alpha$ will serve as a shorthand for $\iota_{0,\alpha}$. Often times we will be a bit careless and say that $\iota_\delta$ is an iteration rather than referring to the sequence $\langle \iota_{\alpha, \beta} \mid \alpha \leq \beta \leq \delta\rangle$.
\end{conv}

The following technical lemmas were established in \cite{HP}:
\begin{lemma}\label{lem;representing-elements-in-direct-limit}
Let $\langle \iota_{\alpha, \beta} \mid \alpha \leq \beta \leq \delta\rangle$ be a normal iteration using normal measures,  with critical points $\langle \mu_\alpha\mid \alpha<\delta\rangle$. Then, for every set $x \in \mathcal{M}_\delta$ there is a finite normal iteration $\tilde\iota \colon \mathcal{M}_0 \to {\mathcal{N}}$ together with a (factor) embedding $k \colon {\mathcal{N}} \to \mathcal{M}_\delta$ such that $k \circ \tilde{\iota} = \iota_\delta$ and $x \in \range(k)$. \qed
\end{lemma}
Working in $\mathcal{K}$ let $\kappa$ be measurable and $S\s\mathbb{S}_\kappa$  a generic filter over $\mathcal{K}$. 
\begin{lemma}\label{lem;lifting-S}
Suppose $j\colon \mathcal{K}[S] \to M$ is an arbitrary elementary embedding such that $\iota = j \restriction \mathcal{K}$ is a normal iteration of measures in $\mathcal{K}$, say $\langle \iota_{\alpha,\beta} \mid \alpha \leq \beta \leq \delta\rangle$.    
Then, there is an iteration $\langle j_{\alpha,\beta} \mid \alpha\leq\beta\leq\gamma\rangle$ such that each embedding $j_{\alpha,\beta}$ lifts the corresponding embedding $\iota_{\alpha,\beta}$. \qed
\end{lemma}

Quite often during this paper we will be invoking the so-called \emph{Mitchell's covering theorem up to $o(\kappa)=\kappa^{++}$} \cite{MitIter}:
\begin{theorem}\label{CoreModelTheorem}
Assume there is no inner model for $``\exists\alpha\,(o(\alpha)=\alpha^{++})$''. Then there is a unique 
inner model $\mathcal{K}$ with the following properties:
\begin{enumerate}
    \item\label{Maximality}  If $\mathcal{U}$ is a $\mathcal{K}$-normal and $\kappa$-complete $\mathcal{K}$-ultrafilter over $\kappa$ then $\mathcal{U}\in \mathcal{K}$.\footnote{$\mathcal{K}$-normal means that $\mathcal{U}$ is normal relative to every function $f\colon \kappa\rightarrow \kappa$ in $\mathcal{K}$.}
    \item\label{coremodelIterations} Any non-trivial elementary embedding $j\colon \mathcal{K}\rightarrow M$ into a transitive class $M$ is a normal linear iteration using normal measures in $\mathcal{K}$.
    \item\label{GenericAbsoluteness}  Let $\mathbb{P}$ be a set-sized forcing and  $G\s \mathbb{P}$ be generic. Then $\mathcal{K}^{V[G]}=\mathcal{K}$.
    \item\label{weakcovering} If $\delta$ is a singular strong limit cardinal then $\delta^+=(\delta^+)^{\mathcal{K}}$.
\end{enumerate}
\end{theorem}
The above model $\mathcal{K}$ is the so-called \emph{core model up $o(\kappa)=\kappa^{++}$} but in this paper we shall refer to it simply as the \emph{core model}. Aside from the above-displayed properties, $\mathcal{K}$ has some other gentle combinatorial features - for instance, $\mathcal{K}\models \GCH$.\label{GCHinK} For more see Mitchell's handbook chapter \cite{Mithand2}. We will not need the precise definition of the core model in this paper.

\section{The Gluing Property}\label{sec: gluing}
A key property studied in this paper is the \emph{Gluing Property} \cite{HP}:
\begin{definition}[The Gluing Property]\label{def:gluing-property}
Let $\kappa$ be a measurable cardinal and $\alpha\geq \omega$ be an ordinal. We say that $\kappa$ has the \emph{$\alpha$-Gluing Property} ($\alpha$-$\gp$) if for every $\alpha$-sequence $\langle U_\beta \mid \beta < \alpha\rangle$ of non-principal $\kappa$-complete ultrafilters over $\kappa$  
     there is a (definable) elementary embedding $j \colon V \to M$   with $\crit(j) = \kappa$ and $M^\kappa \subseteq M$,
   and an increasing sequence of ordinals  $\langle \eta_\beta \mid \beta < \alpha\rangle$ such that
     $$U_\beta = \{X \subseteq \kappa \mid \eta_\beta \in j(X)\}.\footnote{
    $\eta_\beta\geq \kappa$ as a consequence of the non-principality of $U_\beta$.}$$

 We will say that   
    \emph{$\kappa$ has the ${<}\alpha$-$\gp$} provided $\kappa$ has the $\beta$-$\gp$ for all $\beta<\alpha.$
\end{definition}

\begin{notation}
    We will often refer to sequences $\langle \eta_\beta\mid \beta<\alpha\rangle$ as \emph{seeds} and say that $\langle \eta_\beta\mid \beta<\alpha\rangle$  and $j$ \emph{glue $\langle U_\beta\mid \beta<\alpha\rangle$} if the above property holds.
\end{notation}

 In \cite{HP} it was proved that strongly compact cardinals have the $\lambda$-$\gp$ for all cardinals $\lambda$.\footnote{Moreover, the $2^{\kappa}$-$\gp$ via ultrafilters holds. This is a consequence of a classical theorem of Kunen \cite[Theorem~2.3]{Ketonen} for he showed that if $\kappa$ is strongly compact then $\langle\mathfrak{U}_\kappa,\leq_{\mathrm{RK}}\rangle$ is $(2^\kappa)^+$-directed. This combined with our Theorem~\ref{thm: EquivalencewithRK} yields the  claim.} Refining this argument we were able to show  that $\lambda$-compact cardinals have the $(2^\lambda)$-$\gp$ but not necessarily the $(2^\lambda)^+$-$\gp$. In the same paper it was argued that strong cardinals in the core model $\mathcal{K}$ may not have the $\aleph_0$-$\gp$, yet by performing appropriate forcing one can turn a strong cardinal into a measurable with the $\aleph_0$-$\gp$. This was followed with a proof that $\mathrm{ZFC}+$``There exists a measurable cardinal  with the $\aleph_0$-$\gp$" is equiconsistent with  the theory $\mathrm{ZFC}+``\exists\kappa\,(\text{$\kappa$ is measurable and } o(\kappa)=\omega_1)$".
    
    In the present manuscript we refine our analysis to show that the theory $\mathrm{ZFC}+$``There exists a measurable $\kappa$ with the ${<}\lambda^+$-$\gp$" (for cardinals $\lambda<\kappa$) is equiconsistent with $\mathrm{ZFC}+``\exists\kappa\,(\text{$\kappa$ is measurable and } o(\kappa)=\lambda^+)$". Furthermore, we show that  $\mathrm{ZFC}+$``There exists a measurable $\kappa$ with  ${<}\kappa^+$-$\gp$" is equiconsistent with $\mathrm{ZFC}+``\exists\kappa\,(\text{$\kappa$ is measurable and } o(\kappa)=\kappa)$". These will be the matter of discussion of the forthcoming Section~\ref{sec: A model for gluing}.

    \smallskip

As part of our general analysis on the Gluing Property, we will examine the relationship between this notion and some of its natural variants. 
\begin{definition}
    Let  $\kappa$ be a  measurable cardinal and $\omega\leq \alpha$ an ordinal.
    \begin{enumerate}
         \item $\kappa$ has the \emph{super $\alpha$-$\gp$} if  it has the $\alpha$-$\gp$ and in addition one can take the sequence of seeds $\langle \eta_\beta\mid \beta<\alpha\rangle$  of Definition~\ref{def:gluing-property} inside $M.$
        \item  $\kappa$ has the \emph{$\alpha$-$\gp$ via ultrafilters} if it has the $\alpha$-$\gp$ and this is witnessed by ultrapower embeddings via  $\kappa$-complete ultrafilters  $W$ on $\kappa^{<\kappa}$. 
        \item $\kappa$ has the \emph{weak $\alpha$-$\gp$} if the conclusion of Definition~\ref{def:gluing-property} holds but $\langle \eta_\beta\mid \beta<\alpha\rangle$ is not necessarily increasing.
    \end{enumerate}
\end{definition}
\begin{remark}
If $W$ is a $\kappa$-complete ultrafilter on $\kappa^{<\kappa}$, $\beta<\alpha$ and $U_\beta=\{X\s\kappa\mid \eta_\beta\in j_W(X)\}$,  then we can assume that $\sup_{\beta<\alpha} \eta_\beta\leq \sup([\id]_W)$: Otherwise, let $W^*:=\{Y\s \kappa^{<\kappa}\mid [\id]_W{}^\smallfrown[\id]_{j_W(W)}\in j(Y)\}$ where $j$ is the composition $j_{j_W(W)}\circ j_W$.  Clearly, $W^*$ is a $\kappa$-complete ultrafilter with $\sup_{\beta<\alpha}\eta_\beta\leq \sup([\id]_{W^*})$, and it glues the ultrafilters $U_\beta$ for $\beta<\alpha$. 
\end{remark}

Next, we prove a result clarifying the implications between these notions.

\begin{theorem}\label{theorem: variantsofgp}
Let $\kappa$ be a measurable cardinal and $\omega \leq \alpha<\kappa^{++}$ an ordinal. Then, each of the following properties implies the next one in the list:
    \begin{enumerate}
        \item $\kappa$ has the super $\alpha$-$\gp$.
        \item $\kappa$ has the $\alpha$-$\gp$ via ultrafilters.
        \item $\kappa$ has the $\alpha$-$\gp$.
        \item $\kappa$ has the weak $\alpha$-$\gp$.
    \end{enumerate}
    Moreover, if $\alpha<\kappa^+$ then (1)--(3) are equivalent.
\end{theorem}
\begin{proof}

    (2) $\Rightarrow$ (3) and (3) $\Rightarrow$ (4) are trivial. Also,   if $\alpha<\kappa^+$ then the closure  under $\kappa$-sequences of the models $M$ that are the target of a gluing embedding $j\colon V\rightarrow M$ yield (3) $\Rightarrow$ (1).  
    Therefore, it suffices to prove  (1) $\Rightarrow$ (2).

    \smallskip

    (1) $\Rightarrow$ (2). Let $\langle U_\beta\mid \beta<\alpha\rangle$ be $\kappa$-complete ultrafilters. Since $\kappa$ has the super $\alpha$-$\gp$ there is an elementary embedding $j\colon V\rightarrow M$ and an increasing sequence $\langle \eta_\beta\mid \beta<\alpha\rangle$ below $j(\kappa)$ such that $\langle \eta_\beta\mid \beta<\alpha\rangle\in M$ and
    $$U_\beta=\{X\s \kappa\mid \eta_\beta\in j(X)\}.$$
    Since $\eta_0\geq\kappa$ by plugging (if necessary) an extra element in the sequence $\langle \eta_\beta\mid \beta<\alpha\rangle$ we may assume without loss of generality that $\eta_0=\kappa$.
    
    Suppose first that $\alpha\leq \kappa^+.$ Define an ultrafilter  $W$ on $\kappa^{<\kappa}$ by stipulating that $X\in W$ if and only if $X$ is a subset of the increasing sequences on $\kappa^{<\kappa}$ such that $\langle \eta_\beta\mid \beta<\alpha\rangle\in j(X)$. Clearly, $W$ is $\kappa$-complete and the map $k\colon M_W\rightarrow M$ given by $j_W(f)([\id]_W)\mapsto j(f)([\id]_W)$ is elementary. {Moreover, $k$ has critical point greater than $\kappa$: To see that $\kappa\in \range(k)$ note that $\kappa=j(f_\kappa)([\id]_W)$ where $f_\kappa\colon \kappa^{<\kappa}\rightarrow \kappa$ is the map given by  $f_\kappa(s):=\min(s)$. 

 Since $W$ is $\kappa$-complete, $M_W$ is closed under $\kappa$-sequences of its elements.}

   \smallskip

   Note that the $\eta_\beta$'s are also in the range of $k$; that is, $\eta_\beta=j(g_\beta)([\id]_W)$ for some function $g_\beta\colon \kappa^{<\kappa}\rightarrow\kappa$: Indeed, $\eta_\beta$ is represented via the function $$\text{$g_\beta\colon s\mapsto$ ``The $(c_\beta\circ f_\kappa(s))^{\mathrm{th}}$-element of $s"$,}$$ where $c_\beta$ is the canonical function representing $\beta$ when $\beta\in (\kappa,\kappa^+)$, the identity when $\beta=\kappa$ and the constant function with value $\beta$ when $\beta< \kappa$.

   Thus, for each $\beta<\alpha$, $\eta_\beta=k(\bar{\eta}_\beta)$. By elementarity, $\bar{\eta}_\beta<\bar{\eta}_\gamma$ for all $\beta<\gamma<\alpha$. Also, it is easy to check that $U_\beta=\{X\s \kappa\mid \bar{\eta}_\beta\in j_W(X)\}.$

    \smallskip

    Suppose now that $\alpha\in (\kappa^+,\kappa^{++})$ and assume that $(1)\Rightarrow (2)$ has been proven for all ordinals $\beta<\alpha$. As usual, we distinguish among two cases:

    \smallskip

    \underline{Successor case:} Suppose that $\alpha=\bar{\alpha}+1$. By the induction hypothesis, $\kappa$ has the $\bar{\alpha}$-$\gp$ via ultrafilters. Thus, there is a $\kappa$-complete ultrafilter $W$ on $\kappa^{<\kappa}$ with ultrapower $j\colon V\rightarrow M_W$, $M_W$ is $\kappa$-closed and there is an increasing sequence $\langle \eta_\beta\mid \beta<\bar\alpha\rangle$ of ordinals such that $\sup_{\beta<\bar\alpha}\eta_\beta\leq \sup([\id]_W)$ and $U_\beta=\{X\s \kappa\mid \eta_\beta\in j_W(X)\}.$ Now, we consider the product ultrafilter $Z=W\otimes U_{\bar\alpha}$, noticing that $\{(s,\xi)\in \kappa^{<\kappa}\times \kappa\mid \sup(s)<\xi\}\in Z$.\footnote{This is  where the requirement $\sup_{\beta<\bar\alpha}\eta_\beta\leq \sup([\id]_W)$ is critical. Without it, we will not be able to ensure the existence of a seed for $U_{\bar\alpha}$ that is  above the previous $\eta_\beta.$}

    Let $s^\star, \xi^\star$ be such $[\id]_Z=(s^\star, \xi^\star).$ Thus, $k_W(\sup([\id])_W)=\sup(s^\star)<\xi^\star$.

       Consider the natural commutative diagrams 
 \[
\begin{minipage}{0.30\textwidth}
\centering
\begin{tikzcd}
    V \arrow[r, "{j_{Z}}"] \arrow[d, "j_{W}"'] 
     & M_{Z}  \\
 M_{W} \arrow[ru, "{k_{W}}"] &  
\end{tikzcd}
\end{minipage}
\qquad
\begin{minipage}{0.30\textwidth}
\centering
\begin{tikzcd}
    V \arrow[r, "{j_{Z}}"] \arrow[d, "j_{U_{\bar\alpha}}"'] 
     & M_{Z}  \\
 M_{U_{\bar\alpha}} \arrow[ru, "{k_{{U_{\bar\alpha}}}}"] &  .
\end{tikzcd}
\end{minipage}
\]

For $\beta<\bar\alpha$, denote $\sigma_\beta=k_W(\eta_\beta)$. Also, let  $\sigma_{\bar\alpha}=\xi^\star.$ 
\begin{claim}
    $j_{Z}$ and $\langle \sigma_\beta\mid \beta\leq \bar\alpha\rangle$ glue the sequence $\langle U_\beta\mid \beta\leq \bar\alpha\rangle$.
\end{claim}
\begin{proof}[Proof of claim]
    First, the sequence $\langle \sigma_\beta\mid \alpha\leq \bar\alpha\rangle$ {is increasing}: Indeed, this is clear for $\beta<\bar\alpha$ in that $\langle \eta_\beta\mid \beta<\bar\alpha\rangle$ was already increasing. Also, $$k_W(\eta_\beta)\leq k_{W}(\sup([\id]_W))=\sup(s^\star)<\xi^\star=\sigma_{\bar\alpha}.$$

  Second, for each $X\s \kappa$,
$$X\in U_\beta\Leftrightarrow \eta_\beta\in j_{W}(X)\Leftrightarrow k_W(\eta_\beta)\in j_{Z}(X)\Leftrightarrow \sigma_\beta\in j_{Z}(X),$$
$$X\in U_{\bar\alpha}\Leftrightarrow [\id]_{U_{\bar\alpha}}\in j_{U_{\bar\alpha}}(X)\Leftrightarrow k_{U_{\bar\alpha}}([\id]_{U_{\bar\alpha}})\in j_{Z}(X)\Leftrightarrow \sigma_{\bar\alpha}\in j_{Z}(X).$$
This yields the claim.
\end{proof}
Since $Z$ is isomorphic to a $\kappa$-complete ultrafilter on $\kappa^{<\kappa}$ we are done.

    \smallskip

    \underline{Limit case:} 
Suppose now that $\alpha$ is a limit ordinal. Let $\langle \alpha_\xi\mid \xi<\cf(\alpha)\rangle$ be a cofinal sequence in $\alpha$. For each $\xi<\cf(\alpha)$($\leq \kappa^+$), our inductive assumption allows us to glue the sequence of ultrafilters $\langle U_{\beta}\mid \beta<\alpha_\xi\rangle$ via a $\kappa$-complete ultrafilter $W_\xi$ on $\kappa^{<\kappa}$. This yields, for each $\xi<\cf(\alpha)$, an increasing sequence of seeds $\langle \eta^\xi_\beta\mid \beta<\alpha_\xi\rangle$ that together with $j_{W_\xi}$ glue  $\langle U_{\beta}\mid \beta<\alpha_\xi\rangle$.

\smallskip

 Let $\varphi\colon \kappa^{<\kappa}\leftrightarrow \kappa$ be a bijection and let $\bar{W}_\xi:=\varphi_*W_\xi.$
This is now a $\kappa$-complete ultrafilter over $\kappa$. Thus, since $\cf(\alpha)\leq \kappa^+$ our induction hypothesis allows us to find an ultrafilter $W^*$ on $\kappa^{<\kappa}$ that glues all the $\bar{W}_\xi$'s.

\smallskip

Finally, let us consider the commutative diagram
        \[
\begin{tikzcd}
   V \arrow[r, "{j_{W^*}}"] \arrow[d, "j_{U_\beta}"'] 
     & M_{W^*} \\
   M_{U_\beta} \arrow[r, "k_{\beta,\xi}"'] &  M_{W_\xi}  \arrow[u, "k_\xi"'].
\end{tikzcd}
\]   
$k_{\beta,\xi}$ is the natural embedding induced from the fact that $j_{W_\xi}$ glues the $U_\beta$'s\footnote{I.e., the map $k_{\beta,\xi}$ defined as $k_{\beta,\xi}(f)([\id]_{U_\beta}):=j_{W_\xi}(f)(\eta^\xi_\beta).$}; $k_\xi$ is the  embedding induced from $W_\xi\equiv_{\RK}\bar{W}_\xi$ and $j_{W^*}$ glues the $\bar{W}_\xi$'s.

Denote $\langle \sigma^\xi_\beta\mid \beta<\alpha_\xi\rangle:=\langle k_{\xi}(\eta^\xi_\beta)\mid \beta<\alpha_\xi\rangle$. It turns out that
$$U_\beta=\{X\s\kappa\mid \sigma^\xi_\beta\in j_{W^*}(X)\}$$
for all $\beta<\alpha_\xi$ and $\xi<\cf(\alpha).$

\smallskip

For each $\beta<\alpha$, let $\eta_\beta:=\min_{\xi<\cf(\alpha)}\{\sigma_\beta^\xi\mid \beta<\alpha_\xi\}$.  

Given ordinals $\bar\beta<\beta<\alpha$ we have
$$\eta_{\beta}=\sigma^{\xi_\beta}_\beta>\sigma^{\xi_{\beta}}_{\bar\beta}\geq \eta_{\bar\beta},$$ 
$$\eta_\beta=\sigma^{\xi_\beta}_\beta=k_\xi(\eta^{\xi_\beta}_\beta)\leq k_\xi(\sup[\id]_{W_\xi})= \sup([\id]_{W^*}).$$
So, $j_{W^*}$ and $\langle \eta_\beta\mid \beta<\alpha\rangle$ glue the ultrafilters $\langle U_\beta\mid \beta<\alpha\rangle.$ 
\end{proof}

The implications $(2)\Rightarrow (3) \Rightarrow (4)$ do hold for arbitrary ordinals $\alpha\geq \omega$. Later we will demonstrate that if $\alpha < \kappa^+$ then (1)--(4) are in fact equivalent.  The above theorem is sharp in the sense that the above implications do not hold if $\alpha\in\{\kappa^+,\kappa^{++}\}$. The first evidence of this is the next proposition:
\begin{prop}
    The super $\kappa^{++}$-$\gp$ implies the $\kappa^{++}$-$\gp$ yet the former does not imply the $\kappa^{++}$-$\gp$ via ultrafilters. 
\end{prop}
\begin{proof}
The first claim is obvious. For the second, observe that the super $\kappa^{++}$-$\gp$ is consistent with $2^\kappa=\kappa^{+}$ (as it holds, for instance,  if $\kappa$ is $\kappa^{++}$-supercompact) yet  the $\kappa^{++}$-$\gp$ via ultrafilters entails $2^\kappa\geq \kappa^{++}$: Let $\langle U_\alpha\mid \alpha<\kappa^{++}\rangle$ be a collection of (non-necessarily distinct) non-principal $\kappa$-complete ultrafilters on $\kappa$. By the $\kappa^{++}$-$\gp$ via ultrafilters there is an increasing sequence of seeds $\langle \eta_\alpha\mid \alpha<\kappa^{++}\rangle$ and a $\kappa$-complete ultrafilter $W$ on $\kappa^{<\kappa}$ which together glue $\langle U_\alpha\mid \alpha<\kappa^{++}\rangle$. In particular, $\kappa^{++}\leq j_W(\kappa)$. Since $W$ is an ultrafilter on $\kappa^{<\kappa}$  we deduce that $\kappa^{++}\leq |j_W(\kappa)|=2^\kappa$.
\end{proof}

Let us continue discussing the optimality of Theorem~\ref{theorem: variantsofgp}. In spite we have not been able to elucidate  if the  $\kappa^+$-$\gp$ entails  the $\kappa^+$-$\gp$ via ultrafilters (see Section~\ref{sec: open problems}),   we can show that the $\kappa^+$-$\gp$ via ultrafilters does not entail the super $\kappa^+$-$\gp$. To prove this,  
we connect the super $\kappa^+$-$\gp$ with  the ability of capturing arbitrary subsets of $\kappa^+$ in ultrapowers via $\kappa$-complete ultrafilters on $\kappa.$  The property we obtain is slightly weaker than the \emph{Local Capturing Property} $\mathrm{LC}(\kappa,\kappa^+)$ considered by Habi\v{c}-Honzik \cite{HavicHonzik} for our witnessing ultrafilter $W$ might not be normal.\footnote{Note that taking the witnessing ultrafilter to be non-normal indeed weakens the large-cardinal assumptions. For example, it is possible that there is a non-normal measure $U$ whose Rudin-Keisler projection to a normal measure $U_{nor}$ belongs to its ultrapower. This is the case for a $\mu$-measurable cardinal, \cite{Mitchell1979}.}

\begin{prop}\label{prop: capturing}
    Suppose that $\kappa$ has the super $\kappa^+$-$\gp$ and $2^\kappa=\kappa^+$. 
    Then, for each $a\s\kappa^+$ there is a $\kappa$-complete ultrafilter $W$ on $\kappa$ such that $a\in \Ult(V,W).$ In particular, if  $\kappa$ has the super $\kappa^+$-$\gp$ then $o^{\mathcal{K}}(\kappa)\geq (\kappa^{++})^{\mathcal{K}}$.
\end{prop}
\begin{proof}
    Fix $a\s\kappa^+$. For each $\alpha<\kappa^+$ we write $\alpha\in^0 a$ whenever $\alpha\in a$; otherwise, we write $\alpha\in^1 a.$
    Let $W_0, W_1$ be distinct $\kappa$-complete ultrafilters on $\kappa$ and let  $X_0\in W_0\setminus W_1$. Define $\langle U_\alpha\mid \alpha<\kappa^+\rangle$ by taking $U_\alpha=W_i$ if and only if $\alpha\in^i a$. By the super $\kappa^+$-$\gp$ there is $j\colon V\rightarrow M$ with $\crit(j)=\kappa$ and $M^\kappa\s M$, and an increasing sequence $\langle \eta_\alpha\mid \alpha<\kappa^+\rangle\in M$ such that $U_\alpha=\{X\s \kappa\mid \eta_\alpha\in j(X)\}$. Let $W$ be the $\kappa$-complete ultrafilter on $\kappa^{<\kappa}$ defined in the proof of Theorem~\ref{theorem: variantsofgp} and $k\colon M_W\rightarrow M$  the corresponding factor embedding. Since $\crit(k)\geq (\kappa^{++})^{M_W}$ and each $\eta_\alpha\in \range(k)$ we deduce that $\langle k^{-1}(\eta_\alpha)\mid \alpha<\kappa^+\rangle\in M_W$. Also, in the  lemma we showed  that $\langle k^{-1}(\eta_\alpha)\mid \alpha<\kappa^+\rangle$ is an increasing sequence of ordinals such that
    $$U_\alpha=\{X\s \kappa\mid k^{-1}(\eta_\alpha)\in j_W(X)\}.$$

   Note that $a\in M_W$ in that 
    $a=\{\alpha<\kappa^+\mid k^{-1}(\eta_\alpha)\in j_W(X_0)\}.$

    Since $W$ is isomorphic to a $\kappa$-complete ultrafilter on $\kappa$ the result follows. 

    \medskip

    For the latter assertion, suppose that for certain $\beta<\kappa^{++}$ we have a $\lhd$-increasing sequence $\langle U_\alpha\mid \alpha<\beta\rangle$ of $\kappa$-complete (non-necessarily normal) ultrafilters on $\kappa$. Since we can code $\langle U_\alpha\mid \alpha<\beta\rangle$ as a subset of $\kappa^+$ (by $2^\kappa=\kappa^+$) there is a $\kappa$-complete ultrafilter $W$ on $\kappa$ such that $\langle U_\alpha\mid \alpha<\beta\rangle\in M_W$. Alas, this is not enough to conclude that $o(\kappa)=\kappa^{++}$ as  $W$ might not be normal. 
    
    Let $\mathcal{K}$ be a core model satisfying maximality (i.e. Clauses (1) and (2) of Theorem \ref{CoreModelTheorem}) and linearity of the Mitchell order. For example, one can take Mitchell's core model for a strong cardinal under the hypothesis that there is no mouse with overlapping extenders or the Jensen-Steel core model under the hypothesis that there is no inner model with a Woodin cardinal. 

    \smallskip
    
    We have the following fact in the core model:
    \begin{claim}
        Assume that for every $a\subseteq \kappa^{+}$ there is a $\kappa$-complete measure $W$ on $\kappa$ such that $a\in \Ult(V,W)$ then $o^{\mathcal{K}}(\kappa)\geq (\kappa^{++})^{\mathcal{K}}$.
    \end{claim}

    \begin{proof}[Proof of claim]
       Assume, towards a contradiction, that $o^\mathcal{K}(\kappa) = \beta < (\kappa^{++})^{\mathcal{K}}$. Let  $\langle U_\gamma \mid \gamma < \beta\rangle$ be the sequence (ordered according to their Mitchell order) of all normal measures on $\kappa$ in $\mathcal{K}$. Let $a \subseteq \kappa^{+}$ be a set (in $V$) coding this sequence of measures.  By our assumption, there is a $\kappa$-complete ultrafilter $W$ such that $a \in \Ult(V,W) = M_W$. Let us look at $j_W\restriction \mathcal{K} \colon \mathcal{K} \to \mathcal{K}^{M_W}$. Then, by maximality, $j_W$ is a normal iteration of extenders in $\mathcal{K}$.\footnote{In general, this iteration might be a branch in an iteration tree, but we only care about the first step.}

        As the critical point of $j_W$ is $\kappa$, the first step is an extender $E\in \mathcal{K}$ on $\kappa$, and the strength of the extender is below the next critical point of the iteration (either by assuming that there are no overlapping extenders, or using the properties of branch in the iteration tree). In particular, $\langle U_\gamma\mid \gamma < \beta\rangle \in \Ult(\mathcal{K}, E)$. Factoring $E$ through its normal measure, $E_\kappa$, and noting that the critical point of the factor map is $(\kappa^{++})^{\mathcal{K}}$ of the ultrapower, we obtain that $\langle U_\gamma \mid \gamma < \beta\rangle \in \Ult(\mathcal K, E_\kappa)$. But this yields a contradiction with the fact that $o^{\mathcal{K}}(\kappa)=\beta<(\kappa^{++})^{\mathcal{K}}.$ 
    \end{proof}
  
    This completes the proof of the lemma.
\end{proof}
\begin{remark}
  The above proposition is akin to Mitchell's  theorem from \cite{MitIter} saying that if $\kappa$ is measurable and $2^{\kappa}= \kappa^{++}$ then there $o^{\mathcal{K}}(\kappa)=\kappa^{++}$.
\end{remark}

\begin{cor}
  Suppose that $\kappa$ is strongly compact and $2^\kappa=\kappa^+$. Then, there is a generic extension where the $\kappa^+$-$\gp$ via ultrafilters holds yet the super $\kappa^+$-$\gp$ fails. Moreover, in that generic extension  the $\lambda$-$\gp$ holds for all cardinals $\lambda$. 
  \end{cor}
  \begin{proof}
      Force with Magidor's iteration from \cite{MagSuper} making $\kappa$ both the first measurable and first strongly compact. This iteration preserves  $2^\kappa=\kappa^+$.  

   In the generic extension, the $\kappa^+$-$\gp$ via ultrafilters holds precisely because $\kappa$ is strongly compact. However, $\kappa$ cannot have the super $\kappa^+$-$\gp$: Any $\kappa$-complete ultrafilter $U$ on $\kappa$ can be coded as a subset of $\kappa^+$ (as $2^\kappa=\kappa^+$) hence   there is a $\kappa$-complete ultrafilter $W$ such that $U\in \Ult(V,W)$ (by Proposition~\ref{prop: capturing}). This implies that $\kappa$ is not the first measurable because if it were, $\Ult(V,W)$ would think that $j_W(\kappa)$ is the first measurable, but $\kappa<j_W(\kappa)$ is itself measurable in $\Ult(V,W)$ – a contradiction.

  Since $\kappa$ is strongly compact the $\lambda$-$\gp$ holds  for all $\lambda$.
  \end{proof}

We now continue our discussion by proving yet another equivalence of the Gluing Property –– one that provides a canonical method for gluing ultrafilters. This will help us in Section~\ref{sec: A model for gluing} when we produce a model of the $\kappa$-$\gp$ from optimal assumptions. This is the first indicator that the Gluing Property is equivalent to the directedness of the Rudin-Keisler order:
\begin{lemma}[Canonical gluing of ultrafilters]\label{lemma:omega-gluing-by-a-measure}
Suppose that $\kappa$ is a measurable cardinal and $\omega\leq \alpha< \kappa^+$. Then
the following assertions are equivalent: 
\begin{enumerate}
     \item $\kappa$ has the $\alpha$-$\gp$.
    \item Every sequence $\langle U_\beta \mid \beta<\alpha\rangle$ of $\kappa$-complete ultrafilters over $\kappa$ admits a $\kappa$-complete ultrafilter $W$ on the increasing sequences of $\kappa^{\alpha}$ (if $\alpha<\kappa$) or $\kappa^{<\kappa}$ (if $\kappa\leq \alpha$) such that $U_\beta\leq_{\mathrm{RK}}{W}$. Furthermore, if $\alpha< \kappa$, the witnessing map $\pi_\alpha$  
   is the evaluation map $\eval_\alpha\colon \kappa^{<\kappa}\rightarrow \kappa$, $$\mathrm{eval}_\alpha\colon s \mapsto s(\alpha);$$ 
   otherwise, the witnessing map is  $\eval_\alpha^*\colon \kappa^{<\kappa}\rightarrow \kappa$ given by $$\eval^*_\alpha(s):=s(c_\alpha(|\min(s)|)),$$
   where $c_\alpha\colon \kappa\rightarrow \kappa$ is either the identity (if $\alpha=\kappa$) or the canonical function representing $\alpha$ (if $\kappa<\alpha)$.
\end{enumerate}
\end{lemma}
\begin{proof}
(1) $\Rightarrow$ (2). Since we are assuming that $\alpha<\kappa^+$  the $\alpha$-$\gp$ of $\kappa$ tantamount to the super $\alpha$-$\gp$. Inspecting the proof of Theorem~\ref{theorem: variantsofgp} one realizes that the ultrafilter $W$ constructed therein fulfills the properties described in the statement of the present lemma.

\medskip

(2) $\Rightarrow$ (1). Let $\langle U_\beta\mid \beta<\alpha\rangle$ be a sequence of $\kappa$-complete ultrafilters on $\kappa$. By our assumption, there is a $\kappa$-complete ultrafilter $W$ either over $\kappa^{\alpha}$ on over $\kappa^{<\kappa}$ that is $\leq_{\mathrm{RK}}$-above  the $U_\beta$'s. Let $\langle \eta_\beta\mid \beta<\alpha\rangle$ be the increasing sequence determining $[\id]_W.$ We claim that  $U_\beta=\{X\s \kappa\mid \eta_\beta\in j_W(X)\}$.

Let $X\s \kappa$. Then, $X\in U_\beta$ if and only if $\pi_\beta^{-1}``X\in W$. This is equivalent to $[\id]_W\in j_W(\pi_\beta^{-1}``X)\Leftrightarrow j_W(\pi_\beta)([\id]_W)\in j_W(X)\Leftrightarrow \eta_\beta\in j_W(X).$
\end{proof}

Next we show that if $\lambda$ is a sufficiently compact cardinal and $\kappa$ has the ${<}\lambda$-Gluing Property, then one can bootstrap this to obtain the ${<}\lambda^+$-Gluing Property. This will help us calculate the exact consistency strength of a cardinal $\kappa$ having the ${<}\kappa^+$-gluing property

\begin{lemma}\label{lemma:compactnessofGluing}
    Suppose that $\kappa$ and $\lambda$ are measurable cardinals. If  $\kappa$ has the ${<}\lambda$-$\gp$ then $\kappa$ has the ${<}\lambda^+$-$\gp.$

     In particular, if $\kappa$ has the ${<}\kappa$-$\gp$ then it also has the ${<}\kappa^+$-$\gp.$
\end{lemma}
\begin{proof}
    Let $\langle U_\alpha\mid \alpha<\gamma\rangle$ be $\kappa$-complete ultrafilters on $\kappa$ where $\gamma<\lambda^+$. Let $j\colon V\rightarrow M$ be the ultrapower by a $\lambda$-complete ultrafilter on $\lambda$. Then, $\langle j(U_\alpha)\mid \alpha<\gamma\rangle\in M$ and $M\models``j(\kappa)$ has the ${<}j(\lambda)$-$\gp$". So, we can find an elementary embedding $i\colon M\rightarrow M^*$ and an increasing sequence $\langle \eta_\alpha\mid \alpha<\gamma\rangle$ of ordinals below $i(j(\kappa))$  such that the ultrafilter $j(U_\alpha)$ can be presented as $$j(U_\alpha)=\{X\in \mathcal{P}^{M}(j(\kappa))\mid \eta_\alpha\in i(X)\}.$$
   Therefore, for each $\alpha<\gamma$, we deduce that $$U_\alpha=\{X\in \mathcal{P}(\kappa)\mid \eta_\alpha\in i(j(X))\}.$$
   Since $M^*$ was closed under $j(\kappa)$-sequences in $M$, $M^*$ is internally definable in $M$ and $M$ is closed under $\kappa$-sequences in $V$, we deduce that $M^*$ is closed under $\kappa$-sequences in $V$. This shows that $\kappa$ has the $\gamma$-$\gp$, as claimed. 
\end{proof}

To conclude our preliminary analysis, notice that the gluing property is, \emph{prima facie},  dependent on the order type of the collection of measures being glued – this is because the sequence of seeds $\langle \eta_\beta\mid \beta<\alpha\rangle$ is required to be increasing. 
 However, the inductive argument presented in Theorem~\ref{theorem: variantsofgp} demonstrates that this dependence is just apparent: 
    \begin{lemma}[No order-type dependence]\label{lemma: cardinal dependence of gluing}
        Let $\kappa$ be a measurable cardinal. Then the following two statements are equivalent for all  cardinals:
        \begin{enumerate}
            \item  $\kappa$ has the $\lambda$-$\gp$ via ultrafilters.
            \item $\kappa$ has the ${<}\lambda^+$-$\gp$ via ultrafilters.
        \end{enumerate}
        In particular, for any ordinal $\gamma$, the following two are equivalent:
        \begin{enumerate}
            \item[(1')] $\kappa$ has the $\gamma$-$\gp$ via ultrafilters;
            \item[(2')] $\kappa$ has the $|\gamma|$-$\gp$  via ultrafilters.\qed
        \end{enumerate}
    \end{lemma}

\section{The Gluing Property and the Rudin-Keisler order}\label{sec: SectionRK}

In this section  we demonstrate that the gluing property via ultrafilters is equivalent to the directedness of the Rudin-Keisler order on $\kappa$. 

\smallskip

Let $\kappa$ be a measurable cardinal and denote by $\mathfrak{U}_\kappa$ the set of all $\kappa$-complete ultrafilters over $\kappa$. Recall that $\langle \mathfrak{U}_\kappa,\leq_{\mathrm{RK}}\rangle$ is a preorder set, so it makes sense to talk about it being directed. The main goal of this section is to show that the statement $``\kappa$ has the $\lambda$-$\gp$" can be characterized in terms of the $\lambda$-directedness of  $\langle \mathfrak{U}_\kappa,\leq_{\mathrm{RK}}\rangle$. Namely, we prove the following:
\begin{theorem}[GP and the RK-order]\label{thm: EquivalencewithRK}
     Suppose that $\kappa$ is a measurable cardinal. Then, the following are equivalent for any cardinal $\lambda\leq 2^\kappa$:
    \begin{enumerate}
       \item $\kappa$ has the ${<}\lambda^+$-$\gp$ via ultrafilters.
        \item $\kappa$ has the $\lambda$-$\gp$ via ultrafilters.
         \item $\kappa$ has the weak $\lambda$-$\gp$ via ultrafilters.
           \item $\langle \mathfrak{U}_\kappa, \leq_{\mathrm{RK}}\rangle$ is $\lambda^+$-directed.
    \end{enumerate}
   Moreover, if $\lambda\leq \kappa$, (1) and (2) and (3) are, respectively, equivalent to
      \begin{enumerate}
        \item[(1')] $\kappa$ has the ${<}\lambda^+$-$\mathrm{GP}$,
        \item[(2')] $\kappa$ has the $\lambda$-$\mathrm{GP}$.
          \item[(3')] $\kappa$ has the weak $\lambda$-$\mathrm{GP}$.
     \end{enumerate}
\end{theorem}
A prequel of this was  proved in \cite[Lemma~2.11]{HP} where the $\aleph_0$-$\gp$ was characterized in terms of the $\aleph_1$-directedness of $\langle \mathfrak{U}_\kappa, \leq_{\mathrm{RK}}\rangle$.
It was recently called to our attention that after \cite{HP}  was published Benhamou–Goldberg 
 independently, and around the same time, observed a similar fact.

\begin{proof}[Proof of Theorem~\ref{thm: EquivalencewithRK}]
(1) $\Leftrightarrow$ (2) is Lemma~\ref{lemma: cardinal dependence of gluing} and (2) $\Rightarrow$ (3) is clear. 

\smallskip

(3) $\Rightarrow$ (4):  Let $\langle U_\alpha \mid \alpha < \lambda \rangle$ be a sequence of $\kappa$-complete ultrafilters. By assumption, there exists a $\kappa$-complete ultrafilter $W$ and a sequence of seeds $\langle \eta_\alpha \mid \alpha < \lambda \rangle$ such that for each $\alpha < \lambda$, 
$U_\alpha = \{ X \subseteq \kappa \mid \eta_\alpha \in j_W(X) \}$. Let $\pi_\alpha \colon \kappa^{<\kappa} \to \kappa$ be  such that $[\pi_\alpha]_W = \eta_\alpha$. This $\pi_\alpha$ witnesses that $U_\alpha \leq_{\mathrm{RK}} W$.

\smallskip

    (4) $\Rightarrow$ (2). Let us argue this by induction on the cardinal $\lambda$. 

    \smallskip

    \underline{Case $\lambda=\aleph_0$:} Let $\langle U_n\mid n<\omega\rangle$ be $\kappa$-complete ultrafilters on $\kappa$. For each $n<\omega$ we consider auxiliary ultrafilters on $\kappa^{n+1}$ by taking products: $$\textstyle \bar{U}_n:=\bigotimes_{m\leq n} U_m.$$

   For each $n<\omega$ let $\varphi_n\colon \kappa^{n+1}\leftrightarrow \kappa$ be a bijection and let $W_n:=(\varphi_n)_{*}(\bar{U}_n)$.

\smallskip

   Since $\langle \mathfrak{U}_\kappa,\leq_{\RK}\rangle$ is $\aleph_0$-directed we can find $W^*$ and functions $\pi_n\colon \kappa\rightarrow\kappa$ witnessing $W_n\leq_{\RK}W^*.$ For each $n<\omega$ let $\sigma_n:=j_{W^*}(\pi_n)([\id]_{W^*}).$  Note that $j_{W^*}(\varphi_n)^{-1}(\sigma_n)$ is an increasing sequence  of ordinals $\langle \eta^n_m\mid m\leq n\rangle$ below $j_{W^*}(\kappa).$ Moreover, for each $m\leq n$, $U_m$ can be presented as
   $$U_m=\{X\s\kappa\mid \eta^n_m\in j_{W^*}(X)\}.$$
Define $\eta_m:=\min\{\eta^n_m\mid m\leq n<\omega\}.$ It is clear that $\eta_m< \eta_{m+1}$ for all $m<\omega$. Therefore, we conclude that $\kappa$ has the $\aleph_0$-$\gp$.


\medskip

\underline{Inductive step:} Suppose that $\langle\mathfrak{U}_\kappa,\leq_{\RK}\rangle$ is $\lambda$-directed.

\smallskip

$\br_1$ Suppose first that $\lambda$ is a limit cardinal. Let $\langle U_\alpha\mid \alpha<\lambda\rangle$ be $\kappa$-complete ultrafilters. By our induction hypothesis, $\kappa$ has the ${<}\lambda$-$\gp$ via ultrafilters. 

For each $\omega\leq \alpha<\lambda$ use the $\alpha$-$\gp$ via ultrafilters  to find a $\kappa$-complete ultrafilter $W_\alpha$ on $\kappa^{<\kappa}$ and an increasing sequence of seeds $\langle \eta^\alpha_\beta\mid \beta<\alpha\rangle$ such that $U_\beta=\{X\s \kappa\mid \eta^\alpha_\beta\in j_{W_\alpha}(X)\}$ for all $\beta<\alpha.$ For each $\alpha<\lambda$ let a bijection $\varphi\colon \kappa^{<\kappa}\leftrightarrow \kappa$ and induce the  ultrafilter
$\bar{W}_\alpha:=\varphi_* W_\alpha.$ These $\bar{W}_\alpha$ are now ultrafilters on $\kappa$, so that we can appeal to the $\lambda$-directedness of $\langle \mathfrak{U}_\kappa,\leq_{\RK}\rangle$ to find $W^*\in\mathfrak{U}_\kappa$ that is $\leq_{\RK}$-above all the $\bar{W}_\alpha$. 

The above discussion yields a commutative diagram:
       \[
\begin{tikzcd}
   V \arrow[r, "{j_{W^*}}"] \arrow[d, "j_{W_\alpha}"'] 
     & M_{W^*} \\
   M_{W_\alpha} \arrow[r, "\bar{k}_\alpha"'] &  M_{\bar{W}_\alpha}  \arrow[u, "k_\alpha"'].
\end{tikzcd}
\]
For each $\beta<\alpha<\lambda$ denote $\sigma^\alpha_\beta:=(k_\alpha\circ\bar{k}_\alpha)(\eta^\alpha_\beta)$. Since $\langle \eta^\alpha_\beta\mid \beta<\alpha\rangle$ was an  increasing sequence so is $\langle \sigma^\alpha_\beta\mid \beta<\alpha\rangle$ as well.

Chasing the diagram one shows that
$$U_\beta=\{X\s \kappa\mid \sigma^\alpha_\beta\in j_{W^*}(X)\}.$$
Given $\beta<\lambda$ the above holds true for all $\alpha\in (\beta,\lambda)$. Thus, as before, it suffices to show that one can pick a sequence $\langle \sigma^{\alpha_\beta}_\beta\mid \beta<\lambda\rangle$ that is increasing. But this is easily guaranteed taking $\eta_\beta:=\min\{\sigma_\beta^\alpha\mid \beta<\alpha<\omega_1\}$. 

\medskip

$\br_2$ Suppose that $\lambda$ is a successor cardinal, $\mu^+$. By the induction hypothesis, $\kappa$ has the $\mu$-$\gp$ via ultrafilters and {thus the ${<}\mu^+$-$\gp$ via ultrafilters.}  Using this  one can argue exactly as before to show that $\kappa$ has  the $\mu^+$-$\gp$. 
\end{proof}
Kunen \cite[Theorem~2.3]{Ketonen} and, independently,  Comfort–Negrepontis \cite[Theorem~4.3]{ComfortNegrepontis} showed that if $\kappa$ is $\kappa$-compact then $\langle \mathfrak{U}_\kappa,\leq_{\mathrm{RK}}\rangle$ is $(2^\kappa)^+$-directed, and thus it has the $2^\kappa$-$\gp$ via ultrafilters. This observation improves the conclusion of \cite[Theorem~4.2]{HP} which established the $2^\kappa$-$\gp$ for this type of cardinals. Also, note that this implies that a $\kappa$-compact cardinal cannot provide a counter-example to our conjecture that the $\kappa^+$-$\gp$ does not imply the $\kappa^+$-$\gp$ via ultrafilters.

\section{The Gluing Poset}\label{sec: the gluing iteration}
\begin{setup}
   For the scope of this section we assume that  $\kappa$ is a measurable cardinal and that $\mathcal{U}=\langle U(\alpha,\zeta)\mid \alpha\leq  \kappa,\, \zeta<o^{\mathcal{U}}(\alpha)\rangle$ is a  coherent sequence of  measures such that {$o(\alpha)< \alpha$} for all $\alpha< \kappa$.  
\end{setup}

Let $\mathbb{S}$ be the forcing adding a fast function  $\ell\colon \kappa\rightarrow H(\kappa)$ using conditions whihc are partial functions $s\colon \kappa\rightarrow H(\kappa)$ with non-stationary support (Definition~\ref{def: fast function forcing}). In what follows we shall work inside $V[S]$ where $S\s \mathbb{S}$ is a $V$-generic filter and $\ell:=\bigcup S$ is the induced  generic fast function. For each $\alpha<\kappa$ inaccessible, we denote by $S_\alpha$ the $\mathbb{S}_\alpha$-generic filter over $V$ induced by $S$.  We define, by induction on  $\alpha<\kappa$, a non-stationary support iteration of Prikry-type forcings $\langle \mathbb{P}_\alpha, \dot{\mathbb{Q}}_\beta\mid \beta<\alpha\leq \kappa\rangle$  as follows. Suppose that $\mathbb{P}_\alpha$ was defined. If $\alpha$ is non measurable in $V$ then we let $\dot{\mathbb{Q}}_\alpha$ to be the standard $\mathbb{P}_\alpha$-name for the trivial forcing.  Otherwise, we define $\dot{\mathbb{Q}}_\alpha$ by recursion on the ``order" of $\alpha$ in the style of Gitik \cite{ChangingCofinalities} (see also \cite[\S4]{BenUng}). The precise meaning of the   ``order of $\alpha$" that is appropriate for our purposes is a  variation of the usual \emph{Mitchell order} of $\alpha$. To be in a position to present this notion, first we recall the notion of a \emph{code for a measure on $\alpha$}  introduced in \cite[Theorem~6.21]{HP}. 
\begin{definition}[Codes]\label{def: codes}
    Working in $V[S_\alpha]$ we say that a $5$-tuple $$c=\langle \vec \rho, \vec \zeta, \vec a, f, g\rangle$$ is a \emph{code for a measure on $\alpha$} if there is $n\geq 1$ for which the following hold (using the notations of Definition \ref{Iteration}):
\begin{enumerate}
    \item $\vec\rho=\langle \rho_0,\dots, \rho_{n-1}\rangle$  where $\rho_i \colon \alpha^i \to (\alpha+1)$ and $\rho_0(\langle\rangle):=\alpha$.
    \item $\vec\zeta=\langle \zeta_0,\dots, \zeta_{n-1}\rangle$ is a sequence of ordinals ${\leq}\alpha$. 
    \item $\vec a =\langle a_0,\dots, a_{n-1}\rangle$ is a sequence of functions $a_i \colon \alpha^i \to H(\alpha^{+})$.
    \item $f \colon \alpha^n \to  \mathbb{P}_\alpha$. 
    \item $g \colon \alpha^n \to \alpha$.
    \item For each $i < n$, $\iota_i(\rho_i)(\langle \mu_j \mid j < i\rangle) = \mu_i$ and $\zeta_i < o^{\iota_i(\mathcal{U})}(\mu_i)$. $$(\text{Here } \iota_{i,i+1}\colon \mathscr{N}_i \to \mathscr{N}_{i+1} \cong \Ult(\mathscr{N}_i, \iota_i({\mathcal{U}})(\mu_i, \zeta_i)),\,  \mathcal{N}_0 := V\,\text{and}\,\iota_0 := \id.)$$ 
\item Let $\iota^*_n$ be \textbf{the} lifting of $\iota_n$ to $V[S_\alpha]$ determined by $\langle \bar a_i\mid i<n\rangle$, where $\bar a_i=\iota_i(a_i)(\langle \mu_j\mid j<i\rangle)$.\footnote{\label{Footnoteliftings}Such lifting exists by Lemma~\ref{lem;lifting-S} – it is unique and definable in $V[S_\alpha]$. Indeed, such lifting is determined by $\iota_n$, the $\mathbb{S}_\alpha$-generic $S_\alpha$ and a sequence $\langle \bar{a}_i\mid i<n\rangle \in \prod_{i<n}H(\mu_i^+)^{V}$. Here we take $\bar{a}_i$ to be $\iota_i(a_i)(\langle \mu_j \mid j<i\rangle)$.} Set $$\text{$r = \iota^*_n(f)(\langle \mu_i \mid i < n\rangle),$  $\varepsilon = \iota^*_n(g)(\langle \mu_i \mid i < n\rangle)$,}$$ and let $\dot{U}$ be the $\mathbb{P}_\alpha$-name (over $V[S_\alpha]$) for the set 
\[\begin{matrix} \dot{U} &= &\{\dot{X}_{S_\alpha \ast \dot{G}_\alpha} \subseteq \alpha \mid & \exists p \in \dot{G}_\alpha \; \exists s\in \iota^*_n(\mathbb{P}_\alpha), \\ & & & s \leq^*_{\Gamma_{\iota_n}} \iota^*_n(p) \wedge r, s\in \mathscr{N}_n, \\ & & & \mathscr{N}_n \models s \Vdash_{\iota^*_n(\mathbb{P}_\alpha)} \varepsilon \in \iota^*_n(\dot{X})\}.\end{matrix}\]
Then, it is the case that
$$\one\forces^{V[S_\alpha]}_{\mathbb{P}_\alpha}\text{$``\dot{U}$ is an $\alpha$-complete ultrafilter on $\alpha$''.}$$
\end{enumerate}

 In the above-described situation we will say that $\dot{U}$ is \emph{coded} by $\langle \vec\rho, \vec\zeta, \vec a, f, g\rangle$. For the sake of a more concise language, we will say that $c=\langle \vec\rho, \vec\zeta, \vec a, f, g\rangle$ is an \emph{$\alpha$-code} rather than a code for a measure on $\alpha$.
\end{definition}

\begin{remark}
    Note that an $\alpha$-code $c=\langle \vec \rho, \vec \zeta, \vec a, f, g\rangle$ always belongs to $H(\alpha^+)^V$. Therefore, $c$ is a legitimate choice for $s(\alpha)$, being $s$ a condition in the fast function  forcing $\mathbb{S}$.  {Also, notice that codes are absolute in the following sense: If $\mathscr{M}\s V[S_\alpha]$ is a $\kappa$-closed inner model for which $\mathscr{M}\models ``c$ is a code" then the measure decoded from $c$ is the same in $V[S_\alpha\ast G_\alpha]$ and in $\mathscr{M}[G_\alpha]$.} The forthcoming \textbf{Coding Lemma} (see Lemma~\ref{lemma: coding lemma}) will show that every $\alpha$-complete ultrafilter on $\alpha$ in $V[S_\alpha\ast G_\alpha]$ admits a code. This will provide a complete characterization of the $\alpha$-complete ultrafilters (so, including non necessarily normals) in the generic extension by $V[S_\alpha\ast G_\alpha].$
\end{remark}

With the concept of code in place, we  present our variation of the Mitchell order for a measurable cardinal $\alpha$. Recall that $\ell\colon \kappa\rightarrow H(\kappa)$ is fixed.
\begin{definition}\label{ellorder}
Let $\Omega$ be a non-zero ordinal.
We write \emph{$o^\ell(\alpha)=\Omega$} if 
$$\one\forces^{V[S_\alpha]}_{\mathbb{P}_\alpha}\text{``$\ell(\alpha)=\langle \langle c_\beta\mid \beta<\omega^\Omega\rangle, \langle \zeta^\rho_\beta\mid 1\leq\rho<\Omega\,\wedge\,\omega^\rho\cdot\beta<\omega^\Omega \rangle\rangle$''},$$ and the following properties hold:
\begin{enumerate}
    \item $\langle c_\beta\mid \beta<\omega^\Omega\rangle$ is forced to be a sequence of $\alpha$-codes.

\item $\zeta^\rho_\beta<o^{V}(\alpha)$ and if $\iota_{\alpha,\zeta^\rho_\beta}\colon V\rightarrow \mathscr{N}_{\alpha, \zeta^\rho_\beta}\simeq\Ult(V, U(\alpha,\zeta^\rho_\beta))$ then:

\smallskip

\begin{enumerate}
 \item  $\mathscr{N}_{\alpha, \zeta^\rho_\beta}[\iota_{\alpha, \zeta^\rho_\beta}(S_\alpha)]\models ``\one\forces_{\mathbb{P}_\alpha}\text{$\vec{c}^{\;\rho}_\beta$ are $\alpha$-codes''}$.

 \item $\iota_{\alpha, \zeta^\rho_\beta}(\ell)(\alpha)=\langle \vec{c}^{\,\rho}_\beta,\vec{\zeta}^{\;\rho}_\beta\rangle.$

\item $\langle U(\alpha,\zeta)\mid \zeta\in \vec{\zeta}^{\,\rho}_\beta \rangle\in \mathscr{N}_{\alpha, \zeta^\rho_\beta}$.
\end{enumerate}

\end{enumerate}
Above we used the following notation:
\begin{itemize}
    \item $\vec{c}^{\;\rho}_\beta:=\langle c_\gamma\mid \gamma\in [\omega^\rho\cdot\beta,\omega^\rho\cdot(\beta+1))\rangle$.
    \item $\vec{\zeta}^{\,\rho}_\beta:=\langle \zeta^\sigma_\gamma\mid 1\leq \sigma<\rho\,\wedge\, \gamma\in[\omega^\sigma\cdot \beta, \omega^\sigma\cdot(\beta+1)) \rangle\rangle$.
\end{itemize}

We say that $\alpha$ has \emph{$\ell$-order $\Omega$} if $o^\ell(\alpha)=\Omega$. Similarly, we say that $\alpha$ has  \emph{$\ell$-order $0$} (in symbols, $o^\ell(\alpha)=0$)  if for no ordinal $\Omega> 0$, $o^\ell(\alpha)=\Omega$. 
\end{definition}

    So, for instance, $o^\ell(\alpha)=1$ means that $\mathbb{P}_\alpha$ forces (over $V[S_\alpha]$) that $\ell(\alpha)=\langle c_n\mid n<\omega\rangle$ is a sequence of $\alpha$-codes. Similarly, $o^\ell(\alpha)=2$ means that $\mathbb{P}_\alpha$ forces (over $V[S_\alpha]$) that $\ell(\alpha)=\langle \langle c_\beta\mid \beta<\omega^2\rangle, \langle \zeta_\beta\mid \beta<\omega \rangle\rangle$ – the first being a sequence of $\alpha$-codes and the second witnessing that $$\mathscr{N}_{\alpha,\zeta_\beta}[\iota_{\alpha,\zeta_\beta}(S_\alpha)]\models``\text{$\alpha$ has $\iota_{\alpha, \zeta_\beta}(\ell)$-order  $1$"}.$$
    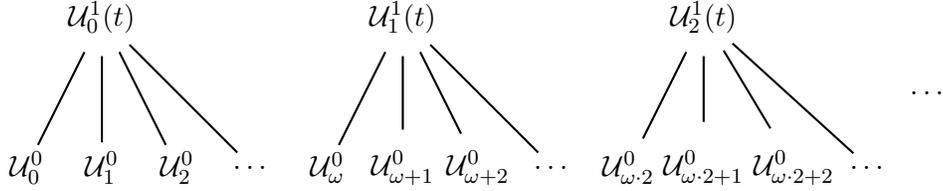
\begin{figure}[h]
\centering
\begin{tikzpicture}
    [every node/.style={draw=none, circle, inner sep=0.3pt},
    every path/.style={thick}]

    \node (U10) at (0,2) {$\mathcal{U}_0^1(t)$};
    \node (U11) at (4,2) {$\mathcal{U}_1^1(t)$};
    \node (U12) at (8,2) {$\mathcal{U}_2^1(t)$};
    
    \node (U00) at (-1,0) {$\mathcal{U}_0^0$};
    \node (U01) at (0,0) {$\mathcal{U}_1^0$};
    \node (U02) at (1,0) {$\mathcal{U}_2^0$};
    \node (U03) at (2,0) {$\dots$};

    \node (U0W1) at (3,0) {$\mathcal{U}_\omega^0$};
    \node (U0W2) at (4,0) {$\mathcal{U}_{\omega+1}^0$};
    \node (U0W3) at (5,0) {$\mathcal{U}_{\omega+2}^0$};
    \node (U0W4) at (6,0) {$\dots$};
    
    \node (U0W5) at (7,0) {$\mathcal{U}_{\omega\cdot2}^0$};
    \node (U0W6) at (8,0) {$\mathcal{U}_{\omega\cdot2+1}^0$};
    \node (U0W7) at (9.2,0) {$\mathcal{U}_{\omega\cdot2+2}^0$};
     \node (U0W8) at (10.2,0) {$\dots$};
    
    \draw (U10) -- (U00);
    \draw (U10) -- (U01);
    \draw (U10) -- (U02);
    \draw (U10) -- (U03);
    
    \draw (U11) -- (U0W1);
    \draw (U11) -- (U0W2);
    \draw (U11) -- (U0W3);
    \draw (U11) -- (U0W4);

    \draw (U12) -- (U0W5);
    \draw (U12) -- (U0W6);
    \draw (U12) -- (U0W7);
    \draw (U12) -- (U0W8);
    
    \node[draw=none, fill=none] (dots) at (11,1) {\dots};
\end{tikzpicture}
\caption{$o^\ell(\alpha)=2.$}
\label{fig:order}
\end{figure}

 Figure~\ref{fig:order} illustrates the situation for $o^\ell(\alpha)=2$. The ultrafilters at the bottom, $\mathcal{U}^0_\beta$, are those decoded in $V[S_\alpha\ast G_\alpha]$ from the given codes $c_\beta$. For each $\omega$-block of measures at the bottom, the $V$-measure $U(\alpha,\zeta_\beta)$ uniquely extends (in $V[S_\alpha]$) to  an $\alpha$-complete ultrafilter $\mathcal{U}^1_\beta$ --- recall that thanks to the non-stationary support, this lifting is fully determined by $\iota_{\alpha,\zeta_\beta}``S_\alpha$ and $\iota_{\alpha,\zeta_\beta}(\ell)(\alpha)=\langle \vec{c}^{\,1}_\beta,\varnothing\rangle$ (Lemma~\ref{lem;lifting-S}). We will extend each $\mathcal{U}^1_\beta$ to an $\alpha$-complete ultrafilter $\mathcal{U}^1_\beta(t)$ inside $V[S_\alpha\ast G_\alpha]$. As in \cite{ChangingCofinalities}, these extensions will depend upon finite increasing sequences $t$ that are appropriate to the Tree Prikry forcing relative to the measures $\langle \mathcal{U}^0_{\omega\cdot\beta+m}\mid m<\omega\rangle$. Then we  define a poset $\mathbb{G}_\alpha^2$, based on the measures of ``order $1$" $\mathcal{U}^1_\beta(t)$ and the measures of ``order $0$" $\mathcal{U}^0_\beta$. After forcing with  $\mathbb{G}_\alpha^2$ we will show that $\langle \mathcal{U}^0_\beta\mid \beta<\omega^2\rangle$ are glued via an $\alpha$-complete ultrafilter on $\alpha^{\omega^2}.$ 

 \begin{remark}
     Let us point out two differences with  Gitik's setup from \cite{ChangingCofinalities}. First, Gitik has a coherent sequence of measures $\mathcal{U}$ all the way up to $\kappa$ and therefore one measure  per order. Instead, here we imagine  the given ultrafilters $\mathcal{U}^0_\beta$ as if they all  had ``Mitchell order" $0$. The second most substantial difference is the following: {The ultrafilters to be glued (i.e., $\mathcal{U}^0_\beta$) may not be normal and as a result their extensions $\mathcal{U}^1_\beta(\varnothing)$ may not be either. This is a key difference with Gitik's iteration from \cite{GitikNonStaI}. Normality is critical in the proof of Claim~\ref{claim: bounding ordinals}, which in turn is instrumental to establish the failure of the $\aleph_0$-$\gp$ (Theorem~\ref{thm: no gluing in gitiks}).} 
 \end{remark}

The $\alpha$-stage  of our  iteration is defined according to the $\ell$-order of $\alpha$ as
  $$\one\forces_{\mathbb{P}_\alpha}^{V[S_\alpha]}\text{$``\dot{\mathbb{Q}}_\alpha=\dot{\mathbb{G}}^{o^\ell(\alpha)}_\alpha$''},$$
 where  $\mathbb{G}^{o^\ell(\alpha)}_\alpha$ denotes the \emph{gluing poset}, to be defined by recursion on $o^\ell(\alpha)$.
  \begin{definition}
       Working in $V[S_\alpha]$, if   $o^\ell(\alpha)\in \{0,1\}$ then we set
  $$\dot{\mathbb{G}}^{o^\ell(\alpha)}_\alpha:=\begin{cases}
     \{\one\}, & \text{if $o^\ell(\alpha)=0$;}\\
      \mathbb{T}(\langle \dot{\mathcal{U}}_n\mid n<\omega\rangle), & \text{if $o^\ell(\alpha)=1$,}
  \end{cases}$$
  where $\langle \dot{\mathcal{U}}_n\mid n<\omega\rangle$ is the standard $\mathbb{P}_\alpha$-name (over $V[S_\alpha]$) for the $\omega$-sequence of measures decoded from each $\alpha$-code $c_n$ mentioned in $\ell(\alpha)=\langle c_n\mid n<\omega\rangle.$
  \end{definition}
  \begin{remark}
        Let us stress that the definition of $\mathbb{G}^{o^\ell(\alpha)}_\alpha$ not only depends on the $\ell$-order of $\alpha$ but also on the information carried by  $\ell(\alpha)$ (i.e., the codes.). However, since $\ell$ is fixed from the outset, we shall refrain from emphasizing the dependence on $\langle\ell(\alpha)_n\mid n<\omega\rangle$.
  \end{remark}

\smallskip

When $o^\ell(\alpha)=1$,  $\mathbb{G}^1_\alpha$ is just the garden variety \emph{Tree Prikry forcing} used in \cite[Theorem~6.21]{HP}. In the cases where $o^\ell(\alpha)\geq 2$ the definition of the gluing poset  is obtained as the outcome of a recursion in the style of \cite{ChangingCofinalities}. To make the construction more transparent we  spell out the details of the case $o^\ell(\alpha)= 3$ and then we take over the general case. For each measurable $\bar\alpha<\alpha$ let us denote by $c_{\bar\alpha}$ the club that is $\mathbb{G}^{o^\ell(\bar\alpha)}_\alpha$-generic  over $V[S_{\bar\alpha}\ast G_{\bar\alpha}]$.

\smallskip

We commence introducing the notion of $\ell$-coherency, which is the natural adaptation of Gitik's coherency from \cite[\S3]{ChangingCofinalities}: 
\begin{definition}[$\ell$-coherency]
Let $\rho\leq o^{\ell}(\alpha)$ be an ordinal.
An increasing sequence of ordinals $t=\langle \alpha_0,\dots, \alpha_{k-1}\rangle\in [\alpha]^{<\omega}$ is called \emph{$(\rho, \ell)$-coherent} if: 
\begin{enumerate}
    \item $t$ is the empty sequence, $\varnothing$, whenever $\rho=0$.
    \item For each index $i<k$, $o^\ell(\alpha_i)<\rho$,
    \item For each $0<i<k$, let $n_i\leq i$ denote  the first index such that $o^\ell(\alpha_j)<o^\ell(\alpha_i)$ for all $j\in [n_i, i)$. We have two options:

    \smallskip
    
\begin{enumerate}
    \item \underline{$n_i<i$}: In this case we require
    $$\textstyle \bigcup_{j\in [n_i, i)}(b_{\alpha_j}\cup\{\alpha_j\})\sq b_{\alpha_i}.$$
    \item \underline{$n_i=i$}: Then we require the generic club on $\alpha_i$ start  above $\alpha_{i-1}$:
    $$\min(b_{\alpha_i})>\alpha_{i-1}.\footnote{Note that in this case $o^{\ell}(\alpha_{i-1})\geq o^{\ell}(\alpha_i)$.}$$
\end{enumerate}

\end{enumerate}
\end{definition}
To streamline the presentation let us refer to $(\rho,\ell)$-coherent sequences just as $\rho$-coherent sequences. Notice that $t$ is $1$-coherent if and only if $t$ is an increasing sequence of ordinals below $\alpha$.

\begin{definition}[Projections]
    Let $t=\langle \alpha_0,\dots, \alpha_{k-1}\rangle$ be a $\rho$-coherent sequence  and $\sigma<\rho$. The \emph{projection of $t$ to $\sigma$} is the sequence $$t\restriction \sigma:=\langle \alpha_{n_\sigma},\dots, \alpha_{k-1}\rangle$$ where $n_\sigma\leq k$ is the least index such that $o^\ell(\alpha_i)<\sigma$ for all $i\in [n_\sigma,k)$.
\end{definition}
Since we allow $n_\sigma=k$ the restriction $t\restriction\sigma$ might possibly be empty. Also, a moment's reflection makes clear that  $t\restriction\sigma$ is  $\sigma$-coherent.
\begin{definition}\label{CharacteristicsOfT}
    Let $t$ be a $\rho$-coherent sequence.
    \begin{enumerate}
        \item The \emph{Magidor/Gitik generic sequence of $t$}, $b_t$, is $\bigcup_{i<|t|}(b_{\alpha_i}\cup\{\alpha_i\})$.
        \item $t$ is \emph{maximal} if the sequence $\langle o^\ell(\alpha_i)\mid i<|t|\rangle$ is $\leq$-weakly decreasing.
        \item If $\sigma<\rho$, the collection of ordinals of $\ell$-order $\sigma$ in $t$  is $$t_\sigma:=\{\alpha\in t\mid o^{\ell}(\alpha)=\sigma\}.$$
        \item Given another $\rho$-coherent sequence $s$, we say that $t$ and $s$ are \emph{equivalent} if they yield the same Gitik generic sequence – to wit, $b_t=b_{s}.$
    \end{enumerate}
\end{definition}
\begin{remark}
    Every $\rho$-coherent sequence is equivalent to a maximal $\rho$-coherent sequence. The idea  is the following: For each $i<|t|$ one can ``collapse" the block of ordinals $b=\langle \alpha_{n_i},\dots, \alpha_i\rangle$ in $t$ to $\langle \alpha_i\rangle$ and yet maintain $\rho$-coherency. The point is that the Gitik sequence carried by $\alpha_i$ is the same as $c_b$. Thus, the information carried by the ordinals $\langle \alpha_{n_i},\dots, \alpha_{i-1}\rangle$ is redundant as far as the generic object is concerned.
\end{remark}


Let $G_\alpha\s \mathbb{P}_\alpha$ be $\mathbb{P}_\alpha$-generic filter over $V[S_\alpha]$. Then, $$V[S_\alpha\ast G_\alpha]\models ``\ell(\alpha)=\langle \langle c_\beta\mid \beta<\omega^3\rangle, \langle \zeta^1_\beta\mid \beta<\omega^2\rangle, \langle \zeta^2_\beta\mid \beta<\omega\rangle\rangle"$$ 
where the first is a sequence of $\alpha$-codes and the other two  are sequences of orders for $V$-measures $U(\alpha,\zeta^\rho_\beta)$ for which the following hold: 
\begin{equation}\label{notationcodes}
    \tag{$\dagger$} \mathscr{N}_{\alpha, \zeta^\rho_\beta}[\iota_{\alpha, \zeta^\rho_\beta}(S_\alpha)\ast G_\alpha]\models \text{$``\vec{c}^{\;\rho}_\beta$ are $\alpha$-codes''},
    \end{equation}
  \begin{equation*}
      \iota_{\alpha,\zeta^1_\gamma}(\ell)(\alpha)=\langle \vec{c}^{\,1}_\gamma,\varnothing\rangle\;\text{and}\; \iota_{\alpha, \zeta^2_\beta}(\ell)(\alpha)=\langle \vec{c}^{\,2}_\beta,\vec{\zeta}^{\,2}_\beta\rangle,
  \end{equation*}
  and $\langle U(\alpha, \zeta)\mid \zeta\in \vec{\zeta}^{\,1}_\beta\rangle\in \mathscr{N}_{\alpha,\zeta^2_\beta}$ for all $\beta<\omega.$
  
  (Look at Definition~\ref{ellorder} to recover the notations $\vec{c}^{\,\rho}_\beta$ and $\vec{\zeta}^{\,\rho}_\beta$.)
\begin{definition}[The measures]\label{Themeasures}\hfill
\begin{enumerate}
    \item For each $\beta<\omega^3$,  denote by $\mathcal{U}^0_\beta$ the $\alpha$-complete measure over $\alpha$ decoded (in $V[S_\alpha\ast G_\alpha]$) from the code $c_\beta$.
    \item For each $\rho\in\{1,2\}$ and $\beta$ be such that $\omega^\rho\cdot\beta<\omega^3$, let   
    $$\mathcal{U}^\rho_\beta:=\{X\s \alpha\mid \alpha\in {\iota}_{\alpha,\zeta^\rho_\beta}(X)\},$$
    where ${\iota}_{\alpha, \zeta^\rho_\beta}\colon V[S_\alpha]\to \mathscr{N}_{\alpha,\zeta^\rho_\beta}[H_\alpha]$ is the unique ($V[S_\alpha]$-internal)  lifting of the ultrapower embedding by $U(\alpha,\zeta^\rho_\beta)$ for which 
    $$\iota_{\alpha,\zeta^\rho_\beta}(\ell)(\alpha)=\langle \vec{c}^{\;\rho}_\beta,\vec{\zeta}^{\,\rho}_\beta\rangle.$$
\end{enumerate}
\end{definition}
\begin{remark}
    Since $\iota_{\alpha, \zeta^\rho_\beta}$ is the ultrapower embedding by the normal measure $U(\alpha,\zeta^\rho_\beta)$ and $\mathbb{S}_\alpha$ is a forcing whose conditions have non-stationary support,  any lifting of $\iota_{\alpha, \zeta^\rho_\beta}$  is fully determined by $\iota_{\alpha,\zeta_\beta}``S_\alpha$ and  $\iota_{\alpha,\zeta_\beta}(\ell)(\alpha).$ 
    In addition, one can show that $\mathscr{N}_{\alpha, \zeta^\rho_\beta}[\iota_{\alpha, \zeta^\rho_\beta}(S_\alpha)]$ is an $\alpha$-closed inner model of $V[S_\alpha]$ – to prove this one uses the fusion properties of $\mathbb{S}_\alpha$ (\cite[Claim~6.21.1]{HP}).
\end{remark}

 Let us stress that the measures $\mathcal{U}^0_\beta$'s  are decoded the same way in the ambient  universe $V[S_\alpha\ast G_\alpha]$ and in the inner models $\mathscr{N}_{\alpha, \zeta^\rho_\beta}[\iota_{\alpha, \zeta^\rho_\beta}(S_\alpha)\ast G_\alpha]$. (The reason is that
 $\mathscr{N}_{\alpha, \zeta^\rho_\beta}[\iota_{\alpha, \zeta^\rho_\beta}(S_\alpha)]$
 is an $\alpha$-closed inner model of $V[S_\alpha]$.)   Also,  note that while $\mathcal{U}^0_\beta$ are measures in $V[S_\alpha\ast G_\alpha]$ their companions $\mathcal{U}^\rho_\beta$'s are measures in $V[S_\alpha]$. Thus, in the forthcoming recursive construction of the gluing poset $\mathbb{G}^3_\alpha$ we keep $\mathcal{U}^0_\beta$ unaltered and lift the $\mathcal{U}^\rho_\beta$'s to  $V[S_\alpha\ast G_\alpha]$. 

 \smallskip
 
 The poset $\mathbb{G}^3_\alpha$ will be designed to glue $\langle \mathcal{U}^0_\beta\mid \beta<\omega^3\rangle$ to a measure on $\alpha^{\omega^3}$.

 \smallskip

\underline{$\ell$-order $0$:} For each $\beta<\omega^3$ let $\mathbb{G}^{0,\beta}_\alpha$ be  the trivial forcing. 

\smallskip

\underline{$\ell$-order $1$:} For each $\beta<\omega^2$ define  (in $V[S_\alpha\ast G_\alpha]$) $\mathbb{G}^{1,\beta}_\alpha$ as follows.  
\begin{definition}\label{Babycase}
    Conditions in $\mathbb{G}^{1,\beta}_\alpha$ are pairs $\langle t, T\rangle$ where $t\in [\alpha]^{<\omega}$ is a $1$-coherent sequence and $T$ is a tree on $[\alpha]^{<\omega}$ with the following properties:
    \begin{enumerate}
        \item $\varnothing\in T$ and every $s\in T$ is $1$-coherent;
        \item  $t^\smallfrown s$ is $1$-coherent (i.e., increasing) for all $s\in T$;
        \item for each $s\in T$, $\mathrm{Succ}_T(s)\in \mathcal{U}^0_{\omega\cdot\beta+|t^\smallfrown s|}$.
    \end{enumerate}
        The order between conditions in $\mathbb{G}_\alpha^{1,\beta}$ is the natural one.
\end{definition}  

Thus, $\mathbb{G}^{1,\beta}_\alpha$ is  the tree Prikry forcing relative to the $\beta$th-block of measures 
$$\langle \mathcal{U}^0_{\omega\cdot\beta+m}\mid m<\omega\rangle.$$  For each $1$-coherent sequence $t$ 
let us extend $\mathcal{U}^1_\beta$ to $\mathcal{U}^1_\beta(t)$, a measure in the generic extension $V[S_\alpha\ast G_\alpha]$. 
For this note that  
our choice of $\zeta^1_\beta$ guarantees that the $\iota_{\alpha,\zeta^1_\beta}(\ell)$-order of $\alpha$ in $\mathscr{N}_{\alpha, \zeta^1_\beta}[\iota_{\alpha,\zeta^1_\beta}(S_\alpha)]$ is $1$ – this is witnessed by the codes in $\vec{c}^{\;1}_\beta$  (see \eqref{notationcodes} on page~\pageref{notationcodes}). In particular, $\{\nu<\alpha\mid o^\ell(\nu)=1\}\in \mathcal{U}^1_\beta.$ 

By elementarity, $\mathscr{N}_{\alpha, \zeta^1_\beta}[\iota_{\alpha,\zeta^1_\beta}(S_\alpha)\ast G_\alpha]$ thinks that the $\alpha$th-stage of the iteration $\iota_{\alpha, \zeta^1_\beta}(\mathbb{P}_\alpha)$ is precisely the tree Prikry forcing relative to the measures decoded from the $\alpha$-codes in $\vec{c}^{\,1}_\beta$. More explicitly, 
$$\mathscr{N}_{\alpha,\zeta^1_\beta}[\iota_{\alpha,\zeta^1_\beta}(S_\alpha)\ast G_\alpha]\models ``\mathbb{Q}_\alpha=\mathbb{G}^{1,\beta}_\alpha".$$
($\mathbb{G}^{1,\beta}_\alpha$ is the same as computed in $\mathscr{N}_{\alpha,\zeta^1_\beta}[\iota_{\alpha,\zeta^1_\beta}(S_\alpha)\ast G_\alpha]$ and in $V[S_\alpha\ast G_\alpha]$.)

\begin{definition}\label{ExtendingTheMeasure1}
Working in $V[S_\alpha\ast G_\alpha]$, let  $\mathcal{U}^1_\beta(t)$ be the collection of all $\dot{X}_{G_\alpha}\s \alpha$ for which there is $p\in G_\alpha$ and $\langle t, T\rangle\in {\mathbb{G}}^{1,\beta}_\alpha$ such that
    $$\mathscr{N}_{\alpha,\zeta^1_\beta}[\iota_{\alpha,\zeta^1_\beta}(S_\alpha)]\models ``p^\smallfrown \langle t , T\rangle^\smallfrown \iota_{\alpha, \zeta^1_\beta}(p)\setminus (\alpha+1)\forces_{\mathbb{P_\alpha}}\alpha\in \iota_{\alpha,\zeta^1_\beta}(\dot{X})".$$
\end{definition}
The next fact is routine:
\begin{lemma}
In $V[S_\alpha\ast G_\alpha]$, $\mathcal{U}^1_\beta(t)$ is an $\alpha$-complete ultrafilter over $\alpha$   extending $\mathcal{U}^1_\beta$. In addition, $\mathcal{U}^1_\beta(t)\ni \{\nu<\alpha\mid o^\ell(\nu)=1\,\wedge\, {t}\sq b_\nu\}.$ 
\qed
\end{lemma}
Next we prove an additional property of the measures $\mathcal{U}^1_\beta(t)$ which will play a role in the proof of the consistency of the gluing property. 
\begin{lemma}\label{lemma: disjointifying levels}
    Let $\beta, m<\omega$ and $A\in \mathcal{U}^0_{\omega\cdot \beta+m}$. Then, for each $t$ with $|t|=m$,
    $$\{\nu<\kappa\mid \min(b_\nu)\setminus (\max(t)+1) \in A \}\in \mathcal{U}^1_\beta(t).$$
\end{lemma}
\begin{proof}
Suppose that the thesis was false and let $p\in G$ and $\dot{T}$ witnessing it. Working in $V[S_\alpha\ast G_\alpha]$, since $\mathrm{Succ}_{T}(\varnothing)\in \mathcal{U}^0_{\omega\cdot\beta+m}$, we can shrink $T$  to another tree $T^*$  with $\mathrm{Succ}_{T^*}(\varnothing)\s A$. By extending $p\in G$ if necessary, $$p\forces_{\mathbb{P}_\kappa}``\langle t,\dot{T}^*\rangle\leq^*\langle t, \dot{T}\rangle".$$ It is clear then that $p{}^\smallfrown \langle t,\dot{T}^*\rangle{}^\smallfrown \iota_{\alpha,\zeta^1_\beta}(p)\setminus (\alpha+1)$ forces $\alpha$ in the $\iota_{\alpha,\zeta^1_\beta}$-image of the set of interests – this yields a contradiction with our assumption.
\end{proof}


This completes the definition of the gluing poset for $\ell$-order $1$.

\medskip

\underline{$\ell$-order $2$:} For each $\beta<\omega$ define (in $V[S_\alpha\ast G_\alpha]$) the poset $\mathbb{G}^{2,\beta}_\alpha:$
\begin{definition}\label{forcingatstage2}
  Conditions in $\mathbb{G}^{2,\beta}_\alpha$ are pairs $\langle t, T\rangle$ where $t\in [\alpha]^{<\omega}$ is a $2$-coherent sequence and $T$ is a tree on $[\alpha]^{<\omega}$ with the following properties:
    \begin{enumerate}
        \item $\varnothing\in T$ and every $s\in T$ is $2$-coherent;
        \item  $t^\smallfrown s$ is $2$-coherent  for all $s\in T$;
        \item for each node $s\in T$, $$\mathrm{Succ}_T(s)\in 
        \mathcal{U}^0_{\omega^2\cdot\beta+\omega\cdot|(t^\smallfrown s)_1|+|(t^\smallfrown s)_0|}\cap \mathcal{U}^1_{\omega\cdot\beta+|(t^\smallfrown s)_1|}((t^\smallfrown s)\restriction 1),$$
        where $(t^\smallfrown s)_i$ denotes the collection of ordinals in $t^\smallfrown s$ of $\ell$-order $i.$
        \end{enumerate}
\end{definition}

\begin{definition}
    Given $\langle t,T\rangle\in \mathbb{G}^{2,\beta}_\alpha$ and $\langle \nu\rangle\in T$ we denote by $\langle t,T\rangle\cat \langle\nu\rangle$ the pair $\langle t{}^\smallfrown\langle \nu\rangle, T_{\langle\nu\rangle}\rangle$ where $T_{\langle \nu\rangle}:=\{s\in [\alpha]^{<\omega}\mid \langle\nu\rangle{}^\smallfrown s\in T\}.$

   For general $\vec\nu\in T$ one defines $\langle t,T\rangle\cat\vec\nu$ by recursion in the obvious fashion.
\end{definition}

\begin{remark}
    Note that $\langle t{}^\smallfrown\langle \nu\rangle, T_{\langle\nu\rangle}\rangle$ is a condition in $\mathbb{G}^{2,\beta}_\alpha.$
\end{remark}

\begin{definition}\label{forcingorder}
     For $q=\langle t,T\rangle$ and $q^*=\langle t^*,T^*\rangle$ in $\mathbb{G}^{2,\beta}_\alpha$ we say that:
    \begin{enumerate}
       \item $q$ and $q^*$ are \emph{equivalent} if $t$ and $t^*$ are equivalent\footnote{See Definition~\ref{CharacteristicsOfT}.} and $T=T^*$.
        \item $q^*\leq^* q$ if $t=t^*$ and $T^*\s T$.
        \item $q^*\leq q$ if there is $\vec\nu\in T$ and $\bar{q}$ equivalent to $q\cat\vec\nu$  such that $q^*\leq^*\bar{q}$. 
    \end{enumerate}
  
\end{definition}

For each $2$-coherent sequence $t$ let us extend $\mathcal{U}^2_\beta$ to $\mathcal{U}^2_\beta(t)$. 
\begin{lemma}\label{lemmacoherency}
    $\mathscr{N}_{\alpha,\zeta^2_\beta}[\iota_{\alpha,\zeta^2_\beta}(S_\alpha)]\models``\text{$\alpha$ has $\iota_{\alpha, \zeta^2_\beta}(\ell)$-order $2$"}$.
    
    In particular,   $\iota_{\alpha,\zeta^2_\beta}(\mathbb{P}_\alpha)$ opts at stage $\alpha$ for the poset
    $\mathbb{G}^{2,\beta}_\alpha.$
\end{lemma}
\begin{proof}
Let us check the clauses of Definition~\ref{ellorder}.

\smallskip

  \textbf{Clause (1):}  By our assumption upon $\zeta^2_\beta$, we know that $\mathbb{P}_\alpha$ forces (over the model $\mathscr{N}_{\alpha,\zeta^2_\beta}[\iota_{\alpha,\zeta^2_\beta}(S_\alpha)]$) that $\vec{c}^{\;2}_\beta=\langle c_\gamma\mid \gamma\in [\omega^2\cdot \beta, \omega^2\cdot(\beta+1))\rangle$ is a sequence of $\alpha$-codes. 

  \smallskip
  
  \textbf{Clause 2(a):}  We have to show that $\mathscr{N}_{\alpha,\zeta^2_\beta}[\iota_{\alpha,\zeta^2_\beta}(S_\alpha)]$ thinks that, for each $\omega\cdot\beta\leq \gamma<\omega\cdot(\beta+1)$, 
   \begin{equation}\label{coherencyeq}
\tag{$\otimes$}\mathscr{N}_{\alpha,\zeta^1_\gamma}[\iota_{\alpha,\zeta^1_\gamma}(S_\alpha)]\models``\one\forces_{\mathbb{P}_\alpha}\text{$``\iota_{\alpha,{\zeta^1_\gamma}}(\ell)(\alpha)=\vec{c}^{\;1}_\gamma$ are $\alpha$-codes''}.
   \end{equation}
   Note that \eqref{coherencyeq} is true in $V[S_\alpha]$ by our assumption that $o^\ell(\alpha)=3.$ The question is whether or not \eqref{coherencyeq} also holds in the inner model $\mathscr{N}_{\alpha,\zeta^2_\beta}[\iota_{\alpha,\zeta^2_\beta}(S_\alpha)]$.

   \smallskip

Fix any of such $\gamma$'s. Then, $U(\alpha,\zeta^1_\gamma)\in \mathscr{N}_{\alpha,\zeta^2_\beta}$ by the very choice of $\zeta^2_\beta$. Recall that $\mathcal{U}^1_\gamma$ is the unique lifting of $U(\alpha,\zeta^1_\gamma)$ where value of the generic object $\iota^\alpha_{\zeta^1_\gamma}(\ell)$ at coordinate $\alpha$ is $\vec{c}^{\;1}_\gamma$. More pedantically,  a set $X\in\mathcal{P}(\alpha)^{V[S_\alpha]}$  belongs to $\mathcal{U}^1_\gamma$ if and only if  there is a condition $p\in S_\alpha$ such that 
$$\iota_{\alpha,\zeta^1_\gamma}(p)\cup \{\langle\alpha,\vec{c}^{\;1}_\gamma\rangle\}\forces_{\iota_{\alpha,\zeta^1_\gamma}(\mathbb{S}_\alpha)}\alpha\in \iota_{\alpha,\zeta^1_\beta}(\dot{X}).$$
Since this lifting is uniquely determined by $\iota_{\alpha,\zeta^1_\gamma}``S_\alpha$ and $\vec{c}^{\,1}_\gamma$, and  the model $\mathscr{N}_{\alpha,\zeta^2_\beta}[\iota_{\alpha,\zeta^2_\beta}(S_\alpha)]$ knows about this parameters, it follows that 
$$\mathcal{U}^1_\gamma=(\mathcal{U}^1_\gamma)^{\mathscr{N}_{\alpha,\zeta^2_\beta}[\iota_{\alpha,\zeta^2_\beta}(S_\alpha)]}\in \mathscr{N}_{\alpha,\zeta^2_\beta}[\iota_{\alpha,\zeta^2_\beta}(S_\alpha)].$$ 

Since $\mathscr{N}_{\alpha,\zeta^1_\gamma}[\iota_{\alpha,\zeta^1_\gamma}(S_\alpha)]$ is the ultrapower by $\mathcal{U}^1_\gamma$, this is computed the same way inside $\mathscr{N}_{\alpha,\zeta^2_\beta}[\iota_{\alpha,\zeta^2_\beta}(S_\alpha)]$. Finally, as \eqref{coherencyeq} is true in the universe $V[S_\alpha]$, it must be also  true in $\mathscr{N}_{\alpha,\zeta^2_\beta}[\iota_{\alpha,\zeta^2_\beta}(S_\alpha)]$, as needed.

\smallskip

\textbf{Clauses 2(b) and 2(c):} Immediate from equation~\eqref{notationcodes} on page~\pageref{notationcodes}.

\medskip

From all the above we infer that 
$$\mathscr{N}_{\alpha,\zeta^2_\beta}[\iota_{\alpha,\zeta^2_\beta}(S_\alpha)]\models``\text{$\alpha$ has $\iota_{\alpha, \zeta^2_\beta}(\ell)$-order $2$"},$$
as witnessed by 
$\iota_{\alpha, \zeta^2_\beta}(\ell)(\alpha)=\langle \vec{c}^{\,2}_\beta,\vec{\zeta}^{\,2}_\beta\rangle.$

In particular, $\iota_{\alpha, \zeta^2_\beta}(\mathbb{P}_\alpha)$ opts at stage $\alpha$ for the poset $\mathbb{G}^{2}_\alpha$ as computed in $\mathscr{N}_{\alpha,\zeta^2_\beta}[\iota_{\alpha,\zeta^2_\beta}(S_\alpha)]$. This is the gluing poset defined using:
\begin{enumerate}
    \item The measures $\langle \mathcal{U}^0_{\omega^2\cdot\beta+\omega\cdot m+n}\mid n,m<\omega\rangle$ decoded from $\vec{c}^{\,2}_\beta$.
    \item The ``$t$-liftings" of the measures $\langle \mathcal{U}^1_\gamma\mid \omega\cdot\beta\leq \gamma<\omega\cdot(\beta+1)\rangle$.
\end{enumerate}
Note that the ``$t$-liftings" of each $\mathcal{U}^1_\gamma$ are uniquely determined by $\iota_{\alpha,\zeta^1_\gamma}``G_\alpha$ and $t$, and since $\mathscr{N}_{\alpha,\zeta^2_\beta}[\iota_{\alpha,\zeta^2_\beta}(S_\alpha)\ast G_\alpha]$ knows about all of these 
$$\mathcal{U}^1_\gamma(t)=(\mathcal{U}^1_\gamma(t))^{\mathscr{N}_{\alpha,\zeta^2_\beta}[\iota_{\alpha,\zeta^2_\beta}(S_\alpha)\ast G_\alpha]}.$$
As a result, $\mathbb{G}^{2}_\alpha$ as computed in  $\mathscr{N}_{\alpha,\zeta^2_\beta}[\iota_{\alpha,\zeta^2_\beta}(S_\alpha)\ast G_\alpha]$ is exactly the poset  $\mathbb{G}^{2,\beta}_\alpha$ described in  Definition~\ref{forcingatstage2}.
\end{proof}
Let $t\in [\alpha]^{<\omega}$ be $2$-coherent. We lift $\mathcal{U}^2_\beta$ to $\mathcal{U}^2_\beta(t)$ as follows: 

\begin{definition}\label{ExtendingTheMeasure2}
Working in $V[S_\alpha\ast G_\alpha]$, let  $\mathcal{U}^2_\beta(t)$ be the collection of all $\dot{X}_{G_\alpha}\s \alpha$ for which there is $p\in G_\alpha$ and $\langle t, T\rangle\in {\mathbb{G}}^{2,\beta}_\alpha$ such that
    $$\mathscr{N}_{\alpha,\zeta^2_\beta}[\iota_{\alpha,\zeta^1_\beta}(S_\alpha)]\models ``p^\smallfrown \langle t , T\rangle^\smallfrown \iota_{\alpha, \zeta^2_\beta}(p)\setminus (\alpha+1)\forces_{\mathbb{P_\alpha}}\alpha\in \iota_{\alpha,\zeta^2_\beta}(\dot{X})".$$
\end{definition}
\begin{lemma}\label{lemma: concentarting on orderd 2}
In $V[S_\alpha\ast G_\alpha]$, $\mathcal{U}^2_\beta(t)$ is an $\alpha$-complete ultrafilter over $\alpha$   extending $\mathcal{U}^2_\beta$. In addition, $\mathcal{U}^2_\beta(t)\ni \{\nu<\alpha\mid o^\ell(\nu)=2\,\wedge\, {t}\sq b_\nu\}.$ 
\qed
\end{lemma}
The next is the analogue of Lemma~\ref{lemma: disjointifying levels}:
\begin{lemma}\label{lemma: least member of bnu enters A}
    Let  $A$ be a large set with respect to $\mathcal{U}^0_{\omega^2\cdot \beta+\omega\cdot m +n}$. Then, for each $2$-coherent sequence  $t$ with $|t_1|=m$ and $|t_0|=n$, 
    $$\{\nu<\kappa\mid \min(b_\nu)\setminus (\max(t)+1) \in A \}\in \mathcal{U}^2_\beta(t).$$
\end{lemma}
\begin{proof}
    The proof is analogue to that of Lemma~\ref{lemma: disjointifying levels}: Suppose the conclusion of the lemma was false and let $p,\dot{T}$ witnessing it. By definition, $$\mathrm{Succ}_T(\varnothing)\in \mathcal{U}^0_{\omega^2\cdot \beta+\omega\cdot m +n}\cap \mathcal{U}^1_{\omega\cdot \beta+m}(t\restriction 1).$$ The argument in Lemma~\ref{lemma: disjointifying levels} shows that
    $$A^*:=\{\nu<\kappa\mid \min(b_\nu)\setminus (\max(t)+1) \in A \}\in \mathcal{U}^1_{\omega\cdot \beta+m}(t\restriction 1).$$
(Note that, $t\restriction 1$ corresponds to the ordinals of $\ell$-order $0$ and thus $|t\restriction 1|=n$.)

\smallskip
    
    If $S_0\uplus S_1$ is a partition of $\mathrm{Succ}_T(\varnothing)$ according to the $\ell$-order, we can shrink $S_0$ (resp. $S_1$) so that $S_0\s A$ (resp. $S_1\s A^*$). Let $T^*$ be the corresponding shrinkage of $T$. It is clear that $p{}^\smallfrown \langle t, T^*\rangle{}^\smallfrown \iota_{\alpha,\zeta^2_\beta}(p)\setminus (\alpha+1)$ forces the opposite of what we originally assumed.
\end{proof}

\underline{$\ell$-order $3$:} The time is now ripe to define the intended gluing poset $\mathbb{G}^{3}_\alpha:$

\begin{definition}\label{forcingatstage3}
       Conditions in $\mathbb{G}^{3}_\alpha$ are pairs $\langle t, T\rangle$ where $t\in [\alpha]^{<\omega}$ is a $3$-coherent sequence and $T$ is a tree on $[\alpha]^{<\omega}$ with the following properties:
    \begin{enumerate}
        \item $\varnothing\in T$ and every $s\in T$ is $3$-coherent;
        \item  $t^\smallfrown s$ is $3$-coherent  for all $s\in T$;
        \item for each node $s\in T$, $\mathrm{Succ}_T(s)$ belongs to $$
        \mathcal{U}^0_{\omega^2\cdot|(t^\smallfrown s)_2|+\omega\cdot|(t^\smallfrown s)_1|+|(t^\smallfrown s)_0|}\cap\, \mathcal{U}^1_{\omega\cdot|(t^\smallfrown s)_2|+|(t^\smallfrown s)_1|}((t^\smallfrown s)\restriction 1)\cap\, \mathcal{U}^2_{|(t^\smallfrown s)_2|}((t^\smallfrown s)\restriction 2),$$
        where $(t^\smallfrown s)_i$ denotes the collection of ordinals in $t^\smallfrown s$ or $\ell$-order $i.$
    \end{enumerate}
    The order between conditions is defined  as in Definition~\ref{forcingorder}.
\end{definition}


\smallskip

The above completes the description of the gluing poset when $o^\ell(\alpha)=3$. Let us now briefly discuss the general case in which $o^\ell(\alpha)=\rho<\alpha$.

\smallskip

Suppose that we have already defined 
$$\langle \mathbb{G}^{\sigma,\beta}_\alpha, \mathcal{U}^\sigma_\beta(t)\mid \sigma<\rho\,\wedge\, \omega^\sigma\cdot\beta<\omega^\rho\,\wedge\, t\in [\alpha]^{<\omega}\, \text{is $\sigma$-coherent}\rangle$$
where 

\smallskip

   $\br$ $\mathcal{U}^0_\beta$ is the measure decoded from $c_\beta$ in $V[S_\alpha\ast G_\alpha]$;
  \smallskip
  
  $\br$ $\mathcal{U}^{\sigma}_\beta$ is the lift of $U(\alpha,\zeta^{\sigma}_\beta)$ to $V[S_\alpha]$ determined by $\iota_{\alpha,\zeta^\sigma_\beta}(\ell)(\alpha)=\langle\vec{c}^{\;\sigma}_\beta,\vec{\zeta}^\sigma_\beta\rangle.$
\begin{definition}
      Conditions in $\mathbb{G}^{\rho}_\alpha$ are pairs $\langle t, T\rangle$ where $t\in [\alpha]^{<\omega}$ is a $\rho$-coherent sequence and $T$ is a tree on $[\alpha]^{<\omega}$ such that:
    \begin{enumerate}
        \item $\varnothing\in T$ and every $s\in T$ is $\rho$-coherent;
        \item  $t^\smallfrown s$ is $\rho$-coherent for all $s\in T$;
        \item for each $s\in T$, $\mathrm{Succ}_T(s)\in \bigcap_{\sigma<\rho}\mathcal{U}^{\sigma}_{j_{\sigma}}((t^\smallfrown s)\restriction \sigma)$ 
       where $$j_\sigma:=\otp\{\nu\in c_{t^\smallfrown s}\mid o^\ell(\nu)=\sigma\}.$$ 
        
    \end{enumerate}
   The orders $\leq$ and $\leq^*$ are the same as in Definition~\ref{forcingorder}.
\end{definition}
\begin{remark}\label{remark:compatibility}
    Any two conditions of the form $\langle t, T\rangle, \langle t, S\rangle$ are $\leq^*$-compatible.
    Also, the analogues of Lemmas~\ref{lemma: concentarting on orderd 2} and \ref{lemma: least member of bnu enters A} do hold for $\mathbb{G}^\rho_\alpha$.
\end{remark}

The above completes the definition of the gluing poset for all possible $\ell$-orders. So, we are now in a position to present the gluing iteration:

\begin{definition}[Gluing Iteration]  Working in $V[S]$,  the \emph{gluing iteration}  is the non-stationary supported iteration  
$\mathbb{P}_\kappa=\varinjlim \langle \mathbb{P}_\alpha;\dot{\mathbb{Q}}_\alpha\mid \alpha< \kappa\rangle$ where
$$V[S_\alpha]\models ``\one\forces_{\mathbb{P}_\alpha}\dot{\mathbb{Q}}_\alpha=\dot{\mathbb{G}}_\alpha^{o^\ell(\alpha)}"$$
for each measurable cardinal $\alpha<\kappa$.
\end{definition}
In Theorem~\ref{Theoremlambdagluing} we show that if $o(\kappa)\leq \kappa$ is an uncountable regular cardinal then, after forcing over $V[S]$ with the gluing iteration, one gets a model $V[S][G]$ for the ${<}o(\kappa)$-$\mathsf{GP}$. One of the key points of this proof consists in ensuring that every $\kappa$-complete ultrafiler (over $\kappa$) in $V[S][G]$ admits a code in the sense of Definition~\ref{def: codes}. This will be proved in the  forthcoming \S\ref{sec:codinglemma}.

\smallskip

To simplify notations let us denote by  $\mathbb{G}_\alpha$ the gluing poset $\mathbb{G}_\alpha^{o^{\ell(\alpha)}}$. 

\begin{lemma}[Existence of Meets]\label{lemma: existence of meets} Suppose that $1\leq o^\ell(\alpha)<\alpha$. Then:
\begin{enumerate}
    \item If $p,q\in \mathbb{G}_\alpha$ are compatible conditions  then $$\text{$b_{\mathrm{stem}(p)}\sq b_{\mathrm{stem}(q)}$ or $b_{\mathrm{stem}(q)}\sq b_{\mathrm{stem}(p)}.$}$$
    \item There is a $\leq^*$-dense set $D$ of conditions in $\mathbb{G}_\alpha$ with the following property: For each $p, q \in D$, if $p$ and $q$ are compatible conditions and 
$b_{\stem(q)}\sq b_{\stem(p)}$ then $\stem(p)\setminus(\max(\stem(q))+1) \in T^q$ and $$p \wedge q = \langle \stem(p), T^p \cap (T^q)_{\stem(p)\setminus(\max(\stem(q))+1)}\rangle.$$ 
\end{enumerate}
\end{lemma}
\begin{proof}
(1). Let $p=\langle t, T\rangle$ and $q=\langle s, S\rangle$ be compatible conditions in $\mathbb{G}_\alpha$. By definition, there are finite sequences $\vec\nu\in T$ and $\vec\eta\in S$ such that ${t^\smallfrown \vec\nu}$ and ${s^\smallfrown \vec\eta}$ are equivalent – that is, $b_{t^\smallfrown \vec\nu}=b_{s^\smallfrown \vec\eta}$. Suppose, for instance, that $\max(\stem(t))\geq \max(\stem(s))$. Since $b_s\sq b_{s^\smallfrown \vec\eta}=b_{t^\smallfrown \vec\nu} $ and $b_{t}\sq b_{t^\smallfrown \vec\nu}$ it must be the case that $b_s \sq b_t.$ Thus, clause (1) follows.

\smallskip

(2). Let $D$ be the collection of conditions $\langle s, S\rangle\in \mathbb{G}_\alpha$ such that all points in the Prikry/Magidor sequence $b_{\nu}$ (for $\nu \in S$)  belong to the supplying tree $S$; namely, we want to look at $D=\{\langle s, S\rangle \in \mathbb{G}_\alpha\mid \forall  \nu \in \bigcup S\, (b_{\nu}\s \bigcup S)\}$. 
\begin{claim}
    $D$ is $\leq^*$-dense.
\end{claim}
\begin{proof}[Proof of claim]
    Let $\langle s, S\rangle $ be a  in $\mathbb{G}_\alpha$. We want to find $\langle s, S^*\rangle\leq^* \langle s, S\rangle$ such that $\bigcup S^*\s \{\nu\in \bigcup S\mid b_\nu\s \bigcup S\}$. To secure this property, we inductively shrink the levels of $S$ as follows: By definition, $\mathrm{Succ}_S(\varnothing)$ belongs to $\bigcap_{\sigma<o^\ell(\alpha)}\mathcal{U}^\sigma_{j_\sigma}(s\restriction\sigma).$ Since we are assuming that $o^\ell(\alpha)<\alpha$, this set admits a decomposition of the form $\bigcup_{\sigma<o^\ell(\alpha)} S_\sigma$ where each $S_\sigma$ lives in $\mathcal{U}^\sigma_{j_\sigma}(s\restriction\sigma)$ and concentrates on measurables  $\beta<\alpha$ with $o^\ell(\beta)=\sigma$. 
    
    Given $\sigma<o^\ell(\alpha)$ there are $p\in G$ and a $\mathbb{P}_\alpha$-name for a tree $\dot{R}$ such that
    $$p{}^\smallfrown \langle s\restriction\sigma, \dot{R}\rangle^\smallfrown \iota_{\alpha,\zeta^\sigma_{j_\sigma}}(p)\setminus (\alpha+1)\forces \alpha\in \iota_{\alpha,\zeta^\sigma_{j_\sigma}}(\dot{S}_\sigma).$$
    Work in $V[G]$. Denote  $S\restriction\sigma\s S$ the subtree of $S$ where the splitting sets are required to be in the intersection of the  ultrafilters $\langle \mathcal{U}^{\rho}_{j_\rho}(\ast)\mid \rho<\sigma\rangle$. Then, $$\langle s\restriction\sigma, \dot{R}_G\cap (S\restriction\sigma)\rangle\leq^*_{\mathbb{G}_\alpha^\sigma}\langle s\restriction\sigma, \dot{R}_G\rangle.$$ 
    
    Let $r\leq p$ be in $G$ forcing this. 
    It then follows that 
     $$r{}^\smallfrown \langle s\restriction\sigma, \dot{R}_G\cap (S\restriction\sigma)\rangle^\smallfrown \iota_{\alpha,\zeta^\sigma_{j_\sigma}}(r)\setminus (\alpha+1)\forces \alpha\in \iota_{\alpha,\zeta^\sigma_{j_\sigma}}(\{\nu\in \bigcup S\mid b_\nu\s \bigcup S\}),$$
     hence $\{\nu\in \bigcup S\mid b_\nu\s \bigcup S\}\in \mathcal{U}^\sigma_{j_\sigma}(s\restriction\sigma).$ Since $\sigma$ was arbitrarily chosen $$\textstyle
     \{\nu\in \bigcup S\mid b_\nu\s \bigcup S\}\in \bigcap_{\sigma<o^\ell(\alpha)}\mathcal{U}^\sigma_{j_\sigma}(s\restriction\sigma).$$
     
     Thus, we can shrink the original $\mathrm{Succ}_S(\varnothing)$ to this set and still get a $\bigcap_{\sigma<o^\ell(\alpha)}\mathcal{U}^\sigma_{j_\sigma}(s\restriction\sigma)$-large set. Carrying out this filtration  on all the levels of the tree we finally obtain the desired $S^*.$
\end{proof}
Let $p,q\in D$ be compatible conditions with $b_{\stem(q)}\sq b_{\stem(p)}$. 

\smallskip

$\br$ If $\max(\stem(p))=\max(\stem(q))$ then $\stem(p)\setminus (\max(\stem(q))+1)$ is empty and thus it belongs to $ T^q$. Hence it must be that $b_{\stem(q)}= b_{\stem(p)}$, so that $\langle \stem(p), T^p\cap T^q\rangle$ is the $\leq$-greatest extension of both $p$ and $q$.

\smallskip

$\br$ Suppose otherwise. Showing that  $\stem(p)\setminus (\max(\stem(q))+1)\in T^q$  would imply that $\langle \stem(p), T^p \cap (T^q)_{\stem(p)\setminus(\max(\stem(q))+1)}\rangle$ is the meet of $p$ and $q$. So, let $\langle \alpha_0,\dots, \alpha_n\rangle$ be the increasing enumeration of $\stem(p)\setminus (\max(\stem(q))+1)$. Since $q$ and $p$ are compatible, there is $\vec\eta\in T^q$ such that $\{ \alpha_0,\dots,\alpha_{n}\}\s b_{t}\sq b_{s\cat \vec\eta}=b_s\cup \bigcup_{i\leq |\vec\eta|}b_{\vec\eta(i)} $. For each $j\leq n$, there is $i<|\vec\eta|$ such that $\alpha_j\in b_{\vec\eta(i)}$. However, $q\in D$, which yields $\alpha_j\in b_{\vec\eta(i)}\s \bigcup T^q$. All in all, $\langle \alpha_0,\dots, \alpha_n\rangle\in T^q$, as needed.
\end{proof}

\begin{remark}
    In what follows we identify $\mathbb{G}_\alpha$ with the $\leq^*$-dense set identified in the previous lemma – that is, the one for which any two compatible conditions admit a $\leq$-greatest lower bound.
\end{remark}

\begin{lemma}[Properties of $\mathbb{G}_\alpha$]\label{lemma: properties of gluing poset}
Suppose $o^\ell(\alpha)<\alpha$. Then, the following  hold in the generic extension $V[S_\alpha\ast G_\alpha]$:
\begin{enumerate}
    \item $\mathbb{G}_\alpha$ is $\alpha^+$-cc.
    \item $\langle\mathbb{G}_\alpha,\leq^*\rangle$ is $\alpha$-closed. 
    \item {If $p,q\in\mathbb{G}_\alpha$ are compatible then they admit a meet, $p\wedge q$.}
    \item $\mathbb{G}_\alpha$ introduces a club $c\s \alpha$ of order-type $\omega^{o^\ell(\alpha)}$.
\end{enumerate}
\end{lemma}
\begin{proof}
(1) This follows from Remark~\ref{remark:compatibility}.

(2) If $\langle \langle t, T_\beta\rangle\mid \beta<\theta\rangle$ is a $\leq^*$-decreasing sequence with $\theta<\alpha$ then $\langle t, \bigcap_{\beta<\theta} T_\beta\rangle$ is a well-formed condition and incidentally a $\leq^*$-lower bound.

(3) This was the purpose of  Lemma~\ref{lemma: existence of meets}. 

(4) The result is clear for all measurable $\alpha$ with $o^\ell(\alpha)=1$. 

\smallskip

Let us argue the general result by induction. Suppose that
$$\forall \beta<\alpha\,\forall \sigma<o^\ell(\alpha)\; (o^\ell(\beta)=\sigma\;\rightarrow\;\text{$\mathbb{G}^\sigma_\beta$ adds a club of order-type $\omega^\sigma$ to $\beta$}).$$
 Let us assume that $o^\ell(\alpha)=\rho+1$. Notice that every condition $p$ extends to another condition $q=\langle t,T\rangle$ where all $\beta\in t$ have $o^\ell(\beta)=\rho$. In particular, if $G\s \mathbb{G}_\alpha^{\rho+1}$ is a generic filter,
$$\{\beta<\alpha\mid \exists \langle t, T\rangle\in G\, (\beta\in t\, \wedge\, o^\ell(\beta)=\rho)\}$$
is unbounded in $\alpha.$ 

Let $\langle \beta_n\mid n<\omega\rangle$ be the increasing enumeration of the above set. It follows, by the induction hypothesis, that each $\beta_n$ carries a club $c_{n}$ with order-type $\omega^\rho$. Since for each $n<\omega$ there is $\langle t, T\rangle\in G$ with $\beta_n,\beta_{n+1}\in t$,  coherency  yields that $\beta_{n}<\min(c_{\beta_{n+1}})$ and, consequently, that
$$\textstyle c=\bigcup_{n<\omega} c_{n}$$
is a club on $\alpha$. Clearly, $c$ has order-type $\omega^{\rho+1}$, as needed.

Assume that $o^\ell(\alpha) = \rho$ for a limit ordinal $\rho$. For each order $\sigma<\rho$ the set
$$D_{\sigma}:=\{\langle \langle \delta\rangle, S\rangle\in \mathbb{G}^\rho_\alpha\mid\exists \sigma^*<\rho\, (\sigma<\sigma_*\,\wedge\, o^\ell(\delta)=\sigma_*)\}$$
is dense in $\mathbb{G}^\rho_\alpha$.
Thus, for each $\sigma<\rho$ there is a condition  $\langle \langle \delta_\sigma\rangle,  S_\sigma\rangle\in G\cap D_\sigma$. In particular, $\{\langle \langle \delta_\sigma\rangle,  S_\sigma\rangle\mid \sigma<\rho\}$ is a set of compatible conditions in $\mathbb{G}^{\rho}_\alpha$. Therefore, it must be the case that either  $b_{\delta_{\sigma}}$ is an initial segment of $b_{\delta_{\sigma'}}$ or vice-versa (by Lemma~\ref{lemma: existence of meets}). By induction hypothesis, each of these $b_{\delta_\sigma}$ is a club on $\delta_\sigma$ with order-type in $[\omega^{\sigma},\omega^{\rho})$. Therefore, $\bigcup_{\sigma<\rho} b_{\delta_\sigma}$ is a club on $\alpha$ with order-type $\omega^{\rho}$. This completes the proof.
\end{proof}

The following lemma can be proved as the forthcoming Lemma \ref{lemma: SPP for Motis}.
\begin{lemma}
$\langle \mathbb{G}_\alpha,\leq,\leq^*\rangle$ is a Prikry-type forcing. 
\qed
\end{lemma}

\section{Characterizing ultrafilters in the gluing extension }\label{sec:codinglemma}
For the rest of this section we work under the following assumptions:
\begin{equation}\label{BlanketAssumption}
 \tag{$\mathscr{H}$}   \text{$``\rm{V}=\mathcal{K}$'' + ``There is no inner model for ${\exists \alpha\, (o(\alpha)\geq \alpha)}$''.}\footnote{In some technical lemmas in this section, the weaker anti–large cardinal assumption
“There is no inner model satisfying $\exists \alpha\,(o(\alpha)=\alpha^{++})$”
will suffice. The stronger assumption \eqref{BlanketAssumption} is needed to ensure that the Gluing Iteration $\mathbb{P}_\kappa$ is well defined.}
\end{equation}
 Let $S\s \mathbb{S}$ be generic over $\mathcal{K}$ and $\mathbb{P}_\kappa=\varinjlim\langle  \mathbb{P}_{\alpha}, \dot{\mathbb{G}}_{\beta}\mid \alpha<\beta < \kappa\rangle$ be the gluing iteration of \S\ref{sec: the gluing iteration}. Since the GCH holds in $\mathcal{K}[S]$ (see \cite[Proposition~6.2]{HP}) it follows that $|\mathbb{P}_\alpha|\leq 2^\alpha$. Also, by Lemma~\ref{lemma: properties of gluing poset}, the gluing poset $\langle \mathbb{G}_\alpha,\leq^*\rangle$ is $\alpha$-closed and {every two compatible  conditions  $p,q$ admit a meet $p\wedge q.$}

\smallskip

Extending the work of \cite{HP} in this section we prove the \textbf{Coding Lemma} for the gluing iteration $\mathbb{P}_\kappa$. This lemma provides a complete cartography of the $\kappa$-complete ultrafilters in the generic extension $\mathcal{K}[S\ast G]$ and it becomes critical for the future verification of the gluing property (Theorem~\ref{Theoremlambdagluing}).
    \begin{lemma}[Coding Lemma]\label{lemma: coding lemma} 
Let  $G\s \mathbb{P}_\kappa$ be a generic filter over $\mathcal{K}[S]$. Then,  there are $\kappa^{+}$-many $\kappa$-complete ultrafilters in $\mathcal{K}[S][G]$. Moreover, every such ultrafilter $U\in \mathcal{K}[S][G]$ admits a code in the sense that there are
\begin{enumerate}
  \item[$(\alpha)$] a finite sub-iteration $\iota\colon \mathcal{K}[S] \to \mathcal{K}^M[\iota(S)]$ of $j_U\restriction\mathcal{K}[S]$, 
     \item[$(\beta)$] an ordinal $\bar{\epsilon} < \iota(\kappa)$ with $\bar{\epsilon}\in \range(k)$, 
\item[$(\gamma)$] $r \in \iota(\mathbb{P}_\kappa)$ with finite support such that for every $p\in G$, $\iota(p)\wedge r$ exists,
\end{enumerate}
which together have the following property: Working in $\mathcal{K}[S]$, for each $p\in G$, $p\forces_{\mathbb{P}_\kappa}``\dot{X}\in \dot{{U}}$'' if and only if there is a condition $q\in \iota(\mathbb{P}_\kappa)$  such that
\[(k(q)\in j_{{U}}(G),\; q\leq^ *\iota(p) \wedge r,\;\, \supp(q)=\supp(\iota(p)\wedge r),\;\, q\Vdash_{\iota(\mathbb{P}_\kappa)}  \bar{\epsilon} \in \iota (\dot{X})).\]
\end{lemma}

\[
\begin{tikzcd}
     \mathcal{K}[S] \arrow[r, "{j_{\mathcal{U}}\restriction\mathcal{K}[S]}"] \arrow[d, "\iota"'] 
     & M_{{U}}[j_{{U}}(S)]  \\
    \mathcal{K}^{M_{\mathcal{U}}}[\iota(S)] \arrow[ru, "{k}"] & 
\end{tikzcd}
\]
\begin{remark}
 In the terminology of  Definition~\ref{def: codes}, the Coding Lemma says that the $5$-tuple $\langle \vec{\rho}, \vec{\zeta},\vec{a}, f, g\rangle$ defined as follows is a code for $\dot{U}$ in $\mathcal{K}[S]$. 

 First, $\iota^0:=j_U\restriction\mathcal{K}$ is a finite iteration of measures in $\mathcal{K}$ and as such it is determined by a  collection of functions  $\vec\rho=\langle \rho_0,\dots, \rho_{n-1}\rangle$ and ordinals  $\vec{\zeta}=\langle \zeta_0,\dots, \zeta_{n-1}\rangle$ as described in Definition~\ref{def: codes}. This yields, in particular, the collection of critical points $\langle \mu_0,\dots, \mu_{n-1}\rangle$ of $\iota^0$ by stipulating $\mu_i=\iota^0_i(\rho_i)(\langle \mu_j\mid j<i\rangle).$ By Lemma~\ref{lem;lifting-S}, $\iota\colon \mathcal{K}[S]\rightarrow \mathcal{K}^{M}[\iota(S)]$ is the unique lifting of $\iota^0$ determined by $S$ and a vector $\langle \bar a_i,\dots, \bar a_{n-1}\rangle\in \prod_{i<n}H(\mu^+_i)^{\mathcal{K}_i}$. These parameters themselves can be coded via $\vec{a}$ as  $\bar a_i=\iota^0_i(a_i)(\langle \mu_j\mid j<i\rangle).$ Finally, $\bar\epsilon$ and $r$ can be respectively presented as $\iota^0(f)(\langle \mu_i\mid i<n\rangle)$ and $\iota^0(g)(\langle \mu_i\mid i<n\rangle)$ for functions $f\colon \kappa^n\rightarrow \mathbb{P}_\kappa$ and $g\colon \kappa^n\rightarrow\kappa$ in $\mathcal{K}[S]$.
\end{remark}

\subsection{Preliminary comments and technical lemmas}

\begin{setup}\label{setupcoding}
Our ground model $V$ during this section is $\mathcal{K}[S]$, a generic extension by the fast function forcing $\mathbb{S}$. Similarly, $\mathbb{P}_\kappa$ denotes the gluing iteration (defined in $\mathcal{K}[S]$)  as described  in \S\ref{sec: the gluing iteration}.  We fix a $V$-generic filter $G\s\mathbb{P}_\kappa$  and a  ultrapower embedding  $j:=j_{U}\colon  V[G]\to M[H]$  via a $\kappa$-complete (non-necessarily normal) ultrafilter  on $\kappa$ in $V[G].$
\end{setup}

In light of our anti-large-cardinal assumptions (\eqref{BlanketAssumption} on page~\pageref{BlanketAssumption}), $j\restriction\mathcal{K}$ is a normal iteration and as a result so  is $j\restriction V$ (by Lemma~~\ref{lem;lifting-S}) – alas, notice that these iterations might possibly be infinite. By Lemma~\ref{lem;representing-elements-in-direct-limit}, for each $x\in M$, there is a finite normal iteration $\iota\colon V\rightarrow N$ such that $k\circ \iota=j\restriction V$ and $x\in \mathrm{range}(k)$, where $k\colon N\rightarrow M$ is the factor map between $\iota$ and $j$ – to wit, $k$ is the inverse of the Mostowski collapse of the structure with universe $$\{j(f)(\mu_{0},\dots, \mu_{n-1})\mid f\colon \kappa^n\to V\},$$
where $\bar\mu_0=k^{-1}(\mu_{0}),\dots,\bar\mu_{n-1}=k^{-1}(\mu_{n-1})$ are the critical points of the iteration $\iota.$ In these circumstances we say that the iteration \emph{$\iota$ factors $j\restriction V$}.

The above comments suggest  a natural ordering between iterations that factor $j\restriction V$, according to the $k$-images of their critical points.
\begin{definition}
   Let if $\iota\colon V\rightarrow N$ and $\bar{\iota}\colon V\rightarrow \bar{N}$ be iterations factoring $j\restriction V$ as witnessed by embeddings $k$ and $\bar{k}$ respectively.  Then, we write $$\iota\leq_j\bar{\iota}\,:\Leftrightarrow k(\crit(\iota))\s \bar{k}(\crit(\bar{\iota})).\label{eq: ordering between iterations}$$
Here $\crit(\iota)$ and $\crit(\bar\iota)$ denote the critical points of $\iota$ and $\bar\iota$, respectively.  
\end{definition}

\begin{remark}
    $\leq_j$ is transitive and reflexive. In the forthcoming Lemma~\ref{lem;maximal-stem} we will need to appeal to some sort of 'density' in the $\leq_j$-ordering. 
\end{remark}

\begin{definition}
  Given a finite iteration $\iota\colon V\rightarrow N$ that factors $j\restriction V$ through another elementary embedding $k\colon N\rightarrow M$, we define:
  \begin{enumerate}
      \item $\label{gammaiota}
\Gamma_\iota := \{\iota_{n,m}(\mu_n) \mid n \leq m\leq \ell\},$
\item $\mathcal{C}_{\iota} := \{k(\mu_i) \mid i < n\},$
  \end{enumerate}
where $\iota:= \langle \iota_{n,m} \mid n \leq m \leq \ell \rangle$ and, for each $n<\ell$, $\mu_n:=\crit(\iota_{n,n+1})$. 
\end{definition}


   The set $\Gamma_\iota$ played a crucial role in the analysis of \cite[\S6]{HP} leading to the \textbf{Coding Lemma} – it does it here as well. The key idea is: the only coordinates where we may detect discrepancies between conditions $r \in \iota(\mathbb{P}_\kappa)$, that are sent to $H$ via $k$, and conditions in $j``G$ are among the elements of $\Gamma_\iota$. The set of ordinals where these  discrepancies arise, denoted $A_\iota$ (see Definition~\ref{def:Aiota}), will be shown to be contained in $\Gamma_\iota$ (Lemma~\ref{AsubsetGamma}). 

   It will also be important for later purposes to characterize those 
sets $\mathcal{C}\s \crit(j)$ for which there is $\iota\colon V\rightarrow N$ so that $\mathcal{C} = \mathcal{C}_{\iota}$. Similar characterizations to the one given here appeared in work of Gitik–Kaplan \cite{GitikKaplan-nonstationary2022}.
\begin{lemma}\label{lemma: characterizing Ciota}
Let $\mathcal{C}\s \crit(j)$ be finite. Then, the following are equivalent:
\begin{enumerate}
    \item There is a normal finite iteration $\iota$ factoring  $j\restriction V$ such that $\mathcal{C} = \mathcal{C}_{\iota}$.
    \item For each $\mu \in \mathcal{C}$, if $\mathcal{U}$ is the normal measure which was used to obtain the critical point $\mu$ in the iteration $j$ (so $j = j_{end} \circ j_{\mathcal{U}} \circ j_{begin}$), then there is $n<\omega$ and a function $g\colon [\kappa]^n\rightarrow V$ in $V$ such that $$\mathcal{U} \in j_{begin}(g)([\mathcal{C} \cap \mu]^n).$$ 
\end{enumerate}
\end{lemma}
\begin{proof}

   $(1)\Rightarrow (2).$ Suppose $\mathcal{C} = \mathcal{C}_\iota$ for some finite normal iteration $\iota\colon V\rightarrow N$ factoring $j\restriction V$. Let $k\colon N\rightarrow M$ be the  factor embedding and $\langle\mu_0,\dots, \mu_{n-1}\rangle=\langle k(\bar\mu_0),\dots, k(\bar\mu_{n-1})\rangle$  be the increasing enumeration of $\mathcal{C}$.

   Let $\mathcal{U}_{\alpha_0}$ be the measure used by  the iteration $j\restriction V$ to hit the critical point $\mu_0$. Since $\mu_0=k(\bar\mu_0)$ it is routine to show that the diagram
             \[
\begin{tikzcd}[column sep=large, row sep=large] V \arrow[r, "j_{\alpha_0}"] \arrow[d, "\iota_1"] & M_{\alpha_0} \arrow[r, "j_{\alpha_0,\alpha_0+1}"] & M_{\alpha_0+1} \arrow[r, "j_{\alpha_0+1, \Omega}"] & M \\ N_1 \arrow[rru,  "k_1"] \arrow[rrr, "\iota_{1,n}"'] & & & N \arrow[u, "k"]  \end{tikzcd}
\]   
commutes after stipulating that  $k_1(\iota_1(f)(\bar\mu_0)):=j_{\alpha_0+1}(f)(\mu_0)$.

\smallskip

Notice that $\crit(\iota_1)=\crit(j)=\kappa$. Thus, $\bar\mu_0=\kappa$ and $k_1(\kappa)=\mu_0$.

Let $\mathcal{V}_1$ be the measure whose ultrapower embedding is $\iota_1$. Given $X\s \kappa$, $$X\in \mathcal{V}_1\;\Leftrightarrow\; \kappa\in \iota_1(X)\;\Leftrightarrow\;\mu\in j_{\alpha_0+1}(X)\;\Leftrightarrow\; j_{\alpha_0}(X)\in \mathcal{U}_{\alpha_0}.$$
This shows that  $j_{\alpha_0}``\mathcal{V}_1$ is included in the measure $\mathcal{U}_{\alpha_0}$.

 Note that this implies that $j_{\alpha_0}(\kappa)=\mu$, by uniformity of $\mathcal{U}_{\alpha_0}$. 

We claim that the filter generated by $j_{\alpha_0}``\mathcal{V}_1$ in $M_{\alpha_0}$, modulo $\s^*$, (call it $\mathcal{F}$)\footnote{Namely, $\mathcal{F}=\{X\in \mathcal{P}^{M_{\alpha_0}}(j_{\alpha_0}(\kappa))\mid \exists Y\in \mathcal{V}_1\, (j_{\alpha_0}(Y)\s^* X)\}.$} is in fact an $M_{\alpha_0}$-ultrafilter. This shows that $j_{\alpha_0}(\mathcal{V}_1)=\mathcal{F}=\mathcal{U}_{\alpha_0}$, thus confirming the validity of our claim for the first measure under discussion. 

\smallskip

Let $X\in \mathcal{P}^{M_{\alpha_0}}(j_{\alpha_0}(\kappa))$. Then, $X=j_{\alpha_0}(f)(\mu^{\alpha_0}_0,\dots, \mu^{\alpha_0}_{k})$ for an increasing sequence of critical points of $j$ below $j_{\alpha_0}(\kappa)$. Without loss of generality, $f\colon [\kappa]^{k}\rightarrow \mathcal{P}(\kappa)$. For each $s\in [\kappa]^k$, let $A_{s}$ be $f(s)$ if this latter set is $\mathcal{V}_1$-large; otherwise, set $A_s:=\kappa\setminus f(s)$. By normality of $\mathcal{V}_1$,
$$Y:=\{\alpha<\kappa\mid \forall s\in [\kappa]^k\, (\max(s)<\alpha\,\Rightarrow\, \alpha\in f(s))\}\in \mathcal{V}_1.$$
It is easy to check that $j_{\alpha_0}(Y)\setminus (\mu^{\alpha_0}_k+1)\s j_{\alpha_0}(A)_{\langle \mu^{\alpha_0}_0,\dots, \mu^{\alpha_0}_k\rangle}$.

Since $j_{\alpha_0}(Y)\in \mathcal{F}$  we conclude that  $j_{\alpha_0}(A)_{\langle \mu^{\alpha_0}_0,\dots, \mu^{\alpha_0}_k\rangle}$ is in $ \mathcal{F}$ as well. Then either $X\in\mathcal{F}$ or $j_{\alpha_0}(\kappa)\setminus X\in\mathcal{F}$, as needed.

\smallskip

Let us show that something similar occurs with $\mathcal{U}_{\alpha_i}$, the measure used by $j$ to produce $\mu_i$. To prevent cumbersome notations we concentrate on $i=1$, and since the argument is pretty similar to the previous  we just sketch it.

Once again, we have a commutative diagram of elementary embeddings
\[
\begin{tikzcd}[column sep=1.8em, row sep=large, font=\small]  V \arrow[r, "j_{\alpha_0}"] \arrow[d, "\iota_1"] & M_{\alpha_0} \arrow[r, "j_{\alpha_0,\alpha_0+1}"] & M_{\alpha_0+1} \arrow[r, "j_{\alpha_0+1, \alpha_1}"] & M_{\alpha_1} \arrow[r, "j_{\alpha_1, \alpha_1+1}"] & M_{\alpha_1+1} \arrow[r, "j_{\alpha_1+1, \Omega}"] & M \\  N_1 \arrow[rru, "k_1"] \arrow[r, "\iota_{1,2}"'] & N_2 \arrow[to=1-5, "k_2"] \arrow[rrrr, "\iota_{2,n}"'] & & & & N \arrow[u, "k"] \end{tikzcd}
\]
after stipulating $k_2(\iota_2(f)(\bar\mu_0,\bar\mu_1)):=j_{\alpha_1+1}(f)(\mu_0,\mu_1)$.

\smallskip

We want to show that $j_{\alpha_0+1,\alpha_1}(k_1(\mathcal{V}_2))=\mathcal{U}_{\alpha_1}$. This suffices to verify our claim because $\mathcal{V}_2$ can be expressed as
$\mathcal{V}_2=\iota_1(g)(\bar\mu_0)$ and as a result
$$\mathcal{U}_{\alpha_1}=j_{\alpha_0+1,\alpha_1}(j_{\alpha_0+1}(g)(\mu_0))=j_{\alpha_1}(g)(\mu_0).$$

Arguing exactly as before one shows that $(j_{\alpha_0+1,\alpha_1}\circ k_1)``\mathcal{V}_2\s \mathcal{U}_{\alpha_1}.$ In particular, $\mu_{1}=j_{\alpha_0+1,\alpha_1}(k_1(\bar\mu_1))$. We show $(j_{\alpha_0+1,\alpha_1}\circ k_1)\mathcal{V}_2=\mathcal{U}_{\alpha_1}.$

Let $X\s \mathcal{P}^{M_{\alpha_1}}(\mu_1)$ be arbitrary. Then,
$X=j_{\alpha_0+1,\alpha_1}(f)(\mu^{\alpha_1}_0,\dots, \mu^{\alpha_1}_k)$
where $f\colon [k_1(\bar\mu_1)]^k\rightarrow \mathcal{P}^{M_{\alpha_0+1}}(k_1(\bar \mu_1))$ is in $M_{\alpha_0+1}$. For each $s\in [k_1(\bar\mu_1)]^k$ let $A_s$ be either $f(s)$ or its complement, depending on which of them is $k_1(\mathcal{V}_2)$-large. We define $\mathcal{F}$ and $Y$  as before replacing $\kappa$ by $k_1(\bar\mu_1)$. Then, we note that $j_{\alpha_0+1,\alpha_1}(Y)\setminus (\mu^{\alpha_1}_k+1)\s j_{\alpha_0+1,\alpha_1}(A)_{\langle \mu^{\alpha_1}_0,\dots, \mu^{\alpha_1}_k\rangle}$ and conclude that either $X\in \mathcal{F}$ or its complement is in $\mathcal{F}$.

    \smallskip

    $(2) \Rightarrow (1).$ Let $\langle \mu_0,\dots, \mu_{n-1}\rangle$ be the increasing enumeration of  $\mathcal{C}$.

   Borrowing the notations from the previous proof,  our assumption gives $\mathcal{U}_{\alpha_0}=j_{\alpha_0}(\mathcal{V}_1)$ where $\mathcal{V}_1$ is a normal measure on $\kappa$ in $V$. In particular, $\mu_0=j_{\alpha_0}(\kappa)$. Let $\iota_1\colon V\rightarrow N_1\simeq \Ult(V,\mathcal{V}_1)$ and define $k_1\colon N_1\rightarrow M_{\alpha_0+1}$ by 
    $$k_1(\iota_1(f)(\kappa)):=j_{\alpha_0+1}(f)(j_{\alpha_0}(\kappa)).$$
    This map is easily seen to be elementary.

   By our assumption,  $\mathcal{U}_{\alpha_1}$ can be presented as $j_{\alpha_1}(f)(\mu_0)$. Hence,
   $$\mathcal{U}_{\alpha_1}=j_{\alpha_0+1,\alpha_1}(j_{\alpha_0}(f)(\mu_0))=j_{\alpha_0+1,\alpha_1}(k_1(\iota_1(g)(\kappa))).$$
   Let $\mathcal{V}_2\in N_1$ be the unique normal measure such that $j_{\alpha_0+1,\alpha_1}(k_1(\mathcal{V}_2))=\mathcal{U}_{\alpha_1}$. In particular the critical point $\bar\mu_1$ of $\mathcal{V}_2$ satisfies $j_{\alpha_0+1,\alpha_1}(k_1(\bar\mu_1))=\mu_1$.

   Let $\iota_{1,2}\colon N_1\rightarrow N_2\simeq \Ult(N_1, \mathcal{V}_2)$ and define a map $k_2\colon N_2\rightarrow M_{\alpha_1+1}$ by $k_2(\iota_2(f)(\bar\mu_0,\bar\mu_1)):=j_{\alpha_1+1}(f)(\mu_0,\mu_1)$. Clearly, $k_2$ is elementary.

   Proceeding in this fashion we produce a diagram like the one on  previous pages being $\iota\colon V\rightarrow N$ the composition of the various ultrapowers $\iota_i\colon N_i\rightarrow N_{i+1}$ and $k\colon N\rightarrow M$ the factor embedding defined by $$k(\iota(f)(\bar\mu_0,\dots, \bar\mu_{n-1})):=j(f)(\mu_0,\dots, \mu_{n-1}).$$
   By construction, it is clear that $\mathcal{C}=\mathcal{C}_\iota$.
\end{proof}


For later purposes it will be important  to specify where a condition strictly extends another. This motivates the following general definition:
 \begin{definition}
Let $\mathbb{A}_\kappa=\langle \mathbb{A}_{\alpha}, \mathbb{B}_{\beta} \mid \beta < \alpha < \kappa\rangle$ be an iteration of forcings. For $p, q \in \mathbb{A}_\kappa$ and $\Gamma \subseteq \kappa$, write $q \leq_\Gamma p$ if $q \leq p$ and 
\[\{\alpha<\kappa \mid q\restriction\alpha \not\Vdash_{\mathbb{A}_\alpha} \dot{p}(\alpha) = \dot{q}(\alpha)\} \subseteq \Gamma.\]
\end{definition}

In our case $\mathbb{P}_\kappa$ is the non-stationary-supported iteration of the gluing posets. Thus, a typical condition $p\in \mathbb{P}_\kappa$ is a sequence  $\langle p(\alpha)\mid \alpha\in \supp(p)\rangle$ where $p(\alpha)$ is a $\mathbb{P}_\alpha$-name for a condition  $\langle \dot{t}_\alpha, \dot{T}_\alpha\rangle$ and $\supp(p)$ is a nowhere stationary set.  Since $G_\alpha:=G\cap\mathbb{P}_\alpha$ is $V$-generic we can interpret  $\dot{t}_\alpha$ and $\dot{T}_\alpha$, and this interpretation yields a \emph{stem} $t_\alpha$ and a tree $T_\alpha$. Thus,  given $p\in\mathbb{P}_\kappa$ and $\alpha\in \supp(p)$, we will denote $\stem(p(\alpha))_{G_\alpha}:=(\dot{t}_\alpha)_{G_\alpha}.$ 

Notice that for $p,q\in\mathbb{P}_\kappa$, if $q \leq_\Gamma p$ then $\supp(q) \subseteq \supp(p) \cup \Gamma$.

\smallskip

The following technical lemma will become handy later on:
\begin{lemma}\label{lemma:densityconditions}
    The following collection of conditions $p\in \mathbb{P}_\kappa$ is $\leq^*$-dense in $\mathbb{P}_\kappa$: For each $\beta \in \supp(p)$ there is $\bar\beta<\beta$ and a $\mathbb{P}_{\bar\beta}$-name $\sigma$ so that $$p\restriction\beta\forces_{\mathbb{P}_\beta}\stem(p(\beta))=\sigma.$$
\end{lemma}
\begin{proof}
By induction on $\kappa$. The idea is to combine the standard fusion argument with the reduction of dense open sets  described in Lemma~\ref{lemma: reducingdensesets}. 

\smallskip

Our induction hypothesis reads as follows:
\[
(\mathrm{IH}_\kappa) \quad \text{Assume $\gamma + 1 <\kappa$ and $p\in \mathbb{P}_\kappa$ there is $q\leq^* p$ with $q\restriction\gamma = p\restriction\gamma$ }
\]
\[
\text{such that for each $\beta \in \supp(q) \setminus \gamma$ there is $ \bar\beta < \beta$ and a  $\mathbb{P}_{\bar \beta}$-name $\sigma$}
\]
\[
q\restriction\beta \forces_{\mathbb{P}_\beta} \stem(q(\beta)) = \sigma.
\]

\smallskip


Suppose that $(\mathrm{IH}_{\bar\kappa})$ holds for all $\bar\kappa<\kappa$. We distinguish three cases.

\medskip

\underline{\textbf{Successor case:}} Suppose that $\kappa=\bar\kappa+1$. Let $\gamma\leq \bar\kappa$ and $p\in \mathbb{P}_\kappa.$

\smallskip

$\br$ Suppose $\gamma=\bar\kappa$. Then we have nothing to do. 

\smallskip

$\br$ 
Assume $\gamma<\bar\kappa$. If $\gamma+1=\bar\kappa$ then  the forcing at stage $\bar\kappa$ is trivial – in this case it suffices to take $q=p$. Otherwise, $\gamma + 1 < \bar\kappa$ and  we apply the inductive hypothesis to $p\restriction\bar\kappa$ thus obtaining $\bar q\leq^* p\restriction\bar\kappa$ such that $\bar q\restriction\gamma= p\restriction\gamma$ and $\bar q$ witnessing the property displayed in $(\mathrm{IH}_{\bar\kappa})$. Once we have disposed with ordinals $\beta<\bar\kappa$ we handle the  case $\beta=\bar\kappa$. This is accomplished by applying Corollary~\ref{cor: meeting-dense-open-bounded} to condition $\bar q^\smallfrown\dot{p}(\bar\kappa)$. Thus, we obtain $q\leq^* \bar q^\smallfrown\dot{p}(\bar\kappa)$ such that $q\restriction\bar\kappa =\bar q$ and a $\mathbb{P}_\beta$-name $\sigma$ (with $\beta<\bar\kappa$) such that  $\bar{q}\forces_{\mathbb{P}_{\bar\kappa}}``\sigma=\stem(\dot{p}(\bar\kappa))".$ The sought condition is $\bar{q}^\smallfrown \dot{p}(\bar\kappa)$.


\medskip

\textbf{\underline{Singular  case:}} Let $\gamma<\kappa$ and $p\in \mathbb{P}_\kappa$. Fix $\langle \kappa_\xi\mid \xi<\cf(\kappa)\rangle$ a cofinal, increasing, continuous sequence in $\kappa$ such that $\gamma, \cf(\kappa)<\kappa_0.$ Applying the induction hypothesis $(\mathrm{IH})_{\kappa_0}$, we find $p_0\leq^* p\restriction \kappa_0$ such that $p_0\restriction\gamma=p\restriction \gamma$ and $p_0$ witnesses the property described in $(\mathrm{IH})_{\kappa_0}$. Set $p^*_0:=p_0{}^\smallfrown p\restriction [\kappa_0,\kappa)$.  

Proceeding in this fashion, we construct a $\leq^*$-decreasing sequence $$\langle p^*_\xi\mid \xi<\cf(\kappa)\rangle$$
of conditions in $\mathbb{P}_\kappa$ such that $p^*_\xi\restriction \kappa_0=p^*_0\restriction \kappa_0$ for all $\xi<\cf(\kappa)$. The key reason for why this sequence can be constructed is that the closure degree of the $\leq^*$-order of $\mathbb{P}_\kappa\setminus \kappa_0$ is forced to be $\cf(\kappa)^{+}$-closed. (As  $\cf(\kappa)<\kappa_0$.)

Let $p^*$ be a $\leq^*$-lower bound for the above-displayed sequence. It is clear that $p^*\leq^* p$, $p^*\restriction \gamma=p\restriction\gamma$ and that $p^*$ witnesses the property in $(\mathrm{IH}_\kappa)$.



\medskip

\textbf{\underline{Regular case:}} We tackle this case utilizing the standard fusion-type arguments presented in \cite{BenUng}. Fix an ordinal $\gamma<\kappa$ and a condition $p\in \mathbb{P}_\kappa$.

Let $C_0\s \kappa$ be a club disjoint from $\supp(p)$ and let $\gamma_0:=\min(C_0\setminus \gamma+1).$ Invoking $(\mathrm{IH}_{\gamma_0})$ we find $p_0\leq^* p\restriction \gamma_0$ such that $p_0\restriction (\gamma+1)=p\restriction (\gamma+1)$ such that for each $\beta\in \supp(p\restriction\gamma_0)\setminus (\gamma+1)$ the name for the stem $\stem(\dot{p}(\beta))$ can be reduced to some $\bar\beta<\beta$. Finally, set $p^*_0:=p_0{}^\smallfrown p\restriction [\gamma_0,\kappa).$

\smallskip

By induction on $\eta\leq \kappa$, suppose we have defined a sequence $$\langle (p^*_\xi, \gamma_\xi, C_\xi)\mid \xi<\eta\rangle$$
for which the following properties hold, for each $\sigma<\xi<\eta$:

\begin{enumerate}
    \item  $p^*_\xi \restriction \gamma_\sigma + 1 = p^*_\sigma \restriction \gamma_\sigma + 1$.
     \item $\supp(p^*_\xi)\cap C_\xi= \emptyset$.
    \item $\gamma_\xi \in \bigcap_{\bar\xi <\xi} C_{\bar\xi}$.
    \item For all $\beta \in \supp(p^*_\xi) \cap (\gamma,\gamma_\xi)$, there is $\bar{\beta} < \beta$ and $\tau\in V^{\mathbb{P}_{\bar\beta}}$ such that
    $$p^*_\xi\restriction\beta\forces_{\mathbb{P}_\beta}``\stem(p^*_\xi(\beta))=\tau".$$
\end{enumerate}
At successor stages $\xi+1$ we proceed exactly as before: Let $C_{\xi+1}\s C_\xi$ be disjoint from $\supp(p^*_\xi)$ and set $\gamma_{\xi+1}:=\sup(C_{\xi+1}\setminus \gamma_\xi+1).$ Apply $(\mathrm{IH}_{\gamma_{\xi+1}})$ and obtain $p_{\xi+1}\leq^* p_{\xi}\restriction\gamma_{\xi+1}$ such that $p_{\xi+1}\restriction \gamma_\xi+1=p_\xi\restriction \gamma_\xi+1$ and property (4) above holds. Finally, set $p^*_{\xi+1}:=p_{\xi+1}{}^\smallfrown p\restriction [\gamma_{\xi+1}, \kappa)$. (Note that $\supp(p^*_{\xi+1})$ is nowhere stationary – so, this is a legitimate condition in $\mathbb{P}_\kappa$.)

\smallskip

At limit stages $\eta<\kappa$ we take $C_\eta$ to be the limit points of the club $\bigcap_{\xi<\eta} C_{\xi}$ and set $\gamma_\eta:=\sup_{\xi<\eta}\gamma_\xi$. In particular, $\gamma_\eta\in C_\eta$ and $\gamma_\eta\notin \bigcup_{\xi<\eta}\supp(p^*_\xi).$

We amalgamate the conditions constructed thus far as follows. First, set
$$\textstyle p^*_\eta\restriction \gamma_\eta+1:=(\bigcup_{\xi<\eta} p^*_\xi\restriction\gamma_\xi)^\smallfrown \one_{\dot{\mathbb{Q}}_{\gamma_\eta}}.$$
Since the map $\xi\mapsto \gamma_\xi$ is normal, it follows that $\eta\leq \gamma_\eta$. Therefore, 
$$p^*_\eta\restriction \gamma_\eta+1\forces_{\mathbb{P}_{\gamma_\eta}+1}\text{$``\mathbb{P}_\kappa/\dot{G}_{\gamma_\eta+1}$ is $(\gamma_\eta)^+$-closed"},$$
which allows us to take a name $\dot{r}$ for a $\leq^*$-lower bound for the $\leq^*$-decreasing sequence $\langle p^*_\xi\setminus \gamma_\eta+1\mid \xi<\eta\rangle$. Finally, set
$p^*_\eta:=(p^*_\eta\restriction\gamma_\eta+1)^\smallfrown \dot{r}.$ A moment's reflection makes clear that this is a legitimate condition in $\mathbb{P}_\kappa.$

\smallskip

Finally, if $\eta = \kappa$, we proceed as before. The only difference is that we now define $C_\kappa$ as the set of limit points of $\diagonal_{\xi < \kappa} C_\xi$. The sequence $p^*_\kappa$, constructed analogously, forms a well-qualified condition—i.e., it has nowhere stationary support—as witnessed by the club $C_\kappa$. Clearly, $p^*_\kappa$ is a $\leq^*$-extension of $p$  witnessing the statement of $(\mathrm{IH}_\kappa)$
\end{proof}

\subsection{Proving the coding lemma}
In this section we show how to prove the \textbf{Coding Lemma} conditioned to some technical lemmas whose proof  will be deferred to the next section. Our proof closely follows  \cite[\S6]{HP} so we recommend our readers to have a copy of \cite{HP}  at hand.

As in the simpler case of \cite{HP}, we claim that we can code any $\kappa$-complete ultrafilter on $\kappa$ using finitely many ``stems" in the finite iteration of measures. Observe that in the gluing forcing there are different stems that represent the same information. Thus, in this context when comparing stems, we will imaging a stem $s$ extending another stem $t$ if and only if $b_s \sq b_t$. Note that this information might depend on the generic information from previous steps of the iteration, so we must exercise caution with it.
\smallskip

Let $\iota\colon V\rightarrow N$ be a finite normal iteration factoring $j\restriction V$. In analogy to the analysis of  \cite[\S6]{HP}, we will be preoccupied with the coordinates $\alpha\in [\kappa,\iota(\kappa))$ for which there is  $r\in \iota(\mathbb{P}_\kappa)$ mapped into the $M$-generic $H$ and the stem of $k(r)(k(\alpha))$ is larger than the stems of $j(p)(k(\alpha))$, for $p\in G.$

\begin{definition}\label{def:Aiota}
   Let $\iota\colon V\rightarrow N$ be a finite iteration factoring the embedding $j\restriction V.$ 
 We shall denote by $A_\iota$ the set of ordinals defined as follows: 
\[\{\alpha\in [\kappa,\iota(\kappa)) \mid \exists r\in k^{-1}(H)\,\forall p\in G\, (k(r(\alpha))_{H\restriction k(\alpha)}<j(p)(k(\alpha))_{H\restriction k(\alpha)})\}.\]
(Here $p<q$ stands for $p\leq q$ and $b_{\mathrm{stem}(q)} \sqsubset b_{\mathrm{stem}(q)}$.)
\end{definition}

As in \cite[Lemma~6.17]{HP} we obtain the following result:
\begin{lemma}\label{AsubsetGamma}
Then $A_\iota \subseteq \Gamma_\iota$. In particular, $A_\iota$ is finite.\footnote{For the definition of $\Gamma_\iota$ see page~\pageref{gammaiota}.} \qed 
\end{lemma}

In order to prove the \textbf{Coding Lemma} we need to prove the analogue of \cite[Lemma~6.15]{HP} in the current, more complicated, scenario:

\begin{lemma}[Key claim]\label{lem;maximal-stem}
Suppose that $\iota\colon V \to N$ is a finite iteration  factoring $j\restriction V$. Then, $\iota$ admits  a $\leq_j$-extension to a finite iteration $\iota\colon V\to N$ which factors $j\restriction V$ via $k\colon N\rightarrow M$ and such  that for each $\bar\mu\in A_\iota$ the set
\[\bar{H}^{\iota, k}_{\bar\mu} = \{b_{\mathrm{stem}(k(r(\bar\mu))_{H\restriction k(\bar\mu)})} \mid \text{$r$ witnesses $\bar\mu\in A_\iota$}\}\]
has a $\sq$-maximal element.
\end{lemma}
The proof of this will be deferred to the next section. Assuming this is true, next we describe how to establish the \textbf{Coding Lemma}.

\smallskip

By Lemma~\ref{lem;maximal-stem} we can associate to $\leq_j$-densely-many {finite iterations $\iota\colon V\rightarrow N$} factoring $j\restriction V$ a finite list of names for stems $\vec{\eta}:=\langle \dot{\eta}_\alpha\mid \alpha\in A_\iota\rangle$ where $b_{k(\dot{\eta}_\alpha)_{H\restriction k(\alpha)}}$ is the $\sq$-maximal element of $\bar{H}_\alpha$.  We may assume that $\dot{\eta}_\alpha$ is forced by the weakest condition of $\iota(\mathbb{P}_\kappa)_\alpha$ to be a stem at coordinate $\alpha$. \begin{definition}\label{definitionriota}
By $r_\iota$ we denote the condition with $\supp(r_\iota)= A_\iota$ such that $r_\iota (\alpha)$ is a $\iota(\mathbb{P}_\kappa)_\alpha$-name for a maximal tree with stem $\dot{\eta}_\alpha$.
\end{definition}
\begin{remark}
Note that $r_\iota\in \iota(\mathbb{P}_\kappa)$ is a condition such that $k(r_\iota)\in H$. Indeed, working in $M[H]$, $k(r_\iota)$ is the unique condition with (finite) support $k(A_\iota)$ such that for each $\alpha\in A_\iota$, $k(r_\iota)(k(\alpha))=\langle k(\dot{\eta}_\alpha)_{H\restriction k(\alpha)},k(\dot{T}_\alpha)_{H\restriction k(\alpha)}\rangle$. Let $p\in H$ be a condition forcing this. Then, it must be that $p\leq k(r_\iota)$ for otherwise $p$ will not be able to decide that piece of the generic sequence induced by $H$. By upwards closure of $H$ we deduce that $k(r_\iota)\in H.$
\end{remark}
The next results are Lemmas~6.18, 6.19 of \cite{HP}. Their proofs are abstract enough so as to apply to our current,  more complicated, iteration:
\begin{lemma}\label{lem;conditions-in-k-inv-H}
Let $\iota\colon V\rightarrow N$ be a finite iteration factoring $j\restriction V$ via $k$. 
For every $q \in k^{-1}(H)$, there is $p \in G$ and $q'\leq q$ such that  $q' \leq^*_{\Gamma_\iota} 
\iota(p) \wedge r_\iota$. 

In particular, for every $\alpha \notin \Gamma_{\iota}$, $(\iota(p) \wedge r_\iota) \restriction \alpha \Vdash \iota(p)(\alpha) \leq q(\alpha)$.\qed
\end{lemma}
\begin{lemma}\label{MainLemma2}
For each $\epsilon < j(\kappa)$ 
 there is a finite iteration $\iota \colon V \to N$ factoring $j\restriction V$ such that  $\epsilon\in \range(k)$ 
and for which the following holds: 

For every $\mathbb{P}_\kappa$-nice name $\dot{X}$ for a subset of $\kappa$,$$\text{$M[H] \models \epsilon \in j(\dot{X})_H\; \Longleftrightarrow\; \exists q \in \iota(\mathbb{P}_\kappa)\, (k(q) \in H\, \wedge\, q \Vdash_{\iota(\mathbb{P_\kappa})}  k^{-1}(\epsilon) \in \iota(\dot{X}))$}.$$

Moreover, letting $r_\iota$ be the condition of Definition~\ref{definitionriota}, we may take $q$ so that, for some $p\in G$, $q\leq^*_{\Gamma_\iota}\iota(p) \wedge r_\iota$ and $\supp q= \supp (\iota(p) \wedge r_\iota)$. \qed 
\end{lemma} 
Using the above, the  \textbf{Coding Lemma} follows. The proof is the same as in \cite{HP} but (as it is brief enough) we reproduce it for completeness:
\begin{proof}[Proof of the \textbf{Coding Lemma},  Lemma~\ref{lemma: coding lemma}]
Let $\mathcal{U}\in V[G]$ be a $\kappa$-complete ultrafilter on $\kappa$ and let $j\colon V[G]\rightarrow M[H]$ be the induced ultrapower. 
By our setup considerations on page~\pageref{setupcoding} there is a finite iteration $\iota\colon V\rightarrow N$ such that $k\circ\iota=j\restriction V$ and $\epsilon:=[\id]_{\mathcal{U}}$ belongs to $\range(k)$. Invoke Lemma~\ref{MainLemma2} with respect to  these inputs and define $\mathcal{V}$ to be the set of all $X\in \mathcal{P}(\kappa)^{V[G]}$ such that there are $q\in k^{-1}(H)$ and $p\in G$ such that 
 $$q\leq^*_{\Gamma_\iota}(\iota(p)\wedge r_\iota)\;\text{ and }\; q\forces_{\iota(\mathbb{P}_\kappa)}k^{-1}(\epsilon)\in \iota(\dot{X}).$$
Note that any two conditions $q, q'$ witnessing the above admit an explicit $\leq^*_{\Gamma_\iota}$-extension. In particular, $\mathcal{V}$ is a filter. Also, by the moreover part of Lemma~\ref{MainLemma2}, $\mathcal{V}$ contains the $\kappa$-complete ultrafilter $\mathcal{U}$. Thus, $\mathcal{V}=\mathcal{U}.$ \label{proofofcoding}
\end{proof}

Therefore, all what is missing to rigorously prove the \textbf{Coding Lemma} is Lemma~\ref{lem;maximal-stem}, which regards  the existence of maximal stems. 

\subsection{Proving the Key claim}
This section, of eminently technical nature,  proves  Lemma~\ref{lem;maximal-stem} (a.k.a., \textbf{Key claim}). This will be accomplished after showing that the critical points of  any  finite normal iteration $i\colon V\rightarrow N$ factoring $j\restriction V$ remain regular in the extension $M[H]$ (see Lemma~\ref{lem:mu-remains-regular}). 

In order to prove the above, we will need to analyze where the images of $\mu\in \Gamma_\iota$ go after applying the factor embedding $k\colon N\rightarrow M$. One potential hurdle to this analysis is that  $k\colon N\rightarrow M$ may not be internal to $N$. Fortunately, under our anti-large cardinal hypothesis, any normal iteration  $j\colon V\rightarrow M$ extending to an elementary embedding $j\colon V[G]\rightarrow M[H]$, with $M[H]^{\omega}\cap V[G]\s M[H]$,  can be completed to a somewhat universal iteration $\tilde{k}\colon N\rightarrow \tilde{M},$ which is \textbf{internally definable} in $N$. This idea  was first considered in Gitik's work \cite{Gitik1988Ordering}, and variants of it play an instrumental role in obtaining lower bounds for consistency results.

\smallskip

  Recall that $j\colon V[G]\rightarrow M[H]$ is the ultrapower by a $\kappa$-complete ultrafilter in $V[G]$ (see \textbf{Setup} on page~\pageref{setupcoding}), so  $M[H]$ is a generic extension of $M$ that is $\kappa$-closed in $V[G]$. Until the moment of proving the \textbf{Key claim} we shall work under the following anti-large cardinal assumption, weaker than the one subsumed in \eqref{BlanketAssumption} on page~\pageref{BlanketAssumption}:
  \begin{equation}\label{weakerantilargecardinals}
 \tag{$\mathscr{H}^-$}   \text{ ``There is no inner model of ${\exists \alpha\, (o(\alpha)=\alpha^{++})}$''.}
\end{equation}

\begin{lemma}[Universal iteration]\label{lemma: universal iteration} \label{claim ktilde}
Assume \eqref{weakerantilargecardinals} holds. Let  $\iota\colon \mathcal{K}\rightarrow \mathcal{K}^N$  be a finite normal iteration factoring $j\restriction \mathcal{K}$ and  $k\colon \mathcal{K}^N\to \mathcal{K}^M$ be  a  factor elementary embedding. Then,  there is an \textbf{internal}, normal iteration $\tilde{k}\colon \mathcal{K}^N\rightarrow \tilde{\mathcal{K}}$, with ordinal length, and an elementary embedding $\ell\colon \mathcal{K}^M \to \tilde{\mathcal{K}}$ making the following diagram commute:
        \[
\begin{tikzcd}
    \mathcal{K} \arrow[r, "{j}"] \arrow[d, "\iota"'] 
     & \mathcal{K}^M \arrow[d, "\ell"'] \\
   \mathcal{K}^N \arrow[ru, "{k}"]\arrow[r, "\tilde{k}"'] &  \tilde{\mathcal{K}}.
\end{tikzcd}
\]   
\end{lemma}

\begin{proof}
    Let $\mathcal{U}$ be the coherent sequence of measures in $\mathcal{K}^N$ and $\lambda:=(\sup k(\Gamma_{\iota}))^+$.

         We define the iteration $\tilde{k}$ in a \emph{Mitchell-style fashion}, internally to $\mathcal K^N$.
         
         We shall denote by $\tilde{M}_{\alpha}$ the iterates of  $\mathcal{K}^{N}$ leading to  $\tilde{\mathcal{K}}$. 
        
        Suppose by induction that the subiteration of $\tilde{k}$, $$\langle \tilde{k}_{\alpha,\beta}\colon \tilde{M}_\alpha\to \tilde{M}_\beta\mid \alpha\leq \beta<\Omega\in \Ord\cup\{\Ord\}\rangle$$ has been already defined. As usual, we stipulate $\tilde{k}_{0}:=\id$ and $\tilde{M}_0:=\mathcal{K}^N$. 
        
        Denote by $\tilde{k}_\alpha(\mathcal{U})(\tilde{\kappa}_\alpha,\tilde{\zeta}_\alpha)$ the measure in $\tilde{M}_\alpha$ used to define $\tilde{k}_{\alpha,\alpha+1}.$ If $\beta<\Omega$ is a limit ordinal then, for each $\alpha<\beta$, we let $\tilde{k}_{\alpha,\beta}$ the direct limit of $\langle \tilde{k}_{\alpha,\xi}\mid \alpha\leq \xi<\beta\rangle$.  Otherwise,  $\beta$  is of the form $\alpha+1$ and we have two cases:

        \smallskip
      
      ${(\aleph)}$: \label{eq: aleph of the iteration}If $\cf^{\mathcal{K}^N}(\alpha)\neq \omega$ 
      then we let (if it exists) $\tilde{\kappa}_\alpha < \tilde{k}_\alpha(\lambda)$  the least $\tilde{M}_\alpha$-measurable  in $[\sup_{\xi<\alpha}(\tilde{\kappa}_\xi+1), \tilde{k}_\alpha(\lambda))$ and we iterate $\tilde{k}_\alpha(\mathcal{U})(\tilde{\kappa}_\alpha,0).$

      \smallskip

      ${(\beth)}$: \label{eq: beth of the iteration} If $\cf^{\mathcal{K}^N}(\alpha)=\omega$ then, if there is a $\tilde{M}_\alpha$-measurable cardinal $\rho$ in $[\sup_{\xi<\alpha}\tilde{\kappa}_\xi, \tilde{k}_\alpha(\lambda))$ and a Mitchell order $\zeta<o^{\tilde{M}_\alpha}(\rho)$ for which the set
       $$\{\xi < \alpha \mid \tilde{k}_{\xi,\alpha}(\tilde\kappa_\xi) = \rho \text{ and } \tilde{k}_{\xi,\alpha}(\tilde\zeta_\xi) = \zeta\}$$
       is bounded in $\alpha$, we let $\tilde{\kappa}_\alpha$ the first $\tilde{M}_\alpha$-measurable in  $[\sup_{\xi<\alpha}\tilde{\kappa}_\xi,\tilde{k}_\alpha(\lambda))$ and ${\zeta}_\alpha<o^{\tilde{M}_\alpha}(\rho)$ the first order witnessing this. 
  Then, we iterate $\tilde{k}_\alpha(\mathcal{U})(\tilde{\kappa}_\alpha,\tilde{\zeta}_\alpha)$.\footnote{If $(\beth)$ holds because $\{\xi < \alpha \mid \tilde{k}_{\xi,\alpha}(\tilde\kappa_\xi) = \rho\}$ is bounded for some $\rho$, then $\tilde{\zeta}_\alpha=0.$} 

  \smallskip

  $(\gimel)$: If $(\aleph)$ nor $(\beth)$ are fulfilled,  we stipulate that the iteration halts.


        \begin{claim}\label{subclaim:Mitchells}
            The iteration halts after less than  $(\lambda^+)^{\mathcal K^N}$-many steps. 
        \end{claim}
        \begin{proof}[Proof of claim]
         This is a classical argument due to Mitchell which we reproduce both for completeness and for the benefit of our readers. 
         
         Work in $\mathcal K^N$. Suppose that the iteration does not stop after $(\lambda^+)^{\mathcal{K}^N}$-many steps. Thus, we have formed the sequence $\langle (\tilde{\kappa}_\xi, \tilde{\zeta}_\xi)\mid \xi<\lambda^+\rangle.$ Let $S$ denote the ordinals ${<}\lambda^+$ of $\mathcal{K}^N$-cofinality $\omega$. Certainly, this set is stationary.

         For each $\alpha\in S$ there is $\beta_\alpha<\alpha$ such $\tilde{\kappa}_\alpha,\tilde{\zeta}_\alpha\in \range(\tilde{k}_{\beta_\alpha,\alpha}).$ By Fodor's lemma, we find $S_0\s S$ stationary and $\beta^*<\lambda^+$ such that $\beta_\alpha=\beta^*$ for all $\alpha\in S_0$. Similarly, every ordinal ${<}\tilde{k}_{\beta^*}(\lambda)$ in $M_{\beta^*}$ is of the form $\tilde{k}_\beta(f)(\tilde{\kappa}_{\alpha_0}, \dots, \tilde{\kappa}_{\alpha_{n-1}})$ for some $f\colon \tilde{\kappa}_0^n\rightarrow \lambda$ in $\mathcal{K}^N$. Since $\mathcal{K}^N\models \mathrm{GCH}$ and $\lambda$ is a regular  above $\tilde{\kappa}_0$, there are only $\lambda$-many of such functions and consequently at most $\lambda$-many possible representations for ordinals below $\tilde{k}_{\beta^*}(\lambda).$
         
          So, we find another stationary set $S_1\s S_0$  and ordinals $\rho, \zeta$ such that for each $\alpha\in S_1$, $\tilde{\kappa}_\alpha=\tilde{k}_{\beta^*,\alpha}(\rho)$ and $\tilde{\zeta}_\alpha=\tilde{k}_{\beta^*,\alpha}(\zeta)$.  Therefore,
         $$\text{$\tilde{k}_{\alpha,\alpha'}(\tilde{\kappa}_\alpha)=\tilde{\kappa}_{\alpha'}$ and $\tilde{k}_{\alpha,\alpha'}(\tilde{\zeta}_\alpha)=\tilde{\zeta}_{\alpha'}$ for all $\alpha<\alpha'$ in $S_1$.}$$

         This is going to give a contradiction: Indeed, let $\alpha_\omega$ be the $\omega$th-member of $S_1$. Since $\cf^{\mathcal{K}^N}(\alpha_\omega)=\omega$ we were supposed to hit a measure which was not the image of $\tilde{k}_{\alpha_n}(\mathcal{U})(\tilde{\kappa}_{\alpha_n}, \tilde{\zeta}_{\alpha_n})$ for $n<\omega$, but the above contradicts this.   
        \end{proof}




 The above yields $\Omega\in \Ord$, thus  determining completely the iteration $\tilde{k}$.
 By maximality of $\Omega$, it should be clear that  $\cf^{\mathcal{K}^N}(\Omega)=\omega$.

 \begin{claim}\label{claim: things work as expected}
      For each $\alpha<\Omega$ and a $\tilde{M}_\alpha$-measurable cardinal $\rho{\geq \tilde{\kappa}_\alpha}$, there is an stage of the iteration $\tilde{k}$,  ${\alpha_*} \in (\alpha, \Omega]$, for which the following hold:
      $$
      \begin{cases}
        \tilde{k}_{\alpha,\beta}(\rho)\geq \tilde{\kappa}_\beta, & \text{if $\beta<{\alpha_*}$};\\
       \tilde{k}_{\alpha,\beta}(\rho)<\tilde{\kappa}_\beta, & \text{if $\beta\in [{\alpha_*},\Omega)$.}\footnote{Note that $\alpha_*$ in fact depends on $\rho$ as well, but we suppress this dependence to avoid cluttered notation. We also allow $\alpha_* = \Omega$, in which case the second case may be vacuous.
}  
      \end{cases}
      $$  
      Moreover, $\cf^{\mathcal{K}^N}({\alpha_*}) = \omega$ and for every $\zeta < o^{\tilde{M}_{\alpha_*}}(\tilde{k}_{\alpha,{\alpha_*}}(\rho))$ the set 
      $$\{\gamma<{\alpha_*}\mid\gamma\geq \alpha\,\wedge\, \tilde{k}_{\gamma,{\alpha_*}}(\tilde{\kappa}_\gamma)=\tilde{k}_{\alpha,{\alpha_*}}(\rho)\,\wedge\, \tilde{k}_{\gamma,{\alpha_*}}(\tilde{\zeta}_\gamma)=\zeta\}$$
      is unbounded in ${\alpha_*}$.
  \end{claim}
  \begin{proof}
 Let $\alpha<\Omega$ and $\rho$ be a  $\tilde{M}_\alpha$-measurable cardinal. 
 
In case it exists, let ${\alpha_*}$ be the least ordinal in $(\alpha,\Omega)$ for which $\tilde{k}_{\alpha,{\alpha_*}}(\rho) < \tilde{\kappa}_{{\alpha_*}}$. Otherwise, if that ordinal does not exist, we stipulate that  ${\alpha_*}:=\Omega$. 

  If ${\alpha_*}=\Omega$ then the claim is valid by our very choice of ${\alpha_*}$. So, let us assume that ${\alpha_*} < \Omega$. Let $\beta<\Omega$. If $\beta<{\alpha_*}$ then, clearly, $\tilde{k}_{\alpha,\beta}(\rho)\geq \tilde{\kappa}_\alpha.$ 
  Otherwise, $\beta\geq {\alpha_*}$ and then we  verify,  by induction on $\beta\in [{\alpha_*},\Omega)$,  that 
  \[\tilde{k}_{\alpha,{\alpha_*}}(\rho) = \tilde{k}_{\alpha,\beta}(\rho) < \tilde{\kappa}_\beta.\]
 For $\beta = {\alpha_*}$ the above is true by the very definition of ${\alpha_*}$. For a successor ordinal $\beta = {\bar\beta} + 1$, we combine the induction hypothesis (i.e., $\tilde{k}_{\alpha,{\alpha_*}}(\rho)=\tilde{k}_{\alpha,{\bar\beta}}(\rho)<\tilde{\kappa}_{\bar\beta}$) with the fact that $\tilde{k}_{\alpha,{\bar\beta}}(\rho) = \tilde{k}_{\alpha,\beta}(\rho)$ (as  $\crit(\tilde{k}_{{\bar\beta},\beta})=\tilde{\kappa}_{\bar\beta}$). 
 For limit steps, we use  the fact that the critical points $\tilde\kappa_{\bar\beta}$'s are increasing. 

 \smallskip

  Let us now address the second part of our claim.

  $\br$ For the sake of a contradiction, let us suppose that $\cf^{\mathcal{K}^N}({\alpha_*})\neq\omega$. Then, it must be that ${\alpha_*}<\Omega$. Since the iteration {has not halted,}  $\tilde{\kappa}_{{\alpha_*}}$ is the first $\tilde{M}_{{\alpha_*}}$-measurable in the interval $[\sup_{\xi<{\alpha_*}}(\tilde{\kappa}_\xi+1),\tilde{k}_{{\alpha_*}}(\lambda))$. {From our choice of ${\alpha_*}$ and elementarity of $k_{\alpha,{\alpha_*}}$, }$\tilde{k}_{\alpha,{\alpha_*}}(\rho)$ is a $\tilde{M}_{{\alpha_*}}$-measurable in that interval, ergo $\tilde{\kappa}_{{\alpha_*}}\leq \tilde{k}_{\alpha,{\alpha_*}}(\rho)$, contradicting the  definition of ${\alpha_*}$. 

  $\br$ Suppose  that for some Mitchell order $\zeta<o^{\tilde{M}_\beta}(\tilde{k}_{\alpha,{\alpha_*}}(\rho))$
   $$\{\gamma<{\alpha_*}\mid\gamma\geq \alpha\,\wedge\, \tilde{k}_{\gamma,{\alpha_*}}(\tilde{\kappa}_\gamma)=\tilde{k}_{\alpha,{\alpha_*}}(\rho)\,\wedge\, \tilde{k}_{\gamma,{\alpha_*}}(\tilde{\zeta}_\gamma)=\zeta\}$$
   was bounded in ${\alpha_*}$.  Since $\tilde{k}_{\alpha,{\alpha_*}}(\rho)$ is in $[\sup_{\xi < {\alpha_*}} \tilde\kappa_{\xi} , \tilde{k}_{{\alpha_*}}(\lambda))$ and the iteration abides in this case by Clause~$(\beth)$, we have, by minimality of $\tilde{\kappa}_{{\alpha_*}}$, that   $\tilde\kappa_{{\alpha_*}} \leq \tilde{k}_{\beta,{\alpha_*}}(\rho)$ --- a contradiction to the definition of ${\alpha_*}$.
  
\end{proof}
 
Let us record a couple of facts which, although we will not use them directly, provide some insight into the structure of the iteration $\tilde{k}$:
\begin{claim}[Some useful facts about $\tilde{k}$]\label{subclaim: Mitchell}
\hfill
    \begin{enumerate}
        \item For ordinals $\beta < \gamma < \Omega$, such that $\tilde{\kappa}_\gamma < \tilde{k}_{\beta,\gamma}(\tilde\kappa_\beta)$, we have $\gamma_* < {\beta_*}$ where $\gamma_*, {\beta_*}$ are as in Claim~\ref{claim: things work as expected} taking $\rho_\beta:=\tilde{\kappa}_\beta$ and  $\rho_\gamma:=\tilde{\kappa}_\gamma$.
        \item For an ordinal $\alpha<\Omega$, $\tilde{\zeta}_\alpha$ is $0$ unless $\cf^{\mathcal{K}^N}(\alpha)=\omega$ and the set $A=\{\beta<\alpha\mid \tilde{k}_{\beta,\alpha}(\tilde{\kappa}_\beta)=\tilde{\kappa}_\alpha\}$ is bounded in $\alpha$, in which case 
         \[\tilde\zeta_\alpha = \limsup (\{\tilde{k}_{\beta,\alpha}(\tilde\zeta_\beta) + 1\mid \beta\in A\}).\]
    \end{enumerate}
\end{claim}
\begin{proof}[Proof of claim]
 (1).  On the one hand, one can easily conclude that $\gamma_* \leq {\beta_*}$ by noting that for each $\alpha \in [\gamma, \gamma_*)$, \[\tilde\kappa_\alpha \leq \tilde{k}_{\gamma,\alpha}(\tilde\kappa_\gamma) < {\tilde{k}_{\gamma,\alpha} (\tilde{k}_{\beta,\gamma} (\tilde\kappa_\beta)}) = \tilde{k}_{\beta,\alpha}(\tilde\kappa_\beta),\] by elementarity of $\tilde{k}_{\gamma,\alpha}$.   On the other hand, in order to argue that equality is impossible, let us assume towards a contradiction that $\gamma_*= {\alpha_*}$ and note that the same computation shows that for every $\alpha \in [\gamma,\gamma_*)$, 
 \[\tilde{k}_{\alpha,{\beta_*}}(\tilde\kappa_\alpha) < \tilde{k}_{\alpha,{\beta_*}}(\tilde{k}_{\beta,\alpha}(\tilde\kappa_\beta)) = \tilde{k}_{\beta,{\beta_*}}(\tilde\kappa_\beta)\] 
 and in particular - not equal to it. This is impossible by Claim~\ref{claim: things work as expected}. 

 \smallskip

 (2). Since the iteration does not halt at stage $\alpha$, $\tilde\zeta_\alpha$ is the least Mitchell order $\zeta<o^{\tilde{M}_\alpha}(\tilde{\kappa}_\alpha)$ for which the following set is bounded: 
     \[B:=\{\beta \in A \mid \tilde{k}_{\beta,\alpha}(\tilde{\zeta_\beta})=\zeta\}.\]

     Let $\beta < \alpha$ be the minimal ordinal containing  $B$ and for which the Mitchell order $\tilde\zeta_\alpha$ is in $\range (\tilde{k}_{\beta,\alpha})$. 
     Let $\zeta$ be the pre-image of $\tilde{\zeta}_\alpha$ under  $\tilde{k}_{\beta,\alpha}$.
     
     Let us claim that for every $\gamma \in A \setminus (\beta+1)$, we have that $\tilde{k}_{\gamma,\alpha}(\tilde{\zeta}_\gamma) < \tilde{\zeta}_\alpha$. Assume that this is not the case. Let $\gamma\in A\setminus (\beta+1)$ be the least such that
     \begin{equation}\label{eq: Mitchell}
\tag{*}\tilde{k}_{\gamma,\alpha}(\tilde{\zeta}_\gamma)\geq \tilde{\zeta}_\alpha. 
     \end{equation}

    Since $\gamma\notin B$, necessarily $\tilde{k}_{\gamma,\alpha}(\tilde{\zeta}_\gamma)>\tilde{\zeta}_\alpha=\tilde{k}_{\beta,\alpha}(\bar\zeta)$. Therefore, $$\tilde{\zeta}_\gamma>\tilde{k}_{\beta,\gamma}(\bar\zeta).$$ 
    By commutativity of the iteration and since $\gamma\in A$ we have that 
     \[A \cap \gamma = \{\beta < \gamma \mid \tilde{k}_{\beta,\gamma}(\tilde\kappa_\beta) = \tilde\kappa_\gamma\}.\]
      $A\cap \gamma$ has to be unbounded in $\gamma$ for otherwise, by the induction hypothesis,  $\tilde{\zeta}_\gamma$ would be $0$. By the induction hypothesis, we conclude that  \[\tilde\zeta_\gamma = \limsup(\{ \tilde{k}_{\bar\gamma,\gamma}(\tilde\zeta_{\bar\gamma})+ 1\mid \bar\gamma \in A \cap \gamma\}).\]
     In particular, there are unboundedly many $\bar{\gamma} < \gamma$ such that $$\tilde{k}_{\bar\gamma,\gamma}(\tilde{\zeta}_{\bar\gamma}) \geq \tilde{k}_{\beta,\gamma}(\bar\zeta).$$ Every such ordinal $\bar\gamma$ must be a limit ordinal (otherwise, $\tilde{\zeta}_{\bar\gamma}=0$), so $\gamma$ is a limit of limit ordinals, and cofinally many of them are larger than $\beta$. In particular  $\tilde{\zeta}_{\bar\gamma} \geq \tilde{k}_{\beta,\bar\gamma}(\bar\zeta)$ for cofinally many $\beta<\gamma$. However, this poses  a contradiction to the assumption that $\gamma$ is the least ordinal in $A\setminus (\beta+1)$ witnessing \eqref{eq: Mitchell}. Therefore,  for co-boundedly many $\gamma \in A$, $\tilde{k}_{\gamma, \alpha}(\tilde{\zeta}_\gamma) < \tilde{\zeta}_\alpha$. So, 
     $$\tilde\zeta_\alpha \geq \limsup (\{\tilde{k}_{\beta,\alpha}(\tilde\zeta_\beta) + 1\mid \beta \in A\}).$$

Notice that this inequality cannot be strict, for otherwise letting \[\bar\zeta = \limsup (\{\tilde{k}_{\beta,\alpha}(\tilde\zeta_\beta) + 1 \mid \beta \in A\}),\]
 we would get a Mitchell order ${<}\tilde{\zeta}_\alpha$ for which $\{\beta \in A \mid \tilde{k}_{\beta,\alpha}(\tilde{\zeta}_{\beta})=\bar\zeta\}$ is bounded in $\alpha$, thus contradicting the minimality of $\tilde{\zeta}_\alpha$ required in $(\beth)$.
\end{proof}


 \smallskip

 Let us verify that one can define a quotient elementary embedding $$\ell\colon \mathcal{K}^N\rightarrow \tilde{\mathcal{K}}.$$ This will  essentially be the comparison map between the iterations $k$ and $\tilde{k}$.
 
 Let us look at the iteration $k\colon \mathcal{K}^N\to \mathcal{K}^M$. This is  a normal iteration $\langle k_{\alpha, \beta} \mid \alpha \leq \beta \leq \alpha_*\rangle$, so it makes sense to talk about the pair $(\kappa_\alpha, \zeta_\alpha)$ for which  $k_\alpha(\mathcal{U})(\kappa_\alpha,\zeta_\alpha)$ is the measure hitted at stage $\alpha$ of this iteration.

 \smallskip

 Aiming to define the elementary embedding $\ell\colon \mathcal{K}^M\rightarrow\mathcal{K}^{\tilde{M}}$, we need to assign to each $\alpha < \alpha_*$ an ordinal $h(\alpha)$ so that the next equation is fulfilled \[\ell(k(f)(\kappa_{\alpha_0},\dots, \kappa_{\alpha_{m-1}})) = \tilde{k}(f)(\tilde\kappa_{h(\alpha_0)}, \dots, \tilde\kappa_{h(\alpha_{m-1})}).\]

 We need to show that there is such a  function $h$ so that $\ell$ is elementary, and we do it by recursively defining \emph{approximations} to $\ell$. 
  Formally, for each $\alpha < \alpha_*$, we define an elementary   embedding $\ell_\alpha \colon M_\alpha \to \tilde{M}_{\bar\alpha}$ 
 \begin{equation}\label{eq: embedding}
\tag{$\oplus$} \ell_\alpha(k_\alpha(f)(\kappa_{\beta_0},\dots, \kappa_{\beta_{n-1}})):=\tilde{k}_{{\bar\alpha}}(f)(\tilde{\kappa}_{h(\beta_0)},\dots, \tilde{\kappa}_{h(\beta_{n-1})}),    
 \end{equation}
where $\bar\alpha$ is defined as follows: 
$$\bar\alpha:=\begin{cases}
    \sup h``\alpha, & \text{if $\alpha$ is limit;}\\
     h(\alpha - 1) +1, & \text{if $\alpha$ is successor.}\footnote{Note that for successor ordinals $\alpha$, $\sup h\image \alpha = h(\alpha - 1)$.}
\end{cases}$$

Eventually we shall set $\ell := \ell_{\alpha_*}$, thus producing a commutative diagram
\begin{equation}\label{eq: complete diagram}
    \tag{$\square$}
\begin{tikzcd}
\mathcal{K}^N \arrow[d, "{\id}"] \arrow[r, "{k_{0,\beta}}"] 
& M_\beta \arrow[d, "{\ell_\beta}"] \arrow[r, "{k_{\beta,\alpha}}"] 
& M_{\alpha} \arrow[d, "{\ell_\alpha}"] \arrow[r, "{k_{\alpha,\alpha_*}}"]  
& \mathcal{K}^M \arrow[d, "{\ell_{\alpha_*}}"] \\
\mathcal{K}^N \arrow[r, "{\tilde{k}_{0,\bar\beta}}"] 
& \tilde{M}_{\bar\beta} \arrow[r, "{\tilde{k}_{\bar\beta, \bar\alpha}}"] 
& \tilde{M}_{\bar\alpha} \arrow[r, "{\tilde{k}_{\bar\alpha, \bar\alpha_*}}"] 
& \mathcal{K}^{\tilde{M}},
\end{tikzcd}
\end{equation}
being  $0 \leq \beta \leq \alpha \leq \alpha_*$.

\smallskip

For each $\alpha<\alpha_*$, $h(\alpha)$ will be the least $\beta\geq \bar\alpha$ for which 
$\tilde{\kappa}_\beta=\tilde{k}_{\bar\alpha,\beta}(\ell_\alpha(\kappa_\alpha))$ and 
$\tilde{\zeta}_\beta=\tilde{k}_{\bar\alpha,\beta}(\ell_\alpha(\zeta_\alpha))$. 
In other words, we will show that for each $\alpha<\alpha_*$ there is always a stage  
$\beta\geq \bar\alpha$ in the iteration $\tilde{k}$ where the measure used at that stage is exactly 
$\tilde{k}_{\bar\alpha,\beta}(\ell_\alpha(k_\alpha(\mathcal{U})(\kappa_\alpha,\zeta_\alpha)))$. 

Assuming that $\ell_\alpha\colon M_\alpha\rightarrow\tilde{M}_{\bar\alpha}$ is elementary, the above observation will guarantee that $\ell_{\alpha+1}$ as defined in \eqref{eq: embedding} is elementary and that the diagram
 \[
\begin{tikzcd}
 M_{\alpha} \arrow[d, "{\ell_\alpha}"] \arrow[r, "{k_{\alpha,\alpha+1}}"]  
& M_{\alpha+1} \arrow[d, "{\ell_{\alpha+1}}"] \\
 \tilde{M}_{\bar\alpha} \arrow[r, "{\tilde{k}_{\bar\alpha, h(\alpha)+1}}"] 
& \tilde{M}_{h(\alpha)+1}
\end{tikzcd}
\]
commutes.  At limit stages $\alpha\leq \alpha_*$, we will conclude that $\ell_\alpha$ is elementary as well using the inductive hypothesis  ($``\ell_\beta$ is elementary for all $\beta < \alpha$") and appealing to  general facts about direct limits of elementary embeddings.


\smallskip

Let us assume that $\ell_\gamma$ was defined in a way that the values of $h\restriction\alpha$ were chosen at each ordinal to be as small as possible so that for each $\gamma<\alpha$,
$$\tilde{\kappa}_{h(\gamma)}=\tilde{k}_{\bar\gamma,h(\gamma)}(\ell_\gamma(\kappa_\gamma))\,\wedge\,\tilde{\zeta}_{h(\gamma)}=\tilde{k}_{\bar\gamma,h(\gamma)}(\ell_\gamma(\zeta_\gamma)).$$

We will maintain the following inductive hypothesis for $\alpha\leq \alpha_*$:
\begin{enumerate}[label=(\Roman*)]
    \item The embeddings $\langle \ell_\gamma\mid \gamma\leq  \alpha\rangle$ exist, are elementary and yield a commutative system with the $k_{\gamma,\alpha}$'s and $\tilde{k}_{\gamma,\bar\alpha}$'s, as described in \eqref{eq: complete diagram}.
    \item $\tilde{\kappa}_{\bar\alpha}\leq \ell_\alpha(\kappa_\alpha)$.
   
  \item There is an ordinal  $h(\alpha) \in (\bar \alpha, \bar\alpha_*)$ for which $$\text{$\tilde{\kappa}_{h(\alpha)} = \tilde{k}_{\bar\alpha,h(\alpha)}(\ell_\alpha(\kappa_\alpha))$ and
    $\tilde{\zeta}_{h(\alpha)} = \tilde{k}_{\bar\alpha,h(\alpha)}(\ell_\alpha(\zeta_\alpha))$.}$$
    where $\bar\alpha_*$ is the ordinal given by Claim~\ref{claim: things work as expected} when applied with respect to $\ell_\alpha(\kappa_\alpha).$
\end{enumerate} 
\begin{claim}
$\mathrm{(I)}$--$\mathrm{(III)}$ hold for every $\alpha\leq \alpha_*$.
\end{claim}
\begin{proof}[Proof of claim]
We prove this by induction on $\alpha\leq \alpha_*.$ 

\smallskip

As explained above, Clause~(I) follows right away from Clause~(III), and in order to prove the latter we will bear on the auxiliary Clause~(II). As a result, it suffices to prove Clauses~(II) and (III).

     \smallskip

     \underline{\textbf{Clause (II)}:} 
   
     \smallskip
     
     $\br$ If $\alpha=0$ then $\ell_0(\kappa_0)=\kappa_0>\tilde{\kappa}_{\bar 0}$ as, by definition, $\tilde{\kappa}_{\bar 0}=0.$  

     $\br$ If $\alpha=\gamma+1$ is a successor then, by induction hypothesis, $$\tilde{\kappa}_{h(\gamma)}=\tilde{k}_{\bar\gamma, h(\gamma)}(\ell_\gamma(\kappa_\gamma))\; \text{and}\; \tilde{\zeta}_{h(\gamma)}=\tilde{k}_{\bar\gamma, h(\gamma)}(\ell_\gamma(\zeta_\gamma)).$$
     As a result, we can produce a commutative diagram
      \[
\begin{tikzcd}
 M_{\gamma} \arrow[d, "{\ell_{\gamma}}"] \arrow[r, "{k_{\gamma,\alpha}}"]  
& M_{\alpha} \arrow[d, "{\ell_{\alpha}}"] \\
 \tilde{M}_{\bar{\gamma}} \arrow[r, "{\tilde{k}_{\bar\gamma,\bar\alpha}}"] 
& \tilde{M}_{\bar\alpha}
\end{tikzcd}
\]
being 
$\ell_\alpha(k_{\gamma,\alpha}(f)(\kappa_\gamma)):=\tilde{k}_{\bar\gamma,\bar\alpha}(f)(\tilde{\kappa}_{h(\gamma)})$.

Since $k$ is a normal iteration, $\kappa_\gamma< \kappa_\alpha$, hence $\ell_\gamma(\kappa_\gamma)<\ell_\alpha(\kappa_\alpha)$   
 and thus  $\ell_\alpha(\kappa_\alpha)>\tilde{\kappa}_{h(\gamma)}$. By elementarity, $\ell_\alpha(\kappa_\alpha)$ is a $\tilde{M}_{\bar \alpha}$-measurable cardinal, ergo by the way the iteration $\tilde{k}$ was defined (see Clause~$(\aleph)$) we conclude that $$\ell_\alpha(\kappa_\alpha)\geq \tilde{\kappa}_{h(\gamma)+1}:=\tilde{\kappa}_{\bar\alpha}.$$
     

     $\br$ If $\alpha$ is a limit then, we have two cases: either $\cf^{\mathcal{K}}(\alpha)\neq \omega$ or  $\cf^{\mathcal{K}}(\alpha)=\omega$.

\smallskip

$\br\br$ Suppose $\cf^{\mathcal{K}^N}(\alpha)\neq \omega$. In this case, the iteration $\tilde{k}$ has been defined according to Clause~$(\aleph)$ on page~\pageref{eq: beth of the iteration} and as a result $\tilde{\kappa}_{\bar\alpha}$ is the first $\tilde{M}_{\bar\alpha}$-measurable cardinal greater than or equal to $\sup_{\beta<\bar\alpha}\tilde{\kappa}_{\beta}=\sup_{\beta<\alpha}\tilde{\kappa}_{h(\beta)}$. 

By normality of the iteration $k$, $\kappa_\beta<\kappa_\alpha$ for all $\beta<\alpha$. 
With the aid of the induction hypothesis we can define an elementary  embedding $\ell_\alpha\colon M_\alpha\rightarrow \tilde{M}_{\bar\alpha}$ that abides by equation \eqref{eq: embedding}. Therefore, for $\beta<\alpha$, 
$$\ell_\alpha(\kappa_\alpha)>\ell_\alpha(\kappa_\beta)=\ell_\alpha(k_\alpha(\id)(\kappa_\beta))=\tilde{k}_{\bar\alpha}(\id)(\tilde{\kappa}_{h(\beta)})=\tilde{\kappa}_{h(\beta)}.$$
Thereby,  $\ell_\alpha(\kappa_\alpha)\geq \sup_{\beta<\alpha}\tilde{\kappa}_{h(\beta)}$. Since $\ell_\alpha(\kappa_\alpha)$ is $\tilde{M}_{\bar\alpha}$-measurable we conclude, appealing to the minimality of $\tilde{\kappa}_{\bar\alpha}$, that $\tilde{\kappa}_{\bar\alpha}\leq \ell_\alpha(\kappa_\alpha).$

\smallskip

$\br\br$ Suppose $\cf^{\mathcal{K}^N}(\alpha)=\omega$. In this case, $\cf^{\mathcal{K}^N}(\bar\alpha)=\omega$ and the iteration $\tilde{k}$ abides by Clause~$(\beth)$ on page~\pageref{eq: beth of the iteration}.  Let us assume towards a contradiction that $\ell_\alpha(\kappa_\alpha)<\tilde{\kappa}_{\bar\alpha}$. We have already showed that $\ell_\alpha(\kappa_\alpha)\geq \sup_{\beta<\alpha}\tilde{\kappa}_\beta$ so,  by minimality of $\tilde{\kappa}_{\bar\alpha}$,
 $$A=\{\beta<\bar\alpha\mid \tilde{k}_{\beta,\bar\alpha}(\tilde{\kappa}_\beta)=\ell_\alpha(\kappa_\alpha)\,\wedge\,\tilde{k}_{\beta,\bar\alpha}(\tilde{\zeta}_\beta)=\ell_\alpha(\zeta_\alpha) \}$$
 is unbounded in $\bar\alpha$. We will show that this is impossible.
 \begin{subclaim}\label{subclaim sigma closure}
    The set
     \[\{\beta<\alpha\mid (k_{\beta,\alpha}(\kappa_\beta)>\kappa_\alpha)\;\; \text {or}\;\; (k_{\beta,\alpha}(\kappa_\beta)=\kappa_\alpha\,\wedge\, k_{\beta,\alpha}(\zeta_\beta)> \zeta_\alpha)\}\] is bounded in $\alpha$.
 \end{subclaim}
 \begin{proof}[Proof of subclaim]
Let $<_{\mathrm{lex}}$ denote the lexicographic order on $\Ord\times \Ord$.

We claim that for each choice  $R\in\{=,>_{\mathrm{lex}}\}$ the set 
\begin{equation}\label{eq: critical and orders}
 \{\beta<\alpha\mid k_{\beta,\alpha}(\langle\kappa_\beta, \zeta_\beta\rangle) R\, \langle \kappa_\alpha,\zeta_\alpha\rangle\}  
\end{equation}
is bounded in $\alpha$. 
We prove this arguing for the sake of a contradiction.

In case $R$ is $=$  our   assumption yields
$$k_\alpha(\mathcal{U})(\kappa_\alpha,\zeta_\alpha)=\{X\in \mathcal{P}^{M_\alpha}(\kappa_\alpha)\mid \langle \kappa_{\beta_n}\mid n<\omega\rangle\s^* X\}.$$
We claim that this is impossible. Since $M[H]$ is closed under $\omega$-sequences in $V[G]$ it follows that $k_{\alpha}(\mathcal{U})(\kappa_\alpha,\zeta_\alpha)$ belongs to $ M[H]$. Our anti-large cardinal assumption\footnote{I.e., ''There is no inner model for $\exists\alpha\,(o(\alpha)=\alpha^{++})$".} ensures that \emph{Mitchell's Covering Theorem} (Theorem~\ref{CoreModelTheorem}) applies within $M[H]$. As $k_{\alpha}(\mathcal{U})(\kappa_\alpha,\zeta_\alpha)$ is $\mathcal{K}^{M[H]}$-normal, by maximality, $$k_{\alpha}(\mathcal{U})(\kappa_\alpha,\zeta_\alpha)\cap \mathcal{K}^{M[H]}\in \mathcal{K}^{M[H]}.$$ By generic absoluteness of $\mathcal K$,
$k_{\alpha}(\mathcal{U})(\kappa_\alpha,\zeta_\alpha)\cap \mathcal{K}^{M}\in \mathcal{K}^{M}.$
But this is impossible: On the one hand, $\mathcal{P}^{M_\alpha}(\kappa_\alpha)=\mathcal{P}^{\mathcal{K}^M}(\kappa_\alpha)$ (by normality of the iteration $k$) so that $k_\alpha(\mathcal{U})(\kappa_\alpha,\zeta_\alpha)\cap \mathcal{K}^M=k_\alpha(\mathcal{U})(\kappa_\alpha,\zeta_\alpha)\in \mathcal{K}^M$; on the other hand, the measure $k_\alpha(\mathcal{U})(\kappa_\alpha,\zeta_\alpha)$ does not belong to $\mathcal{K}^M$ as, by coherency, once this  is hitted it does not belong to any of the subsequent ultrapowers in $k$.

\smallskip

In case $R$ is  $>_{\mathrm{lex}}$ analogous considerations give a contradiction: 

Let $\langle \beta_n\mid n<\omega\rangle$ be a cofinal sequence in $\alpha$ consisting of ordinals from the set displayed in \eqref{eq: critical and orders} and assume without loss of generality that $\kappa_\alpha,\zeta_\alpha$ belong to the range of $k_{\beta_n,\alpha}$ for all $n$. Let $V_n:=k_{\beta_n}(\mathcal{U})(\bar{\kappa}^n_{\alpha},\bar{\zeta}^n_{\alpha})$ where $\bar{\kappa}_{\alpha}^n, \bar{\zeta}^n_{\alpha}$ are the pre-images of $\kappa_\alpha$ and $\zeta_\alpha$ by way of $k_{\beta_n,\alpha}$. Since at stage $\beta_{n}$ the iteration $k$  hits the measure $k_{\beta_n}(\mathcal{U})(\kappa_{\beta_n}, \zeta_{\beta_n})$ and either $\kappa_{\beta} > \bar\kappa^n_\alpha$ or $\kappa_\beta=\bar\kappa^n_\alpha$ but $\zeta_{\beta_n}>\bar{\zeta}^n_\alpha$, by coherency of $k_{\beta_n}(\mathcal{U})$, $V_n\in M_{\beta_n+1}$ and so $V_n\in \mathcal{K}^M\s M.$  

It is now routine to check that 
      $$k_\alpha(\mathcal{U})(\kappa_\alpha,\zeta_\alpha)=\{X\in \mathcal{P}^{M_\alpha}(\kappa_\alpha)\mid \exists n\, \forall m\geq n\, (X\cap \kappa_{\beta_m}\in V_n)\}.$$
      But then by $\sigma$-closure of $M[H]$ in $V[G]$, and Mitchell's Covering Theorem, we get again that $k_\alpha(\mathcal{U})(\kappa_\alpha,\zeta_\alpha)\in \mathcal{K}^M$. This is impossible.
 \end{proof}

 \begin{subclaim}
 The following superset of $A$ is bounded in $\alpha$: 
     \[B = \{\beta < \bar \alpha \mid \tilde{k}_{\beta,\bar\alpha}(\tilde\kappa_\beta)=\ell_\alpha(\kappa_\alpha)\wedge \tilde{k}_{\beta,\bar\alpha}(\tilde\zeta_\beta)\geq \ell_\alpha(\zeta_\alpha)\}\]
    In particular, $A$ must be  bounded.
 \end{subclaim}
 \begin{proof}[Proof of subclaim]
Suppose otherwise – that $B$ is unbounded in $\bar\alpha$. 

Using this assumption we will argue that 
$$C:=\{\gamma<\alpha\mid \tilde{k}_{h(\gamma),\bar\alpha}(\tilde{\kappa}_{h(\gamma)})=\ell_\alpha(\kappa_\alpha)\,\wedge\,\tilde{k}_{h(\gamma),\bar\alpha}(\tilde{\zeta}_{h(\gamma)})\geq \ell_\alpha(\zeta_\alpha) \}$$
is unbounded in $\alpha$. Combined  with  our induction hypothesis (II),  $$\text{$\ell_\alpha(\kappa_\alpha)=\tilde{k}_{\bar \gamma, \bar\alpha}(\ell_\gamma(\kappa_\gamma))$ and $\ell_\alpha(\zeta_\alpha)\leq  \tilde{k}_{\bar \gamma, \bar\alpha}(\ell_\gamma(\zeta_\gamma))$}$$
holds for all $\gamma\in C$. By commutativity of the diagrams (Clause~(I) of the induction hypothesis),
$k_{\gamma,\alpha}(\kappa_\gamma)=\kappa_\alpha$ and $k_{\gamma,\alpha}(\zeta_\gamma)\geq \zeta_\alpha$ for all $\gamma\in C$. This is at odds with  Claim~\ref{subclaim sigma closure} and yields the desired contradiction. 

\smallskip

So, suppose that $B\s \bar\alpha$ is unbounded. Let $\beta \in B$. 
     If $\beta\in \range(h)$  then we are good. Otherwise, let $\gamma < \alpha$ be the unique ordinal such that $\beta\in [\bar\gamma, h(\gamma))$. {Without loss of generality,  $\beta$ is such that the corresponding $\gamma$ is above the supremum of the bounded set of Subclaim \ref{subclaim sigma closure}.} In particular,  
     $$k_{\gamma,\alpha}(\kappa_\gamma)\leq \kappa_\alpha.$$

    Note that $\ell_\gamma(\kappa_\gamma)$ is a $\tilde{M}_{\bar\gamma}$-measurable $\geq \tilde{\kappa}_{\bar\gamma}$, so by Claim~\ref{claim: things work as expected} there is $\bar\gamma_*\in (h(\gamma),\Omega]$ such that $\tilde{k}_{\bar\gamma,\delta}(\ell_\gamma(\kappa_\gamma))\geq \tilde{\kappa}_\delta$ for all $\delta<\bar\gamma_*$. 
     In particular, 
$$\tilde{\kappa}_{h(\gamma)}=\tilde{k}_{\bar\gamma,h(\gamma)}(\ell_\gamma(\kappa_\gamma))\geq \tilde{k}_{\beta, h(\gamma)}(\tilde{\kappa}_\beta).$$ 
     Moreover, equality holds because:
     \[\ell_\alpha(\kappa_\alpha)=\tilde{k}_{\beta,\bar\alpha}(\tilde{\kappa}_\beta) \leq \tilde{k}_{\beta,\bar\alpha}(\tilde{k}_{\bar\gamma,\beta}(\ell_\gamma(\kappa_\gamma))) = \tilde{k}_{\bar\gamma,\bar\alpha}(\ell_{\gamma}(\kappa_\gamma)) = \ell_\alpha(k_{\gamma,\alpha}(\kappa_\gamma)) \leq \ell_\alpha(\kappa_\alpha).\] 
     The first equality holds because $\beta\in B$; the first inequality follows because $\beta<\bar\gamma_*$;  the third inequality uses the commutativity of the diagram and the latest inequality our choice of $\gamma$.  As a result, we conclude that
    $$\tilde{k}_{h(\gamma),\bar\alpha}(\tilde{\kappa}_{h(\gamma)})=\ell_\alpha(\kappa_\alpha).$$

     For latter use also note that the above equality  yields 
     \begin{equation}\label{eq: **}
       \tag{**}\tilde{k}_{\bar\gamma,\beta}(\ell_\gamma(\kappa_\gamma))=\tilde{\kappa}_\beta.  
     \end{equation}
     {Now, we would like to claim that  $\tilde{k}_{h(\gamma),\bar\alpha}(\tilde{\zeta}_{h(\gamma)})\geq \ell_\alpha(\zeta_\alpha)$. To show this it suffices to justify that $\tilde{\zeta}_{h(\gamma)}\geq \tilde{k}_{\beta, h(\gamma)}(\tilde{\zeta}_\beta)$ (because $\beta\in B$).  Suppose towards a contradiction that $\tilde{\zeta}_{h(\gamma)}<\tilde{k}_{\beta, h(\gamma)}(\tilde{\zeta}_\beta)$. Then $\tilde{k}_{\bar\gamma,\beta}(\ell_\gamma(\zeta_\gamma))<\tilde{\zeta}_\beta.$ Since the iteration $\tilde{k}$ has not halted at stage $\beta$ it must be  that  $\cf^{\mathcal{K}^N}(\beta)=\omega$\footnote{See Claim~\ref{subclaim: Mitchell}~(2).} and that there are cofinally-many $\delta\in (\bar\gamma, \beta)$ such that $\tilde{k}_{\delta,\beta}(\tilde{\kappa}_\delta)=\tilde{\kappa}_\beta$ and $\tilde{k}_{\delta,\beta}(\tilde{\zeta}_\delta)=\tilde{k}_{\bar\gamma,\beta}(\ell_\gamma(\zeta_\gamma)).$ Therefore, using \eqref{eq: **} above,  for cofinally $\delta\in (\bar\gamma,\beta)$,
     $$\tilde{k}_{\bar\gamma,\delta}(\ell_\gamma(\kappa_\gamma))=\tilde{\kappa}_\delta\;\text{and}\; \tilde{k}_{\bar\gamma,\delta}(\ell_\gamma(\zeta_\gamma))=\tilde{\zeta}_\delta.$$
     Yet, $h(\gamma)>\beta$ was the minimal ordinal above $\bar\gamma$ making possible the above two equalities, thus yielding a contradiction with $\tilde{\zeta}_{h(\gamma)}<\tilde{k}_{\beta, h(\gamma)}(\tilde{\zeta}_\beta).$ }
 \end{proof}

The previous claim yields a contradiction and completes the proof of (II).

\smallskip

\underline{\textbf{Clause~(III)}:}
Look at the $\tilde{M}_{\bar\alpha}$-measurable $\ell_\alpha(\kappa_\alpha)$ (which is $\geq \tilde{\kappa}_{\bar\alpha}$, by \textbf{Clause~(II)}) and invoke  Claim~\ref{claim: things work as expected} with respect  to it. In this way we  find an ordinal  $\bar\alpha_*\in (\bar\alpha, \Omega]$, with countable ${\mathcal{K}^N}$-cofinality, for which the set 
$$\{\gamma<\bar\alpha_*\mid \gamma\geq \bar\alpha\,\wedge\, \tilde{k}_{\gamma,\bar \alpha_*}(\tilde{\kappa}_\gamma)=\tilde{k}_{\bar\alpha,\bar\alpha_*}(\ell_\alpha(\kappa_\alpha))\,\wedge\, \tilde{k}_{\gamma,\bar \alpha_*}(\tilde{\zeta}_\gamma)=\tilde{k}_{\bar\alpha,\bar\alpha_*}(\ell_\alpha(\zeta_\alpha))\}$$
is unbounded. Let $h(\alpha)$ be the least member in the above-displayed set above $\bar\alpha$ and use the elementarity of $\tilde{k}_{h(\alpha), \bar\alpha_*}$ to infer that  $$\text{$\tilde{\kappa}_{h(\alpha)}=\tilde{k}_{\bar\alpha, h(\alpha)}(\ell_\alpha(\kappa_\alpha))$ and $\tilde{\zeta}_{h(\alpha)}=\tilde{k}_{\bar\alpha, h(\alpha)}(\ell_\alpha(\zeta_\alpha))$.}$$
Certainly $h(\alpha)<\Omega$ so we are done.
\end{proof}
The previous claim allows us to construct $\ell\colon \mathcal{K}^{M}\rightarrow \tilde{\mathcal{K}}$ making the diagram commute. This completes the proof of the lemma.
\end{proof}

\smallskip


\begin{lemma}\label{lemma:imagesofmu}Assume that \eqref{weakerantilargecardinals} holds. 
    Let $\iota\colon V \to N$ be a finite iteration factoring $j\restriction V$ and let us assume that $\iota$ is normal. Let $k \colon N \to M$ be the factor elementary embedding. Then, for every $\mu \in \Gamma_\iota$ we have $k\image \mu \subseteq \mu$.   
\end{lemma}
\begin{proof}
   Since we want to bound the value of $k$ on ordinals it suffices to look at its restriction  to the core model $\mathcal{K}$. Applying Lemma~\ref{lemma: universal iteration} to $j\restriction\mathcal{K}$, $\iota\restriction \mathcal{K}$ and $k\restriction \mathcal{K}^N$ we find an $\mathcal{K}^N$-internal, normal iteration $\tilde{k}\colon \mathcal{K}^N\rightarrow\tilde{\mathcal{K}}$ such that  
\[
\begin{tikzcd}
    \mathcal{K} \arrow[r, "{j}"] \arrow[d, "\iota"'] 
     & \mathcal{K}^M \arrow[d, "\ell"'] \\
   \mathcal{K}^N \arrow[ru, "{k}"] \arrow[r, "\tilde{k}"'] 
     &  \tilde{\mathcal{K}}
\end{tikzcd}
\]
commutes.
To bound  $k(\gamma)$ for  $\gamma<\mu\in \Gamma_\iota$, we bound the (potentially) larger value $\tilde{k}(\gamma)$. Therefore, we shall prove that for each $\mu\in \Gamma_\iota$,  $\tilde{k}``\mu\s \mu.$



Fix $\mu\in \Gamma_\iota$. For each $\tilde{\kappa}_0\leq \alpha<\mu$ we  show $\tilde{k}(\alpha)<(\alpha^{+3})^{\mathcal{K}^N}.$ Since $\tilde{k}$ is internal to ${\mathcal{K}^N}$ and $\mu$ is ${\mathcal{K}^N}$-measurable we conclude that $\tilde{k}(\alpha)<\mu.$ Towards a contradiction, assume that for some $\alpha<\mu$, $\tilde{k}(\alpha)\geq (\alpha^{+3})^{\mathcal{K}^N}.$

Since $\tilde{k}$ is a linear iteration of normal measures, it must be that the $\tilde{k}$ has length at least $(\alpha^{+3})^{\mathcal{K}^N}$ – otherwise, the iteration would be unable to map $\alpha$ that high. In the next argument we shall care about $\langle (\tilde{\kappa}_\beta,\tilde{\zeta}_\beta)\mid \beta\in S\rangle$ where $S$ is the set of ordinals ${<}(\alpha^{+3})^{\mathcal{K}^N}$ with ${\mathcal{K}^N}$-cofinality $\omega.$ Recall that these are the critical points/Mitchell orders  hitted by $\tilde{k}$ at limit stages of ${\mathcal{K}^N}$-cofinality $\omega$. We use the argument of Claim~\ref{subclaim:Mitchells} to get a contradiction. 

\smallskip

For each $\beta\in S$, let $f_\beta\colon [\tilde{\kappa}_0]^n\rightarrow(\tilde{\kappa}_0+1)$, $g_\beta\colon [\tilde{\kappa}_0]^n\rightarrow \tilde{\kappa}_0$ and generators $\tilde{\kappa}_{\beta_0},\dots, \tilde{\kappa}_{\beta_{n_\beta-1}}<\tilde{\kappa}_{\beta}$ representing $\tilde{\kappa}_\beta$ and $\tilde{\zeta}_\beta$, respectively. Namely,
$$\tilde{\kappa}_\beta=k_\beta(f_\beta)(\tilde{\kappa}_{\beta_0},\dots, \tilde{\kappa}_{\beta_{n_\beta-1}}),\footnote{Recall that we are always assuming that $\tilde{\kappa}_\beta\leq \tilde{k}_{\beta}(\tilde{\kappa}_0)$, so the above makes sense.}$$
$$\tilde{\zeta}_\beta=k_\beta(g_\beta)(\tilde{\kappa}_{\beta_0},\dots, \tilde{\kappa}_{\beta_{n_\beta-1}}).$$
Clearly $\tilde{\kappa}_\beta\leq \alpha$ as otherwise $\alpha$ would not be moved in the further stages.

The above naturally induces a function $F\colon S\rightarrow (H_{|\alpha|^{++}})^{\mathcal{K}^N}$ given by $$\beta\mapsto \langle f_\beta, g_\beta, \langle\tilde{\kappa}_{\beta_0},\dots, \tilde{\kappa}_{\beta_{n-1}}\rangle\rangle.$$
This is well defined: There are at most $(\tilde{\kappa}_0^{+})^{\mathcal{K}^N}$-many such functions  and $(\tilde{\kappa}_0^{+})^{\mathcal{K}^N}\leq |\alpha|^{+{\mathcal{K}^N}}$; since $\mathcal{K}\models${``There is no inner model for $\exists \theta\, (o(\theta)=\theta^{++})$"} the same is true, by elementarity, in ${\mathcal{K}^N}$ and in its ultrapowers $\tilde{M}_\beta.$

\smallskip

Let $S_0\s S$ be a stationary set for which the function $F$ is constant; say, with value $\langle f^*,g^*,\langle \kappa^*_{0},\dots, \kappa^*_{n^*-1}\rangle\rangle.$ It follows that for each $\beta<\gamma$ in $S_0$, $$\tilde{k}_{\beta,\gamma}(\tilde{\kappa}_\beta)=\tilde{\kappa}_\gamma\;\text{and}\; \tilde{k}_{\beta,\gamma}(\tilde{\zeta}_\beta)=\tilde{\zeta}_\gamma.$$
By the same reasoning of Claim~\ref{subclaim:Mitchells}, this is a contradiction. 
\end{proof}

We will also need the following observation concerning the regularity of the critical points of those finite normal iterations $\iota\colon V \to N$ that factor $j\restriction V$. We emphasize that this observation applies to an arbitrary non-stationary support iteration of Prikry-type forcings $\mathbb{P}_\kappa$, as  described in Definition~\ref{Gitikiteration}. In particular, the lemma is not specific to the Gluing Iteration of \S\ref{sec: the gluing iteration}. Consequently, the weaker anti–large cardinal assumption \eqref{weakerantilargecardinals} is sufficient.

\begin{lemma}\label{lem:mu-remains-regular}
Assume that \eqref{weakerantilargecardinals} holds. Let $\iota\colon V\rightarrow N$ be a finite normal iteration factoring $j\restriction V$ and $\mu \in \Gamma_\iota$. Then,  $\mu$ is regular in $M[H]$.\footnote{We remind our readers that as per our \textbf{Setup} in page \pageref{setupcoding} $j\colon V[G]\rightarrow M[H]$ was a $V[G]$-internal ultrapower by a $\kappa$-complete ultrafilter on $\kappa.$}
\end{lemma}
\begin{proof}
    Let $\iota \colon V \to N$  and $\mu \in \Gamma_\iota$ be as above. Clearly, $\mu$ is regular in $N$. 

    \begin{claim}\label{claim: mu is M regular}
        $M\models``\mu$ is regular".
    \end{claim}
    \begin{proof}[Proof of claim]
    Suppose otherwise. Let $a \in M$ be a cofinal subset of $\mu$ of order-type ${<}\mu$. Let $\iota'\colon V \to N'$ be a finite normal iteration, exteding $\iota$,  factoring $j$, and such that $a = k'(b)$ for some $b \in N'$. (Here $k'\colon N'\rightarrow M$ denotes the factoring embedding making $j=k'\circ \iota'$.) By Lemma~\ref{lemma:imagesofmu}, $k' \image \mu \subseteq \mu$ and thus $b\in N'$ must be cofinal in $\mu$ and $\otp b < \mu$: Indeed, suppose that $b\s \mu$ were to be bounded in $\mu$. Since $N'\models ``\mu$ is regular" (as $\mu\in \Gamma_\iota\s \Gamma_{\iota'}$) it must be that $\sigma:=\sup(b)<\mu$. Then $k'(\sigma)<\mu$. So, by unboundedness of $k'(b)$, there is $\alpha\in k'(b)$ such that $k'(\sigma)<\alpha$. By elementarity of $k'$ there is $\alpha\in b$ such that $\sigma<\alpha$, which is impossible.
    
   The above shows that $b$ is a witness for the singularity of $\mu$ inside $N'$, but this is not possible in that $\mu$ is a member of $\Gamma_{\iota'}.$
    \end{proof}


\begin{claim}
    $M[H]\models ``\mu$ is regular".
\end{claim}
\begin{proof}[Proof of claim]
      First, note that $N\models ``\cf(\mu) > \kappa^{+}$" and $N \models ``|\mathbb{P}| \leq \kappa^{+}$". So, $N[G] \models ``\cf(\mu) \geq \kappa^{+}$". As per the \emph{fusion-like} argument in \cite[Corollary~2.6]{BenUng}, the non-stationary supported iteration of Prikry-type forcings $\mathbb{P}_\kappa$ does not change cofinalities of  cardinals with cofinalities ${\geq}\kappa$. Since $N$ is closed under $\kappa$-sequences in $V$,   $ \cf^V(\mu)\geq \kappa^+$ and so $V[G]\models``\cf(\mu) \geq \kappa^{+}$". 

      \smallskip

      Now, suppose that $M[H]\models ``\mu$ is singular". As $M[H]\subseteq V[G]$, the cofinality of $\mu$ in $M[H]$ must be uncountable. Now we use our assumption that there is no inner model of $``\exists \alpha(o(\alpha)=\alpha^{++})$". As $\mu$ is regular in $M$, and in particular regular in $\mathcal{K}^{M}$, and singular in $M[H]$, it is a consequence of 
    Mitchell's covering theorem \cite[Theorem~2.6]{Mithand2} that there is a weak Prikry--Magidor set $C \subseteq \mu$ in $M[H]$ consisting of $\mathcal{K}^M$-regulars.\footnote{Recall that $C\s \mu$ is a  \emph{Prikry--Magidor set} if $|C|<\mu$ and for every $D\in \mathcal{K}^M$ club on $\mu$ the set $D\setminus C$ is bounded in $\mu$.}

    \smallskip

    As the map $k \colon N \to M$ has critical point above $\omega_1$ and $k``\mu\s \mu$, the set
    $$D=\{\alpha<\mu\mid k(\alpha)=\alpha\}$$
    is an $\omega$-club. So, there is a $\mathcal{K}^M$-regular cardinal $\alpha\geq \kappa^+$ with $k(\alpha)=\alpha$ and \begin{equation}\label{eq: cofinalities}
      \tag{**}  \cf^{M[H]}(\alpha)=\cf^{V[G]}(\alpha)=\omega.
    \end{equation}

    \smallskip

     Now comes the end game. By elementarity of $k$, $\mathcal{K}^{N} \models ``\alpha\text{ is regular}$". As the embedding $\iota\colon V\rightarrow N$ is a lifting of $\iota \restriction \mathcal K\colon \mathcal K \to \mathcal{K}^N$, a finite normal iteration  using measures from $\mathcal K$, the $\mathcal{K}^N$-regularity of $\alpha$ entails $\mathcal{K} \models ``\cf(\alpha) \geq \kappa^{+}$". Now recall that $V$ was obtained after forcing over $\mathcal{K}$ with the fast function forcing $\mathbb{S}_\kappa$, hence after forcing with a cofinality preserving poset. As a result, $V\models ``\cf(\alpha)\geq \kappa^+$". Finally, $V[G]$ is a generic extension of $V$ via the non-stationary supported iteration of Prikry-type forcings, $\mathbb{P}_\kappa$. It was mentioned above that this iteration does not change the cofinality of cardinals whose cofinalities are $\geq \kappa$. Hence, $V[G]\models``\cf(\alpha)\geq \kappa^+$". However, this contradicts the  equation displayed in \eqref{eq: cofinalities}.
\end{proof}
The above completes the proof that $M[H]\models ``\mu$ is regular'' for $\mu\in \Gamma_\iota.$
\end{proof}

\begin{remark}
    Note that the previous argument does not use at any point that $\mathbb{P}_\kappa$ is the Gluing Iteration. In other words, the argument is general enough to apply to any non-stationary supported iteration of Prikry-type forcings as in Definition~\ref{Gitikiteration}.
\end{remark}

We are now in a position in which we can prove the \textbf{Key Lemma} (Lemma~\ref{lem;maximal-stem}) about the existence of maximal stems. Since this lemma relies on the specific definition of the Gluing Iteration, this time we have to assume the stronger anti-large cardinal assumption \eqref{BlanketAssumption}. 
\begin{proof}[Proof of Lemma~\ref{lem;maximal-stem}]
    For the sake of a contradiction, let us suppose that the lemma is false. That means that there is a finite iteration $\iota_0\colon V\rightarrow N_0$ and $\bar\mu\in A_{\iota_0}$ such that for every $\leq_j$-extension of $\iota_0$ to a finite iteration $\iota\colon V\rightarrow N$, factoring $j\restriction V$ via $k\colon N\rightarrow M$, the set   $$\bar{H}^{\iota, k}_{\bar\mu}:= \{b_{\mathrm{stem}(k(r(\bar\mu))_{H\restriction k(\bar\mu)})} \mid \text{$r$ witnesses $\bar\mu\in A_\iota$}\}$$ does not admit a $\sq$-maximal element.

    \smallskip
    
    Let $\bar\mu$ be the first of these ordinals and $\iota\colon V\rightarrow N$ a $\leq_j$-extension of the offending $\iota_0$ such that { $\bar\mu\in k^{-1}``\mathcal{C}_\iota.$}\footnote{Recall $\mathcal{C}_\iota=\{k(\bar \mu_i)\mid \bar\mu_i\in \crit(\iota)\}$} Through the proof we write $\mu:=k(\bar\mu)$.

    \smallskip

    For technical reasons that will become apparent when we prove Claim~\ref{claim: key claim indeed}, we shall assume that  $\mathbb{P}_\kappa$ actually is the $\leq^*$-dense subposet of Lemma~\ref{lemma: reducingdensesets}.

    \smallskip

  Since $\bar\mu\in \Gamma_{\iota}$, by Lemma \ref{lem:mu-remains-regular}, at step $\bar\mu$ in the forcing iteration $j(\mathbb{P}_\kappa)$, the poset does not singularize $\bar \mu$. However, {$\mu$ is measurable} in $M$ and as a result it is singularized in $M[H]$ (as per our anti-large cardinal hypothesis \eqref{BlanketAssumption} and the design of the iteration $j(\mathbb{P}_\kappa)$). This has the following consequence:
\begin{claim}
    $b_{\mu}\cap \bar\mu$ is a bounded subset of $\bar\mu$.
\end{claim}
\begin{proof}[Proof of claim]
The argument is sort of buried in the proof of Lemma~\ref{lem:mu-remains-regular}, but since it is short enough we spell it out:
Suppose towards a contradiction that $b_\mu\cap \bar\mu$ was unbounded in $\bar\mu$. Then this is a club in $\bar\mu$ consisting of $M$-regular cardinals. Since $\bar\mu$ is $N$-regular, hence of $V[G]$-cofinality ${\geq}\kappa^+$, and $\crit(k)\geq \omega_1$, the set 
$\{\alpha<\bar \mu\mid k(\alpha)=\alpha\}$
is an $\omega$-club on $\bar\mu$. Ergo, there is a $M$-regular cardinal $\alpha<\bar\mu$ such that $k(\alpha)=\alpha$ and $\cf^{M[H]}(\alpha)=\omega$. By elementarity, $\alpha$ is $N$-regular. Ergo, since $\iota\colon V\rightarrow N$ is a finite iteration, $\cf^{V[G]}(\alpha)\geq (\kappa^+)^{V[G]}$, which yields the sought  contradiction.
\end{proof}
By the previous claim there is  
$\langle t,T\rangle$ in the $\mu$-component of $ H$  such that $C\cap \bar\mu =b_{t}$. 
Let $t=\langle \delta_0,\dots, \delta_{n-1}\rangle$ be the associated stem. Since $\delta_i<\mu\in \crit(j)$ each of these ordinals is represented by critical points below $\mu$; i.e.,
$$\delta_i=j(f_i)(\mu_{0}^{i},\dots,\mu_{k_i}^i) $$
for some function $f_i\colon [\kappa]^{k_i}\rightarrow V$ and $\langle \mu_0^i,\dots \mu^i_{k_i}\rangle\in \mu^{<\omega}$.

\smallskip

Arguing as in  Lemma~\ref{lemma: characterizing Ciota} we find a finite normal iteration $\iota_1\colon V\rightarrow N_1$ factoring $j$ by way of $k_1\colon N_1\rightarrow M$ such that $\mathcal{C}_{\iota_1}=\mathcal{C}_{\iota}\cup \{\mu^i_l\mid i< n,\, l<k_i\}.$
In particular,  $\delta_i\in \range(k_1)$ for all $i<n$. Say   $k_1(\bar\delta_i)=\delta_i$ for $i<n$.

\smallskip

From this,  we obtain an expanded  commutative diagram of embeddings
\[
\begin{tikzcd}[column sep=2.5em, row sep=2.2em]
& V \arrow[dl, "\iota"'] \arrow[d, "\iota_1"] \arrow[dr, "j"] & \\
N \arrow[r, "\ell_1"] \arrow[rr, "k"', out=-90, in=-90] 
& N_1 \arrow[r, "k_1"] & M
\end{tikzcd}
\]
by stipulating $\ell_1(\iota(f)(k^{-1}``\mathcal{C}_{\iota})):=\iota_1(\hat{f})(k_{1}^{-1}``\mathcal{C}_{\iota_1})$, where  $\hat{f}\colon [\kappa]^{|\mathcal{C}_{\iota_1}|}\rightarrow V$ is the function defined by $\hat{f}(\xi_0,\dots, \xi_{|\mathcal{C}_{\iota}|-1},\dots,\xi_{|\mathcal{C}_{\iota_1}|-1}):= f(\xi_0,\dots, \xi_{|\mathcal{C}_{\iota}|-1}).$

Note that $\ell_1(\bar \mu)=\bar\mu$, and so  $k_0(\bar \mu)=k_1(\bar\mu)=\mu$. 

The above ensures that in $N_1$, {$\langle \bar\delta_0, \dots, \bar\delta_{n-1}\rangle$ is a stem for a condition in $\mathbb{G}_{\bar\mu}^{N_1}$  that is sent to $H$ by way of the new factor embedding $k_1\colon N_1\rightarrow M$:} Indeed, a witnessing condition could be, for instance,  the pair $\langle s, S\rangle$ where $s:=\langle \bar\delta_0,\dots, \bar\delta_{n-1}\rangle$ and $S$ is the weakest tree on $\bar \mu$ for which $\langle s, S\rangle$ is a condition. Certainly, $k_1(\langle s, S\rangle)\geq \langle t, T\rangle\in H$, ergo $\langle s, S\rangle\in k_1^{-1}(H).$

\begin{claim}\label{claim: key claim indeed}
  $\bar{H}^{\iota_1,k_1}_{\bar\mu}$ admits a $\sq$-maximal element.   
\end{claim}
\begin{proof}[Proof of claim]
We claim that $$b_{\stem(k_1(s)(k_1(\bar\mu)))_{H\restriction k_1(\mu)}}=b_{\langle \delta_0,\dots, \delta_{n-1}\rangle}$$ is such a $\sq$-maximal element.

\smallskip

On the one hand, we have already argued that $\langle s, S\rangle\in k_{1}^{-1}(H)$ {but does it satisfy $j(p)(k_1(\bar\mu))<k_1(s)(k_1(\bar\mu))$ for all $p\in G$?} As a result, $b_{\stem(k_1(s)(k_1(\bar\mu)))_{H\restriction k_1(\mu)}}$ is part of the union defining $\bar{H}^{\iota_1,k_1}_{\bar\mu}$.

\smallskip

Let  $r$ witnessing $\bar\mu\in A_{\iota_1}$. Then, $k_1(r)\in H$.

Let $\delta<\bar\mu$ be an ordinal such that $$r\restriction \bar\mu\forces^N_{\iota(\mathbb{P}_\kappa)_{\bar\mu}}``\stem(r(\bar\mu))\s \check{\delta}".$$ Such an ordinal exists  because $r\restriction\bar\mu$ forces a $\iota(\mathbb{P}_\kappa)_\eta$-name (with $\eta<\bar\mu$) to be  equal to $\stem(\dot{r}(\bar\mu))$ (by Lemma~\ref{lemma: reducingdensesets}), but this is a bounded subset of a measurable cardinal $\bar\mu$.
In particular, $k_1(r)\restriction \mu\forces^M_{j(\mathbb{P}_\kappa)_\mu}``\stem(k_1(\dot{r})(\mu))\s \bar\mu"$ because $\bar\mu\in \Gamma_{\iota}\s \Gamma_{\iota_1}$ and Lemma~\ref{lemma:imagesofmu} ensures that those ordinals are closure points of the factor embedding $k_1\colon V\rightarrow N_1$.  Since $k_1(r)\in H$,  $b_{\stem(k_1(\dot{r}(\mu)))_{H\restriction \mu}}\s b_{\mu}\cap \bar\mu=b_{\langle \delta_0,\dots, \delta_{n-1}\rangle}$ and we are done.
\end{proof}
Notice that this claim yields the sought contradiction. Indeed, we were originally assuming that $\bar\mu$ was such that no $\leq_j$-extension $\iota$ of $\iota_0$ makes $\bar{H}^{\iota^*,k^*}_{\bar\mu}$ have a $\sq$-maximal element, yet $\iota_1$ witnesses the  opposite conclusion.
\end{proof}

With the proof of Lemma~\ref{lem;maximal-stem} completed, we have now provided the missing piece needed to establish the \textbf{Coding Lemma}. 

\section{The $\kappa$-directednnes of $\langle \mathfrak{U}_\kappa,\leq_{\mathrm{RK}}\rangle$ from optimal assumptions}\label{sec: A model for gluing}

The next is the main theorem of this paper:
\begin{theorem}[$V=\mathcal{K}$]\label{Theoremlambdagluing}
Suppose that $\kappa$ is a measurable cardinal with $$\omega_1\leq \cf(o(\kappa))=o(\kappa)\leq \kappa.$$
and that there are no  measurable cardinals $\alpha<\kappa$ with  {$o(\alpha){\geq}\min\{\alpha,o(\kappa)\}$.}

Then, in a cardinal-preserving generic extension, $\kappa$ has the ${<}o(\kappa)$-$\gp$. 
\end{theorem}
\begin{proof}
     Let $\mathcal{U}=\langle U(\alpha,\zeta)\mid \alpha\leq \kappa,\, \zeta<o^{\mathcal{U}}(\alpha)\rangle$ denote the unique coherent sequence of measures in $\mathcal{K}$. Let $G\s \mathbb{P}_\kappa$ be generic over $\mathcal{K}[S]$ being $\mathbb{P}_\kappa$ {the gluing iteration defined through} \S\ref{sec: the gluing iteration}. Suppose that $\langle \mathcal{U}_\beta\mid \beta<{\omega^{\Omega}}\rangle\in \mathcal{K}[S\ast G]$ is a sequence of $\kappa$-complete measures over $\kappa$ and $\Omega<o(\kappa)$ an ordinal. Utilizing Lemma~\ref{lemma:omega-gluing-by-a-measure}, we  show that this sequence of measures can be glued via a measure on $\kappa^\Omega$ – in particular, the $\Omega$-gluing property will hold. 

     \smallskip
     
   In light of our anti-large cardinal assumption (i.e., {there is no measurable $\alpha<\kappa$ with  $o(\alpha)\geq \alpha$}) we can appeal to the  \textbf{Coding Lemma} (Lemma~\ref{lemma: coding lemma}) thus concluding the existence of a sequence of codes $\langle c_\beta\mid \beta<\omega^{\Omega}\rangle\in H(\kappa^+)^{\mathcal{K}}$  where  each $c_\beta$ codes the corresponding ultrafilter $\mathcal{U}_\beta$. (See Definition~\ref{def: codes}.) 
   
  For each $\ell$-order $1\leq\rho<\Omega\leq \kappa$ and   $\beta$ such that $\omega^\rho\cdot\beta<\omega^{\Omega}$ we denoted
\begin{equation*}
    \vec{c}^{\;\rho}_\beta:=\langle c_\gamma\mid \gamma\in [\omega^\rho\cdot\beta, \omega^\rho\cdot(\beta+1))\rangle,
\end{equation*}
\begin{equation*}
    \vec{\zeta}^{\;\rho}_\beta:=\langle \zeta^\sigma_\gamma\mid 1\leq\sigma<\rho\,\wedge\,\gamma\in [\omega^\sigma\cdot\beta, \omega^{\sigma}\cdot(\beta+1))\rangle.
\end{equation*}

   Take $\zeta^1_\beta<o^{\mathcal{U}}(\kappa)=o(\kappa)$ be an ordinal above all the ordinals mentioned in the codes $c_\gamma\in\vec{c}^{\;1}_\beta$. This choice is possible by regularity of $o(\kappa)$ and the fact that $o(\alpha)<o(\kappa)$ for all measurables $\alpha<\kappa$.
   In general, for each $1\leq \rho<\Omega$ and each $\beta$ such that $\omega^\rho\cdot\beta<\omega^{\Omega}$ we let  $\zeta^\rho_\beta\in (\sup\vec{\zeta}^{\,\rho}_\beta, o(\kappa)).$ 

   \smallskip

    Arguing as in \cite[Theorem~6.21]{HP} we can lift $\iota_{\kappa,\zeta^\rho_\beta}$ to $$\iota_{\kappa,\zeta^\rho_\beta}\colon \mathcal{K}[S]\rightarrow \mathscr{N}_{\kappa,\zeta^\rho_\beta}[\iota_{\kappa,\zeta^\rho_\beta}(S)]$$ in a way that $\iota_{\kappa,\zeta^\rho_\beta}(\ell)(\kappa)=\langle \vec{c}^{\;\rho}_\beta,\vec{\zeta}^{\,\rho}_\beta\rangle.$ Besides, as in \cite[Claim~6.21.2]{HP}, $$\mathscr{N}_{\kappa, \zeta^\rho_\beta}[\iota_{\kappa,\zeta^\rho_\beta}(S)\ast G]\models \text{$``\vec{c}^{\;\rho}_\beta$ is a sequence of $\kappa$-codes''}.$$

    Since $\Omega<o(\kappa)$ and the latter is a regular cardinal, we can let $\zeta_*<o(\kappa)$  above  $\sup\{\zeta^\rho_\beta\mid 1\leq \rho<\Omega\,\wedge\, \omega^\rho\cdot\beta<\omega^\Omega\}$. As before, we can lift the ultrapower embedding $\iota_{\kappa,\zeta_*}$ to $\iota_{\kappa,\zeta_*}\colon \mathcal{K}[S]\rightarrow \mathscr{N}_{\kappa, \zeta_*}[\iota_{\kappa,\zeta_*}(S)]$ in a way that $$\iota_{\kappa,\zeta_*}(\ell)(\kappa)=\langle \langle c_\beta\mid \beta<\omega^\Omega\rangle, \langle \zeta^\rho_\beta\mid 1\leq \rho<\Omega\,\wedge\,\omega^\rho\cdot\beta<\omega^\Omega\rangle \rangle.$$
{Notice also that the sequence of $U(\kappa,\zeta^\rho_\beta)$'s belongs to $\mathscr{N}_{\kappa,\zeta_*}$ because this latter model is closed under $\kappa$-sequences and $\omega^\Omega\leq \kappa$.} 
 Combining this with the above we infer that $o^\ell(\kappa)=\Omega$ in $\mathscr{N}_{\kappa,\zeta_*}[\iota_{\kappa,\zeta_*}(S)]$ (see Lemma~\ref{lemmacoherency}). Therefore,  the iteration $\iota_{\kappa,\zeta_*}(\mathbb{P}_\kappa)$ at stage $\kappa$ opts for the gluing poset $\mathbb{G}^\Omega_\kappa$. (Notice that $\mathbb{G}^\Omega_\kappa$  is computed the same way in $\mathscr{N}_{\kappa,\zeta_*}[\iota_{\kappa,\zeta_*}(S)\ast G]$ and in $\mathcal{K}[S\ast G]$ because the former model contains the measures to define $\mathbb{G}^\Omega_\kappa$.)

    \smallskip

    Define the gluing measure $\mathcal{W}$ as follows.

    \smallskip
    
    For each $X\s \kappa^\Omega$ in $\mathcal{K}[S\ast G]$, $X$ belongs to $\mathcal{W}$ if and only if there is  $p\in G$ and a $\mathbb{P}_\kappa$-name $\dot{T}$ for a tree  such that $p\forces_{\mathbb{P}_\kappa}``\langle \varnothing,\dot{T}\rangle \in \dot{\mathbb{G}}^\Omega_\kappa"$ and
    $$ \mathscr{N}_{\kappa, \zeta_*}[\iota_{\kappa,\zeta_*}(S)]\models ``p^\smallfrown\langle \varnothing,T\rangle^\smallfrown \iota_{\kappa,\zeta_*}(p)\setminus (\kappa+1)\forces_{\mathbb{P}_\kappa}\dot{c}_\kappa\in \iota_{\kappa,\zeta_*}(\dot{X})",$$
    where $\dot{c}_\kappa$ is the standard name for the generic sequence added by $\mathbb{G}^\Omega_\kappa$.

  It is routine to check that $\mathcal{W}$ is a $\kappa$-complete ultrafilter in $\mathcal{K}[S\ast G]$. 

  \smallskip

We will utilize our characterization of the $\Omega$-gluing property stated in Lemma~\ref{lemma:omega-gluing-by-a-measure} to show that $\mathcal{W}$ glues the departing sequence $\langle \mathcal{U}_\beta\mid \beta<\omega^\Omega\rangle.$

\smallskip

  Let 
  $X\in\mathcal{W}$ and fix $\beta<\omega^\Omega$. Suppose towards a contradiction that $$Y:=\{s(\beta)\mid s\in X\}\notin \mathcal{U}_\beta.$$ Let $p\in G$ be a condition forcing $``\dot{\mathrm{eval}_\beta}``\dot{X}=\dot{Y}\notin\dot{\mathcal{U}}_\beta$".  
   
   Since $X\in \mathcal{W}$, by extending $p$ if necessary, we may assume that 
   $$\mathscr{N}_{\kappa, \zeta_*}[\iota_{\kappa,\zeta_*}(S)]\models``p^\smallfrown\langle \varnothing,T\rangle^\smallfrown \iota_{\kappa,\zeta_*}(p)\setminus (\kappa+1)\forces_{\mathbb{P}_\kappa}\dot{c}_\kappa\in \iota_{\kappa,\zeta_*}(\dot{X})".$$
   In particular, 
   $$\mathscr{N}_{\kappa, \zeta_*}[\iota_{\kappa,\zeta_*}(S)]\models ``p^\smallfrown\langle \varnothing,T\rangle^\smallfrown \iota_{\kappa,\zeta_*}(p)\setminus (\kappa+1)\forces_{\mathbb{P}_\kappa}\dot{c}_\kappa(\beta)\in \iota_{\kappa,\zeta_*}(\dot{Y})\cap \kappa".$$
   (Here we used the fact that $\iota_{\kappa,\zeta_*}(\beta)=\beta$, which is true as $\beta<\kappa=\crit(\iota_{\kappa,\zeta_*})$.)
   \begin{claim}
   There is $p\in G$ such that $\iota_{\kappa,\zeta_*}(p)\forces_{\iota_{\kappa,\zeta_*}(\mathbb{P}_\kappa)}\iota_{\kappa,\zeta_*}(\dot{Y})\cap\kappa=\dot{Y}$.
   \end{claim}
   \begin{proof}
   To streamline the presentation  denote the embedding $\iota_{\kappa,\zeta_*}$ by $\iota.$

   Let $g\s \mathbb{G}_\kappa^\Omega$ generic over $V[G]$. Working in $V[G\ast g]$, for each $\alpha<\kappa$ we look at the following set: $$D_\alpha:=\{q\in \iota(\mathbb{P}_\kappa)/{G\ast g}\mid q\forces_{\iota(\mathbb{P}_\kappa)/{G\ast g}} \alpha\in \check{Y}\leftrightarrow \alpha\in \iota(\dot{Y})_{{G\ast g}}\}.$$
   Note that $D_\alpha$ is a $\leq^*$-dense open subset of $\iota(\mathbb{P}_\kappa)/{G\ast g}$  for all $\alpha<\kappa$. 

   \smallskip

   Let $d_\alpha\colon \kappa\rightarrow V$ be a function representing $D_\alpha$; to wit, $D_\alpha=\iota(d_\alpha)(\kappa)_{{G\ast g}}$. Notice that we may assume, without losing any generality, that $d_\alpha(\eta)$ is a $\mathbb{P}_{\eta+1}$-name for a $\leq^*$-dense open subset of $\mathbb{P}_\kappa/G_{\eta+1}$.

   \smallskip

Our goal is to show that the following set is $\leq^*$-dense:
$$\textstyle E=\{p\in \mathbb{P}_\kappa\mid \exists C\in \mathrm{Cub}_\kappa\,\wedge\,\forall \eta\in C\, (p\restriction \eta+1\forces_{\mathbb{P}_{\eta+1}}``p\restriction[\eta+1,\kappa)\in \bigcap_{\delta<\eta}d_\delta(\eta)")\}.$$
If this is the case then any $p\in E\cap G$ will satisfy that 
$\iota(p)\forces_{\iota(\mathbb{P}_\kappa)}\dot{Y}=\iota(\dot{Y})\cap \kappa.$

\smallskip

So, fix $p\in \mathbb{P}_\kappa$. We define via a fusion argument a sequence of triples $\langle \langle p_\delta, \nu_\delta, C_\delta\rangle \mid \delta<\kappa\rangle$ as follows. Using the Strong Prikry Property relative to $\nu_0:=0$ and the function $d_0$ we find $p_0\leq^* p$, $p_0\restriction \nu_0+1=p\restriction\nu_0+1$ and a club $C_0\s \kappa$ (which we may assume $C_0\cap \supp(p_0)=\emptyset$) such that 
$$p_0\restriction\eta+1\forces_{\mathbb{P}_{\eta+1}}p_0\restriction[\eta+1,\kappa)\in d_0(\eta),$$
for all $\eta\in C_0.$

\smallskip

Suppose that $\langle \langle p_\eta, \nu_\eta, C_\eta\rangle \mid \eta<\delta\rangle$ has been defined. 

\smallskip

$\br$ \underline{Case $\delta=\bar{\delta}+1$:} In this case we proceed as before. More explicitly, we  invoke the Strong Prikry Property with respect to $\nu_{\delta}:=\min(C_{\bar{\delta}}\setminus \nu_{\bar{\delta}}+1)$ and the function $d_\delta$, thus finding $p_\delta\leq^* p_{\bar{\delta}}$ and a club $C_{\delta}\s C_{\bar{\delta}}$  such that
\begin{itemize}
    \item $C_{\delta}\cap \supp(p_\delta)=\emptyset,$
    \item $p_\delta\restriction\nu_{\delta}+1=p_{\bar\delta}\restriction\nu_\delta+1$,
    \item and $p_{\delta}\restriction\eta+1\forces_{\mathbb{P}_{\eta+1}}``p_{\delta}\restriction[\eta+1,\kappa)\in d_\delta(\eta)"$ for all $\eta\in C_\delta.$
\end{itemize}

$\br$ \underline{Case $\delta$ limit:} In this case we stipulate $\nu_\delta=\sup_{\eta<\delta}\nu_\eta$ and take
$$\textstyle \bar{p}_{\delta}\restriction\nu_\delta+1=\bigcup_{\eta<\delta}p_\eta\restriction\nu_\eta+1$$
and $\bar{p}_{\delta}\restriction [\nu_\delta+1,\kappa)$ a $\leq^*$-lower bound for the sequence of tails $p_\eta\restriction[\nu_\eta+1,\kappa).$ Let $\bar{p}_\delta$ be $\bar{p}_\delta\restriction \nu_\delta+1{}^\smallfrown p_\delta\restriction [\nu_{\delta+1},\kappa)$. Note that  $\bar{p}_\delta$ is a legitimate condition in $\iota(\mathbb{P}_\kappa)$ (i.e., the support of $\bar{p}_\delta$ is nowhere stationary) as witnessed by the club $\bar{C}_\delta:=\bigcap_{\bar{\delta}<\delta} C_{\bar\delta}.$  Now  invoke the Strong Prikry Property with respect to the ordinal $\nu_{\delta}$, the function $d_\delta$ and the condition $\bar{p}_\delta$, thus finding $p_\delta\leq^* \bar{p}_{{\delta}}$ and a club $C_{\delta}\s \bigcap_{\alpha<\delta}C_{\alpha}$  with the properties mentioned above.

\smallskip

In this fashion, we obtain a sequence $\langle \langle p_\delta, \nu_\delta, C_\delta\rangle \mid \delta<\kappa\rangle$. Now let 
$$\textstyle p^*:=\bigcup_{\delta<\kappa}p_\delta\restriction \nu_\delta+1,\;C^*:=C(\nu)\cap \triangle_{\delta<\kappa} C_\delta,$$
where $C(\nu)$ is the club of closure points of $\nu\colon \delta\mapsto \nu_\delta.$

\smallskip

We claim that $p^*\in E$ as witnessed by the club $C^*$. 

Let $\eta\in C$. Then, $\eta=\sup_{\delta<\eta}\nu_\delta$
and $\eta\in \bigcap_{\delta<\eta} C_\delta$. From this we deduce,  
$$\textstyle p^*\restriction \eta+1\forces_{\mathbb{P}_{\eta+1}}p^*\restriction [\eta+1,\kappa)\leq^* p_\delta\restriction[\eta+1,\kappa)\in d_\delta(\eta)$$
for all $\delta<\eta.$ By openess of $d_\delta(\eta)$ we finally infer 
  $$\textstyle p^*\restriction \eta+1\forces_{\mathbb{P}_{\eta+1}}p^*\restriction [\eta+1,\kappa)\in \bigcap_{\delta<\eta} d_\delta(\eta),$$
  for all $\eta\in C$. Thereby, $p^*\leq^* p$ is in $E$.
   \end{proof}
   Let $\vec{\rho}=\langle \rho_i\mid i\leq k\rangle$ be a decreasing sequence of ordinals with $\rho_k=0$ and a sequence of non-negative integers $ \vec{n}=\langle n_i\mid i\leq k\rangle$ such that $\beta=\sum_{i=0}^k\omega^{\rho_i}\cdot n_i$. 
   
 Working in $\mathcal{K}[S\ast G]$,  we denote by $\mathcal{T}({\vec\rho,\vec n})$ the collection of $\Omega$-coherent sequences $t\in T$ having exactly $n_i$-many ordinals of $\ell$-order $\rho_i$ and no ordinals of $\ell$-order $\sigma$ for $\sigma$'s outside $\range(\vec\rho).$ More formally, $t$ is such that
 $$|t_\sigma|:=\begin{cases}
     n_i, & \text{if $\sigma=\rho_i$;}\\
     0, & \text{otherwise.}
 \end{cases}$$
  For each such $t\in \mathcal{T}({\vec\rho,\vec n})$ the definition of $\mathbb{G}^\Omega_\kappa$ requires $$\textstyle \mathrm{Succ}_T(t)\in \bigcap_{\rho<\Omega}\mathcal{U}^\rho_{j_\rho}(t\restriction\rho)\s  \bigcap_{i\leq k}\mathcal{U}^{\rho_i}_{j_{\rho_i}}(t\restriction\rho_i)\footnote{By convention $\mathcal{U}^{\rho_k}_{j_k}(\varnothing)=\mathcal{U}_{j_k}$.}$$ where
  $j_{\rho_i}$ is the order-type of the collection of ordinals with $\ell$-order $\rho_i$ in the Gitik/Magidor generic sequence stemming from $t$; namely, 
  $$j_{\rho_i}:=\otp\{\nu\in c_{t}\mid o^\ell(\nu)=\rho_i\}.$$

  We can compute $j_{\rho_i}$ explicitly as follows:
   $$\text{$j_{\rho_0}:=n_0$ and $j_{\rho_{i+1}}:=\omega^{\rho_i-\rho_{i+1}}\cdot j_{\rho_i}+n_{i+1}$}.$$

   \smallskip

   Let us define an $\Omega$-tree $T^*\s T$ \emph{disjoint} from $Y$. We do this going  over the levels of $T$: For  $n\leq n^\star:=(\sum_{i\leq k} n_i)$ let us stipulate that $$\mathrm{Lev}_n(T)=\mathrm{Lev}_n(T^*).$$

\smallskip


For each $t\in \mathcal{T}(\vec{\rho},\vec{n})$, let $\biguplus_{i<\Omega} S^{\rho}_t$ be a partition of $\mathrm{Succ}_T(t)$ according to the $\ell$-ordering.   Since $Y$ is $\mathcal{U}_\beta$-small,  $S^{0}_t\cap (\kappa\setminus Y)\in\mathcal{U}_\beta$ and also
$$\bar{S}^\rho_t:=\{\nu\in S^\rho_t\mid \min(b_\nu\setminus (\max(t\restriction\rho)+1))\in (\kappa\setminus Y) \}\in \mathcal{U}^{\rho}_{j_\rho}(t\restriction \rho).$$
For each $t\in \mathcal{T}(\vec{\rho},\vec{n})$, define   $$\mathrm{Succ}_{T^*}(t):=(S^{0}_t\cap (\kappa\setminus Y))\cup\textstyle\bigcup_{\rho>0} \bar{S}^\rho_t.$$
In general, we stipulate $$\textstyle \mathrm{Lev}_{n^\star}(T^*):=(\bigcup_{t\notin \mathcal{T}(\vec\rho,\vec n)}\mathrm{Succ}_T(t))\cup (\bigcup_{t\in \mathcal{T}(\vec\rho,\vec n)}\mathrm{Succ}_{T^*}(s)).$$ 
The other levels of $T^*$ are obtained copying the nodes of $T$ above $\mathrm{Lev}_{n^\star}(T^*)$.


   \smallskip


   By extending $p\in G$, $\mathscr{N}_{\kappa, \zeta_*}[\iota_{\kappa,\zeta_*}(S)]\models ``p\forces_{\mathbb{P}_\kappa}\langle \varnothing, \dot{T}^*\rangle\leq^* \langle \varnothing, \dot{T}\rangle$''. So, 
   $$\mathscr{N}_{\kappa, \zeta_*}[\iota_{\kappa,\zeta_*}(S)]\models``p^\smallfrown\langle \varnothing,T^*\rangle^\smallfrown \iota_{\kappa,\zeta_*}(p)\setminus (\kappa+1)\forces_{\mathbb{P}_\kappa}\dot{c}_\kappa(\beta)\in \dot{Y}".$$
  But in fact $p^\smallfrown\langle \varnothing,T^*\rangle^\smallfrown \iota_{\kappa,\zeta_*}(p)\setminus (\kappa+1)$ forces the opposite. Contradiction. 
\end{proof}
\begin{cor}
Let $\kappa$ be a measurable cardinal and $\lambda$ a cardinal with $\lambda\leq \kappa$. The following theories are equiconsistent, modulo $\mathrm{ZFC}$:
\begin{enumerate}
    \item $\kappa$ has the ${<}\lambda^+$-gluing property.
    \item $o(\kappa)=\begin{cases}
        \lambda^+, & \text{if $\lambda<\kappa$;}\\
        \kappa, & \text{if $\lambda=\kappa$.}
    \end{cases}$
\end{enumerate}
\end{cor}
\begin{proof}
    (1) $\Rightarrow$ (2). If $\lambda<\kappa$  then the argument in \cite[Theorem~7.1]{HP} (see also \cite[Remark~7.2]{HP}) yields $o^{\mathcal{K}}(\kappa)\geq \lambda^+$. If $\kappa=\lambda$ then the same argument yields $o^{\mathcal{K}}(\kappa)\geq \kappa$. This establishes (1) $\Rightarrow$ (2).
    
    (2) $\Rightarrow$ (1). This is precisely Theorem~\ref{Theoremlambdagluing} when $\lambda<\kappa$. If $\lambda=\kappa$ then Theorem~\ref{Theoremlambdagluing} produces a generic extension where  $\kappa$ has the ${<}\kappa$-$\gp$ and Lemma~\ref{lemma:compactnessofGluing} implies that in the same model $\kappa$ has the ${<}\kappa^+$-$\gp$.
\end{proof}
The previous corollary combined with Theorem~\ref{thm: EquivalencewithRK} yields the following:
\begin{cor}
    Let $\kappa$ be a measurable cardinal and $\lambda<  \kappa$ be a cardinal. The following theories are equiconsistent, modulo $\mathrm{ZFC}$:
    \begin{enumerate}
        \item  $\langle\mathfrak{U}_\kappa,\leq_{\RK}\rangle$ is $\lambda^+$-directed.
        \item $``(\text{$\kappa$ is measurable with }o(\kappa)=\lambda^+)$".
    \end{enumerate}
     In addition, the following theories are equiconsistent modulo $\mathrm{ZFC}$:
  \begin{enumerate}
      \item $\langle\mathfrak{U}_\kappa,\leq_{\RK}\rangle$ is $\kappa^+$-directed.
      \item  $``(\text{$\kappa$ is measurable with }o(\kappa)=\kappa^+)$".
      \end{enumerate}
\end{cor}

Let us close this section providing  an upper bound for the existence of a cardinal with $\lambda$-$\gp$ for every $\lambda$. The argument is in essence due to Gitik (\cite[Theorem~2.1]{GitikOnMeasurables}) modulo the following fix –– $\kappa$ is not assumed to be $\kappa$-compact  (as Gitik did originally) but to be a ``gluing cardinal".
\begin{theorem}\label{thm: full gluing}
Suppose that $\kappa$ has the $\lambda$-$\gp$ for every cardinal $\lambda$. Then, there is an inner model for a strong cardinal.     
\end{theorem}
\begin{proof}
  Suppose otherwise – that there is no inner model with a strong cardinal. Let $\mathcal{K}$ be the corresponding Mitchell-Steel core model. By the core model theorem (see \cite{MitchellUpto,Schi}), the restriction to $\mathcal{K}$ of any non-trivial elementary embedding $j\colon V\rightarrow M$  is a normal iteration of extenders in $\mathcal{K}$. In particular, this applies to elementary embeddings where ${}^{\omega_1} M\cap V\s M$.

  By Gitik's \cite[Theorem~1.1]{GitikOnMeasurables}, in the iteration $j\restriction\mathcal{K}$ no extender was used more than $\omega_1$-many times. In particular   the images under those embeddings $$\mathcal{I}_\kappa:=\{j(\kappa)\mid j\colon V\rightarrow M,\, \crit(j)=\kappa,\, M^{\omega_1}\cap V\s M\}$$ form a set. Set $\lambda:=\sup(\mathcal{I}_\kappa).$

  Back in $V$, $\kappa$ has the $\lambda^+$-$\gp$.\footnote{This is exactly the place where Gitik's original assumptions have to be amended. Indeed, it is consistent for a  $\kappa$-compact to fail to satisfy the  $(2^\kappa)^+$-$\gp$ \cite{HP}.} Thus, given any $\kappa$-complete ultrafilter over $\kappa$ there is an elementary embedding $j\colon V\rightarrow M$ with $\crit(j)=\kappa$ and ${}^\kappa M\cap V\s M$, and an increasing sequence of ordinals $\langle \eta_\alpha\mid \alpha<\lambda^+\rangle$  such that $U=\{X\s \kappa\mid \eta_\alpha\in j(X)\}$. Thus, $j(\kappa)>\lambda^+$ –– a contradiction.
\end{proof}

\section{The failure of the Gluing Property in Gitik's model}\label{sec: no gluing in Gitiks}
In \cite{HP} we showed that if $o(\kappa)=\omega_1$ and the universe is the core model up to $o(\kappa)=\kappa^{++}$, then  after forcing with a non-stationary support iteration of Tree Prikry forcings, $\kappa$ possesses the $\aleph_0$-$\gp$. In light of  this it is natural to expect that the consistency of  the $\aleph_1$-$\gp$ will require forcing with a non-stationary support iteration of Gitik's posets $\mathbb{P}(\alpha,o(\alpha))$ from \cite{ChangingCofinalities}. After a series of seminar lectures delivered by the first author at the Hebrew University, Gitik suggested that the $\gp$ should hold in his model of \cite{ChangingCofinalities}. In this section we show that the $\aleph_0$-$\gp$ fails in Gitik's model. This shows that the Gluing Poset introduced in \S\ref{sec: the gluing iteration} is necessary to produce a model with arbitrary degrees of gluing below the measurable cardinal $\kappa.$

\smallskip

Through this section we suppose that $\kappa$ is a measurable cardinal and that   $$\mathcal{U}=\langle \mathcal{U}_{\alpha,\beta}\mid \alpha\leq \kappa,\,\beta<o^{\mathcal{U}}(\alpha)\rangle$$
is a coherent sequence of normal measures with $o^{\mathcal{U}}(\kappa)<\kappa.$ As usual, denote $$\dom_1(\mathcal{U}):=\{\alpha\leq \kappa\mid \exists \beta<o^{\mathcal{U}}(\alpha)\; (\mathcal{U}_{\alpha,\beta}\in \mathcal{U})\}.$$

We shall denote $\mathbb{P}^{\mathrm{E}}(\mathcal{U})$ (resp. $\mathbb{P}^{\mathrm{NS}}(\mathcal{U})$)  Gitik's Easton support iteration of \cite{ChangingCofinalities} (resp. the  non-stationary support version  considered in \cite{BenUng}).

\smallskip

The main theorem of this section reads as follows:

\begin{theorem}\label{thm: no gluing in gitiks}
    Assume that $V=\mathcal{K}$ and that there is no inner model for $``\exists\alpha\, (o(\alpha)=\alpha^{++})"$. Suppose that $\kappa$ is a measurable cardinal with $o(\kappa)<\kappa$ as witnessed by a coherent sequence $\mathcal{U}$. Let $\mathbb{P}\in \{\mathbb{P}^{\mathrm{E}}(\mathcal{U}), \mathbb{P}^{\mathrm{NS}}(\mathcal{U})\}$. Then, 
 $$\text{$\one \forces_{\mathbb{P}}``\check{\kappa}$ does not have the $\check{\omega}$-$\gp$".}$$
In particular, $\langle \mathfrak{U}_\kappa,\leq_{\mathrm{RK}}\rangle$ is not $\aleph_1$-directed.
\end{theorem}


\smallskip

Recall that $\mathbb{P}$ forces (non-trivially) only at measurables $\alpha\in \dom_1(\mathcal{U})$ and for those measurables the $\alpha$th-stage of $\mathbb{P}$, $\mathbb{P}(\alpha,o^{\mathcal{U}}(\alpha))$, introduces a generic club set $b_\alpha\s \alpha$ with $\otp(b_\alpha)=\omega^{o^{\mathcal{U}}(\alpha)}.$ In addition,  for each Mahlo  $\eta\leq \kappa$, every $V[\mathbb{P}_{\eta}]$-club on $\eta$ contains a $V$-club on $\eta$: This  is clear whenever $\mathbb{P}=\mathbb{P}^{\mathrm{E}}(\mathcal{U})$ (by the $\eta$-cc of $\mathbb{P}_\eta$). In the context of non-stationary support $\mathbb{P}^{\mathrm{NS}}(\mathcal{U})$, this was proved by Ben-Neria–Unger in \cite[Corollary~3.7]{BenUng}. 

\smallskip

The next is the Strong Prikry Lemma for Gitik's poset $\mathbb{Q}_\eta:=\mathbb{P}(\eta,o^{\mathcal{U}}(\eta))$. 

\begin{lemma}\label{lemma: SPP for Motis}
   Suppose $o^{\mathcal{U}}(\eta)<\eta.$ Let $p\in \mathbb{Q}_\eta$ and $D\s \mathbb{Q}_\eta$ be dense open. Then, there is a direct extension $q\leq^* p$ and a finite sequence of Mitchell orders $\vec \alpha\in o^{\mathcal
    U}(\eta)^{<\omega}$  such that $q\cat s\in D$ for every sequence  $s\in T^p$ with $\langle o^{\mathcal{U}}(s(i))\mid i<|s|\rangle=\vec \alpha$. 
\end{lemma}
\begin{proof}
    Let $p=\langle t, T\rangle$ and $D$ be as above. Set $o:=o^{\mathcal{U}}(\eta).$ The argument consists of an induction, lasting for $\omega$-many steps, where the original tree $T$ is refined to some  $T_\omega$ in such a way that  $\langle t, T_\omega\rangle$ has the desired property.

    \smallskip
    
   Let us describe the first step of this induction. For each  $\alpha<o$ define
    $$X^0_{\alpha}:=\{\delta\in \mathrm{Succ}_{T,\alpha}(\varnothing)\mid \exists S\, (\langle t^\smallfrown \langle\delta\rangle, S\rangle\leq \langle t, T\rangle\, \wedge\, \langle t^\smallfrown \langle\delta\rangle, S\rangle\in D)\}$$
    and $X^1_\alpha:=\mathrm{Succ}_{T,\alpha}(\varnothing)\setminus X^0_\alpha.$ 

    \smallskip
    
    Let $i_\alpha\in\{0,1\}$ be the unique index such that $X^{i_\alpha}_\alpha\in \mathcal{U}_{\eta,\alpha}(\varnothing).$ Let $T_1$ be the tree with $\varnothing \in T_1$, first level $\mathrm{Succ}_{T_1}(\varnothing):=\bigcup_{\alpha<o} X^{i_\alpha}_\alpha$ and for each $\delta\in \mathrm{Succ}_{T_1,\alpha}(\varnothing)$  the subtree   $(T_1)_{\langle \delta\rangle}$ is defined as follows:
    $$(T_1)_{\langle \delta\rangle}:=\begin{cases}
    S, & \text{$i_\alpha=0$;}\\
    T_{\langle \delta\rangle} & \text{$i_\alpha=1,$}
    \end{cases}$$
    being $S$ one of the witnessing trees.

    This procedure yields a condition $\langle t, T_1\rangle\leq^* \langle t, T\rangle$

    \smallskip

    Next we describe the construction of $\langle t, T_3\rangle$, assuming $\langle t, T_2\rangle$ is given. (This contains the main ideas to construct $\langle t, T_{n+1}\rangle$ for a general $n<\omega.$)

    \smallskip

    Fix orders $\alpha,\beta,\gamma<o$. For each $\delta_0\in \Succ_{T_2, \alpha}(\varnothing)$,  $\delta_1\in \Succ_{T_2,\beta}(\langle \delta_0\rangle)$ and $\delta_2\in \Succ_{T_2, \gamma}(\langle \delta_0,\delta_1\rangle)$ and an order $\varepsilon<o$ define $X^0_\varepsilon(\langle \delta_0,\delta_1,\delta_2\rangle)$ as follows:
    $$\{\eta\in \mathrm{Succ}_{T_2,\varepsilon}(\langle \delta_0,\delta_1,\delta_2\rangle)\mid \exists S\, (\langle t^\smallfrown \langle\delta_0,\delta_1,\delta_2, \eta\rangle, S\rangle\leq \langle t, T\rangle\, \wedge\, \langle t^\smallfrown \vec\delta^\smallfrown\langle\eta\rangle, S\rangle\in D)\}.$$
    Similarly, $X^1_\varepsilon(\langle \delta_0,\delta_1,\delta_2\rangle):=\Succ_{T_2,\varepsilon}(\langle \delta_0,\delta_1,\delta_2\rangle)\setminus X^0_\varepsilon(\langle \delta_0,\delta_1,\delta_2\rangle).$

    \smallskip

    Let $i_\varepsilon(\langle \delta_0,\delta_1,\delta_2\rangle)\in\{0,1\}$ be the unique index $i$ such that $$X^{i}_\varepsilon(\langle \delta_0,\delta_1,\delta_2\rangle)\in \mathcal{U}_{\eta,\varepsilon}(t^\smallfrown\langle \delta_0,\delta_1,\delta_2\rangle \restriction \varepsilon).$$

    Consider the function $$\varphi_{\gamma}(\langle \delta_0,\delta_1\rangle)\colon \delta\in \mathrm{Succ}_{T_2,\gamma}(\langle \delta_0,\delta_1\rangle)\mapsto \langle i_\varepsilon(\langle \delta_0,\delta_1,\delta\rangle)\mid \varepsilon < o\rangle\in\{0,1\}^o.$$
    By $\eta$-completeness of $\mathcal{U}_{\eta,\gamma}(t^\smallfrown \langle \delta_0,\delta_1\rangle\restriction \gamma)$ we find a large set $X_\gamma(\langle \delta_0,\delta_1\rangle)$ where the above function is constant with value $\langle i_{\gamma, \varepsilon}(\langle \delta_0,\delta_1\rangle)\mid \varepsilon<o\rangle$. (Note that this value does not depend on $\delta$ anymore, but rather on the fixed  order $\gamma$.)

    We can repeat this process with the function 
    $$\varphi_{\beta, \gamma}(\langle \delta_0\rangle)\colon \delta\in \mathrm{Succ}_{T_2,\beta}(\langle \delta_0\rangle)\mapsto \langle i_{\gamma,\varepsilon}(\langle \delta_0,\delta\rangle)\mid \varepsilon < o\rangle\in\{0,1\}^o,$$
    thus obtaining a large set  $X_{\beta,\gamma}(\langle \delta_0\rangle)$ where the funcion is constant with value $\langle i_{\beta, \gamma, \varepsilon}(\langle \delta_0\rangle)\mid \varepsilon<o\rangle$. Repeating exactly the same argument with 
     $$\varphi_{\alpha, \beta, \gamma}\colon \delta\in \mathrm{Succ}_{T_2,\alpha}(\varnothing)\mapsto \langle i_{\beta,\gamma,\varepsilon}(\langle \delta\rangle)\mid \varepsilon < o\rangle\in\{0,1\}^o$$
     we obtain $X_{\alpha,\beta,\gamma}\in \mathcal{U}_{\eta, \alpha}(t\restriction \alpha)$ where the function takes constant value $\langle i_{\alpha, \beta, \gamma, \varepsilon}\mid \varepsilon<o\rangle$. At this point, we have made the value of the indices $i_\varepsilon(\langle \delta_0,\delta_1,\delta_2\rangle)$ dependent only of $\alpha,\beta, \gamma$ – the orders of  $\delta_0,\delta_1,\delta_2.$ 

     \smallskip

     Let us now define the tree $T_3$. We stipulate that $\varnothing\in T_3$ and that
     \begin{itemize}
         \item For each $\alpha<o$, $\mathrm{Succ}_{T_3,\alpha}(\varnothing)=\bigcap_{\beta,\gamma<o}X_{\alpha, \beta,\gamma}$.\footnote{Note that this is $\mathcal{U}_{\eta,\alpha}(\varnothing)$-large because $o<\eta$, the completeness of the ultrafilters.}
         \item For each $\langle \delta_0\rangle\in T_3$ and $\beta<o$, $\mathrm{Succ}_{T_3,\beta}(\langle \delta_0\rangle)=\bigcap_{\gamma<o}X_{ \beta,\gamma}(\langle \delta_0\rangle)$.
          \item For each $\langle \delta_0,\delta_1\rangle\in T_3$ and $\gamma<o$, $\mathrm{Succ}_{T_3,\gamma}(\langle \delta_0,\delta_1\rangle)=X_{\gamma}(\langle \delta_0,\delta_1\rangle)$.
          \item For each $\langle \delta_0,\delta_1,\delta_2\rangle\in T_3$ and $\varepsilon<o$, $\mathrm{Succ}_{T_3,\varepsilon}(\langle \delta_0,\delta_1,\delta_2\rangle)=X^{i}_\varepsilon(\langle \delta_0,\delta_1,\delta_2\rangle)$ where $i=i_\varepsilon(\langle \delta_0,\delta_1,\delta_2\rangle).$
          \item For each $\langle \delta_0,\delta_1,\delta_2,\eta\rangle\in T_3$ define
          $$(T_3)_{\langle \delta_0,\delta_1,\delta_2,\eta\rangle}:=\begin{cases}
              S, & \text{if $i_{o^{\mathcal{U}(\eta)}}(\langle\delta_0,\delta_1,\delta_2\rangle)=0$;}\\
              (T_2)_{\langle \delta_0,\delta_1,\delta_2,\eta\rangle}, & \text{if $i_{o^{\mathcal{U}(\eta)}}(\langle\delta_0,\delta_1,\delta_2\rangle)=1$,}
          \end{cases}$$
          where $S$ is a tree witnessing $\eta\in X^0_{o^\mathcal{U}(\eta)}(\langle \delta_0,\delta_1,\delta_2\rangle).$
     \end{itemize}
     Clearly, $\langle t, T_3\rangle$ is a condition and $\langle t, T_3\rangle\leq^* \langle t, T\rangle.$

    \smallskip

    In general, we define a $\leq^*$-decreasing sequence $\langle \langle t, T_n\rangle\mid n<\omega\rangle$ of conditions following the previous procedure and we take $T_\omega:=\bigcap_{n<\omega} T_n.$ The sought condition is $q=\langle t, T_\omega\rangle$ and the sequence of orders $\vec \alpha\in o^{\mathcal{U}}(\eta)$ is determined as follows: By density, there is $\langle s, S\rangle\leq \langle t, T_\omega\rangle$ in $D$. Thus, there is $\langle \delta^*_0,\dots, \delta^*_{n-1}\rangle\in T_\omega$ such that $\langle s, S\rangle$ is equivalent to $\langle t^\smallfrown \langle \delta^*_0,\dots, \delta^*_{n-1}\rangle, S\rangle.$ Since $\langle t^\smallfrown \langle \delta^*_0,\dots, \delta^*_{n-1}\rangle, S\rangle\leq^* \langle s, S\rangle$, by openess, $\langle t^\smallfrown \langle \delta^*_0,\dots, \delta^*_{n-1}\rangle, S\rangle\in D$.

    Set $\langle \alpha_i\mid i<n\rangle:=\langle o^{\mathcal{U}}(\delta^*_i)\mid i<n\rangle$. We claim that for every sequence of ordinals  $\langle \delta_0,\dots,\delta_{n-1}\rangle\in T_\omega$ such that, for each $i<n$, $o^{\mathcal{U}}(\delta_i)=\alpha_i$, $$q\cat \langle \delta_0,\dots, \delta_{n-1}\rangle= \langle t^\smallfrown \langle \delta_0,\dots, \delta_{n-1}\rangle, (T_\omega)_{\langle \delta_0,\dots, \delta_{n-1}\rangle}\rangle\in D.$$

    To better illustrate the argument suppose that $n=4$. Thus, we have $\langle \delta_0,\delta_1,\delta_2,\delta_3\rangle\in T_\omega\s T_3.$ By definition of $T_3$, $\delta_0\in \bigcap_{\beta,\gamma<o} X_{\alpha_0,\beta,\gamma}\s X_{\alpha_0,\alpha_1,\alpha_2}$, $\delta_1\in X_{\alpha_1,\alpha_2}(\langle \delta_0\rangle)$ and $\delta_2\in X_{\alpha_2}(\langle \delta_0,\delta_1\rangle)$. By the choice of these sets, $i_{\alpha_3}(\langle \delta_0,\delta_1,\delta_2\rangle)=i_{\alpha_0,\alpha_1,\alpha_2, \alpha_3}=i_{\alpha_3}(\langle \delta^*_0,\delta^*_1,\delta^*_2\rangle)$. Note that, because $\delta^*_3\in X^0_{\alpha_3}(\langle \delta^*_0,\delta^*_1,\delta^*_2\rangle)$, it must be the case that  $i_{\alpha_0,\alpha_1,\alpha_2, \alpha_3}=0.$ In particular,
    $$(T_\omega)_{\langle \delta_0,\delta_1,\delta_2,\delta_3\rangle}\s (T_3)_{\langle \delta_0,\delta_1,\delta_2,\delta_3\rangle}=S$$
    is a witnessing tree for $\delta_3\in X^0_{\alpha_3}(\langle \delta_0,\delta_1,\delta_2\rangle).$ Therefore, $$\langle t^\smallfrown \langle \delta_0,\delta_1,\delta_2,\delta_3\rangle, (T_\omega)_{\langle \delta_0,\delta_1,\delta_2,\delta_3\rangle}\rangle\in D.$$
This finishes the proof of the Strong Prikry Lemma.
\end{proof}

The time is now ripe to establish the failure of the $\gp:$

\begin{proof}[Proof of Theorem~\ref{thm: no gluing in gitiks}]
    Since we assumed that $V=\mathcal{K}$, the coherent sequence $$\mathcal{U}=\langle \mathcal{U}_{\alpha,\beta}\mid \alpha\leq \kappa,\, \beta<o^{\mathcal{U}}(\kappa)\rangle$$ mentions all the $\alpha$-complete normal ultrafilters in $\mathcal{K}$, for all  $\alpha\leq \kappa$.

    \smallskip

    Let $G\s \mathbb{P}$ be generic over $V$. As per Gitik's analysis \cite{ChangingCofinalities} (or \cite{BenUng}), in $V[G]$ there is a collection of $\kappa$-complete ultrafilters over $\kappa$, $$\{\mathcal{U}_{\kappa,1}(t)\mid t\in [\kappa]^{<\omega}\cap V\,\wedge\, \text{$t$ is $1$-coherent}\},$$
   each of which has the following three properties:
    \begin{enumerate}
        \item[$(\aleph)$] $\mathcal{U}_{\kappa,1}(t)$ extends $\mathcal{U}_{\kappa,1}$.
          \item[$(\beth)$] $\{\alpha<\kappa\mid t\sq b_\alpha\}\in \mathcal{U}_{\kappa,1}(t).$
        \item[$(\gimel)$] $\mathcal{U}_{\kappa,1}(t)$ contains the club filter $\mathrm{Cub}_\kappa^{V[G]}$.
    \end{enumerate}

   (About $(\gimel)$: If $\mathbb{P}=\mathbb{P}^E(\mathcal{U})$ this follows combining the $\kappa$-cc of $\mathbb{P}^E(\mathcal{U})$, the normality of $\mathcal{U}_{\kappa,1}$ and clause (1). If $\mathbb{P}=\mathbb{P}^{\mathrm{NS}}(\mathcal{U})$ one replaces the $\kappa$-cc by the above-mentioned  fact about clubs proved in \cite[Corollary~3.7]{BenUng}.)

\smallskip
    
    Let $\{\alpha_n\mid n<\omega\}$ be a set (in $V$) of distinct inaccessible cardinals below $\kappa$. Thus, $\{\mathcal{U}_{\kappa,1}(\langle \alpha_n\rangle)\mid n<\omega\}$ consists of $\kappa$-complete ultrafilters concentrating on $V$-measurable cardinals $\alpha$ (with $o^{\mathcal{U}}(\alpha)=1$) whose associated Prikry sequence $b_\alpha$ starts with the $V$-inaccessible $\alpha_n$. For each $n<\omega$, denote $$\mathcal{U}_n:=\mathcal{U}_{\kappa,1}(\langle \alpha_n\rangle).$$
  We claim that in $V[G]$, $\langle \mathcal{U}_n\mid n<\omega\rangle$ cannot be glued.  

   \smallskip

   Suppose otherwise and let $j\colon V[G]\rightarrow M$ be an elementary embedding with $\crit(j)=\kappa$ and $M^\kappa \s M$, and $\langle \eta_n\mid n<\omega\rangle$ an increasing sequence of ordinals such that
   $\mathcal{U}_n=\{X\s \kappa\mid \eta_n\in j(X)\}$. By virtue of our anti-large hypothesis and the core model Theorem~\ref{CoreModelTheorem}, the restriction  $$j\restriction \mathcal{K}\colon \mathcal{K}\rightarrow \mathcal{K}^M$$ is a normal iteration via normal measures. Thus, $j\restriction\mathcal{K}$ is the direct limit of a directed system of elementary embeddings $$\langle \iota_{\alpha,\beta}\colon \mathcal{K}_\alpha\rightarrow \mathcal{K}_\beta\mid \alpha\leq \beta\leq\delta\rangle.$$ Set $\iota:=\iota_{0,\delta}$ and let $H\s \iota(\mathbb{P})$ a $\mathcal{K}^M$-generic filter for which $M=\mathcal{K}^M[H].$

   \smallskip

   Note that $\eta_n$ is a strong generator for the iteration $\iota$ – that is, $\eta_n\in \iota(C)$ for all $C\in \mathrm{Cub}_\kappa^{\mathcal{K}}$: Indeed,  $ \mathrm{Cub}_\kappa^{\mathcal{K}}\s \mathrm{Cub}_\kappa^{V[G]}\s \mathcal{U}_n$ and $\eta_n$ ``glues" $\mathcal{U}_n$. 

   Let $\eta_\omega:=\sup_{n<\omega}\eta_n$. Clearly, $\eta_\omega$ is a strong generator for $\iota$ as well.

\begin{claim}\hfill
     \begin{enumerate}
       \item There is $\alpha<\delta$ such that $\eta_\omega=\iota_\alpha(\kappa)$ and $\crit(\iota_{\alpha,\delta})=\iota_\alpha(\kappa)$.
       \item $\eta_\omega$  is a $\mathcal{K}^M$-measurable cardinal with $o^{\mathcal{U}}(\eta_\omega)<\eta_\omega$ (and, as such, a non-trivial stage of forcing for the iteration $\iota(\mathbb{P})\in \mathcal{K}^M$).
   \end{enumerate}
\end{claim}
\begin{proof}[Proof of claim]
   (1). This was established in \cite[Claims~7.1.2, 7.1.7]{HP}.

   (2). In \cite{HP} we showed that $\eta_\omega$ is the critical point of some stage of the iteration $j\restriction \mathcal{K}$ and, as such, $\mathcal{K}^M$-inaccessible. Also, $M^\omega\s M$, hence $\cf^M(\eta_\omega)=\omega<\eta_\omega$. Combining this with our anti-large cardinal assumption and Mitchell's covering theorem we deduce that $\eta_\omega$ is a $\mathcal{K}^M$-measurable.
\end{proof}

   Working in $M$ (i.e., in $\mathcal{K}^M[H]$) we let $b_{\eta_{\omega}}$ the Magidor sequence introduced by the generic $H$ at stage $\eta_\omega.$ If we show that $\langle \eta_n\mid n<\omega\rangle$ is eventually contained in $b_{\eta_\omega}$ we will get the desired contradiction: Suppose that for some $n^*<\omega$ and for each $n\geq n^*$, $\eta_n\in b_{\eta_{\omega}}$. First, by $o^{\mathcal{K}^M}(\eta_\omega)$-coherency, $b_{\eta_n}\sq b_{\eta_\omega}$. Second,  each of these $\eta_n$'s glue the corresponding measure $\mathcal{U}_n$ which  yields $\eta_n\in j(\{\alpha<\kappa\mid \langle\alpha_n\rangle\sq b_\alpha\})$ – equivalently, $\langle \alpha_n\rangle\sq b_{\eta_n}$.\footnote{Note the use of $\crit(j)=\kappa>\alpha_n.$} Note, however, that this is impossible, for it would imply that the first member of $b_{\eta_\omega}$ is $\alpha_n$ for all $n\geq n^*$, yet we assumed that these ordinals were different.

   \smallskip

   So everything amounts to show that $\langle \eta_n\mid n<\omega\rangle$ is eventually contained in the Magidor sequence $b_{\eta_\omega}.$ The intuition for why this is true comes from a well-known fact about Prikry forcing $\mathbb{Q}$; namely, if $Q\s \mathbb{Q}$ is generic and $A\in V[Q]$ is a subset of $\kappa$ that is almost included in every $C\in \mathrm{Cub}_\kappa^V$ then $A$ must be eventually contained in the Prikry sequence induced by $Q$.

       

 \smallskip

 We will show that $\langle \eta_n\mid n<\omega\rangle$ is fully contained in $b_{\eta_\omega}$. 

   \begin{claim}\label{claim: bounding ordinals}
       Suppose  $p\in \iota(\mathbb{P})_{\eta_\omega+1}$ and $\dot{\alpha}$ is a $\iota(\mathbb{P})_{\eta_\omega+1}$-name such that $$p\forces_{\iota(\mathbb{P})_{\eta_\omega+1}}\dot{\alpha}<\check{\eta}_\omega\,\wedge\, \dot{\alpha}\notin\dot{b}_{\eta_\omega}.$$
      Then, there exist $q\leq p$, $i<\omega$ and a $\iota(\mathbb{P})_{\eta_\omega}$-name  $\dot{g}\colon \eta_\omega^i \rightarrow \eta_\omega$ such that 
       $$q\forces_{\iota(\mathbb{P})_{\eta_\omega+1}} \exists x\in {(\dot{b}_{\eta_\omega}\cap \dot{\alpha})^i}\, (\dot{g}(x)=\dot{\alpha}).$$
   \end{claim}
   \begin{proof}[Proof of claim]
       Let $K\s \iota(\mathbb{P})_{\eta_\omega}$ generic over $\mathcal{K}^M$ with $p\restriction \eta_\omega\in K$. Our ground model for the rest of the argument is  $\mathcal{K}^M[K]$. Let $p:=(\dot{p}_{\eta_\omega})_K.$

       For simplicity of notations, let us assume that $p=\langle \varnothing, T^p\rangle.$

       In the ground model the forcing at stage $\eta_\omega$ (call it $\mathbb{Q}$) has the Strong Prikry Property. Thus, there is $q_0\leq^* p$ and a sequence of orders $\langle \alpha_0,\dots, \alpha_{n-1}\rangle$ such that for each $\vec\gamma=\langle \gamma_0,\dots,\gamma_{n-1}\rangle$, with $o^{\mathcal{U}}(\gamma_i)=\alpha_i$,  $q_0\cat \vec\gamma$ decides $\dot{\alpha}$. In particular, there is a function $f\colon [\eta_\omega]^{n}\rightarrow \eta_\omega$ such that for each sequence $\vec{\gamma}=\langle \gamma_0,\dots, \gamma_{n-1}\rangle\in T^{q_0}$ consisting of ordinals with  $o^{\iota(\mathcal{U})}$-orders $\langle \alpha_0,\dots, \alpha_{n-1}\rangle$, $q_0\cat \vec\gamma\forces \check{f}(\vec\gamma)=\dot{\alpha}$. Since $q_0\cat\vec\gamma\forces \dot{\alpha}\notin \dot{b}_{\eta_\omega}$ it must be  that
       $$f(\vec\gamma)\in (0,\gamma_0)\cup (\gamma_0,\gamma_1)\cup \dots \cup (\gamma_{n-1},\eta_\omega).$$
       For each $\vec\gamma\in T^{q_0}$ as before denote $i_{\vec\gamma}\in \{0, \dots, n\}$  the unique index for which $f(\vec\gamma)\in (\vec\gamma(i-1),\vec\gamma({i}))$. Using the completeness of the ultrafilters, one can argue as in Lemma~\ref{lemma: SPP for Motis} to  produce a ``large" subtree $S\s \mathrm{Lev}_n(T)$ where the map $\vec\gamma\mapsto i_{\vec\gamma}$ is constant with value $i:=i(\alpha_0,\dots, \alpha_{n-1})$.

       For each $\vec\gamma$ such that $\vec\gamma^\smallfrown\langle\gamma\rangle\in S$ for some $\gamma$, the function $$f_{\vec\gamma}\colon \gamma\mapsto f(\vec\gamma^\smallfrown \langle\gamma\rangle)\in \gamma_i$$
       is constant on a $\mathcal{U}_{\eta_\omega,\alpha_n}(\vec\gamma\restriction \alpha_n)$-large set, simply by $\eta_\omega$-completeness of the measure. 
      {Repeating this argument we can shrink $S$ to $S^*$ in such a way that for each $\langle \gamma_0,\dots, \gamma_{n-1}\rangle\in S^*$ the value of $f(\langle \gamma_0,\dots,\gamma_{n-1}\rangle)$ depends on $\langle \gamma_0,\dots, \gamma_{i}\rangle$ only. Let $g\colon \eta_\omega^{i+1}\rightarrow \eta_\omega$ be this function} and (to simplify notations) let us keep denoting the refined tree by $T^{q_0}.$

       \smallskip

Let us show that we can get rid of the dependence of the $i$th-coordinate of $\langle \gamma_0,\dots,\gamma_{n-1}\rangle$. If $o^{\mathcal{U}}(\gamma_i)=0$ this is immediate because of the normality of the ultrafilter $\mathcal{U}_{\eta_\omega,0}(\varnothing)$. So, suppose that $o^{\mathcal{U}}(\gamma_i)\neq 0$.

\begin{subclaim}
    For each $\langle \gamma_0,\dots, \gamma_{i-1}\rangle$ as above, the map
$$\pi_{\langle \gamma_0,\dots, \gamma_{i-1}\rangle}\colon \gamma \mapsto \min\{\mu\in b_\gamma\mid \mu>g(\langle \gamma_0,\dots,\gamma_{i-1}, \gamma\rangle),\, o^{\mathcal{U}}(\mu)=0\}$$ 
defines a Rudin-Keisler projection from $\mathcal{U}_{\eta_\omega, \alpha_i}(\langle \gamma_0,\dots, \gamma_{i-1}\rangle\restriction \alpha_i)$ to $\mathcal{U}_{\eta_\omega, 0}(\varnothing)$.
\end{subclaim}
\begin{proof}[Proof of claim]
The definition of the ultrafilters depend on the support of $\mathbb{P}$. Let us assume that $\mathbb{P}$ has Easton support. (The other case is analogous.)

\smallskip

    Let $X\in \mathcal{U}_{\eta_\omega, 0}(\varnothing)$ and $Y\in \mathcal{U}_{\eta_\omega, \alpha_i}(\langle \gamma_0,\dots, \gamma_{i-1}\rangle\restriction \alpha_i)$. 
    
    It suffices to show that 
    $Y\cap \pi_{\langle \gamma_0,\dots, \gamma_{i-1}\rangle}^{-1} X$ is non-empty.

By definition, there is $p\in K$, a name $\dot{\Upsilon}$ and $\alpha<\kappa^{+}$ such that
$$p^\smallfrown \{\langle \langle \gamma_0,\dots,\gamma_{i-1}\rangle\restriction \alpha_i,  \dot{\Upsilon}\rangle\}^\smallfrown r_\alpha\forces \kappa\in j_{\mathcal{U}_{\eta_\omega,\alpha_i}}(\dot{Y}).$$
Consider
$$X^*:=\{\beta\in \mathrm{Succ}_{T^*,0}(\varnothing)\mid \beta\in X\,\wedge\, \beta>j_{\mathcal{U}_{\eta_\omega, \alpha_i}}(g)(\langle \gamma_0,\dots, \gamma_{i-1}, \kappa\rangle)\}.$$
This is $\mathcal{U}_{\eta_\omega, 0}(\varnothing)$-large, by uniformity. Thus, there is a tree $\Upsilon'\s \dot{\Upsilon}_K$ such that $\mathrm{Succ}_{\Upsilon',0}(\varnothing)\s X^*$. Let $q\leq p$ be in $K$ forcing this property. Then,
$$q^\smallfrown \{\langle \langle \gamma_0,\dots,\gamma_{i-1}\rangle\restriction\alpha_i, \dot{\Upsilon}'\rangle\}^\smallfrown r_\alpha\forces \kappa\in j_{\mathcal{U}_{\eta_\omega,\alpha_i}}(\dot{Y})$$
and also
$$q^\smallfrown \{\langle \langle \gamma_0,\dots,\gamma_{i-1}\rangle\restriction\alpha_i, \dot{\Upsilon}'\rangle\}^\smallfrown r_\alpha\forces j_{\mathcal{U}_{\eta_\omega,\alpha_i}}(\pi_{\langle \gamma_0,\dots, \gamma_{i-1}\rangle})(\kappa)\in j_{\mathcal{U}_{\eta_\omega,\alpha_i}}(\dot{X}),$$
because of  our choice of $X^*.$ 

So, $Y\cap \pi^{-1}_{\langle \gamma_0,\dots, \gamma_{i-1}\rangle} X\in \mathcal{U}_{\eta_\omega,\alpha_i}(\langle \gamma_0,\dots, \gamma_{i-1}\rangle\restriction\alpha_i)$ and we are done.
\end{proof}
Now we proceed as follows. For each $\langle \gamma_0,\dots, \gamma_{i-1}\rangle\in T^{q_0}$ with orders $\langle \alpha_0,\dots, \alpha_{i-1}\rangle$ we let $X(\langle \gamma_0,\dots, \gamma_{i-1}\rangle)$ the image of $\mathrm{Succ}_{T^{q_0},\alpha_i}(\langle \gamma_0,\dots, \gamma_{i-1}\rangle)$  under $\pi_{\langle \gamma_0,\dots, \gamma_{i-1}\rangle}$. By the above claim, $X(\langle \gamma_0,\dots, \gamma_{i-1}\rangle)\in \mathcal{U}_{\eta_\omega,0}(\varnothing)$.

Clearly, the map $h_{\langle \gamma_0,\dots, \gamma_{i-1}\rangle}\colon X({\langle \gamma_0,\dots, \gamma_{i-1}\rangle})\rightarrow \eta_\omega$ given by
 $$h_{\langle \gamma_0,\dots, \gamma_{i-1}\rangle}\colon \pi_{\langle \gamma_0,\dots, \gamma_{i-1}\rangle}(\gamma)\mapsto g(\langle \gamma_0,\dots, \gamma_{i-1},\gamma\rangle)$$
 is regressive. By normality of $\mathcal{U}_{\eta_\omega,0}(\varnothing)$ we find  a large set $$Y({\langle \gamma_0,\dots, \gamma_{i-1}\rangle})\s X({\langle \gamma_0,\dots, \gamma_{i-1}\rangle})$$ 
 where $h_{\langle \gamma_0,\dots, \gamma_{i-1}\rangle}$ is constant – say, with value $g({\langle \gamma_0,\dots, \gamma_{i-1}\rangle}).$

 Therefore, $\pi^{-1}_{{\langle \gamma_0,\dots, \gamma_{i-1}\rangle}} Y({\langle \gamma_0,\dots, \gamma_{i-1}\rangle})$ is a $\mathcal{U}_{\eta_\omega,\alpha_i}({\langle \gamma_0,\dots, \gamma_{i-1}\rangle}\restriction\alpha_i)$-large set where the function $g(\langle \gamma_0,\dots, \gamma_{i-1},\bullet\rangle)$ is constant, and obtains the value $g({\langle \gamma_0,\dots, \gamma_{i-1}\rangle}).$

       \smallskip

       Using the above it is easy to define a condition $q\leq^* q_0$ such that
       $$q\cat \langle \gamma_0,\dots, \gamma_{n-1}\rangle\forces_{\mathbb{Q}_{\eta_\omega}} \dot\alpha=g(\langle \gamma_0,\dots, \gamma_{i-1}\rangle)$$
       for all $\langle \gamma_0,\dots, \gamma_{n-1}\rangle\in T^{q}$ with orders $\langle \alpha_0,\dots, \alpha_{n-1}\rangle$.

       A moment's reflection should convince our readers that the above yields
      $$q\forces_{\mathbb{Q}_{\eta_\omega}} ``\exists x\in (\dot{b}_{\eta_\omega}\cap { \dot\alpha)}^i\,\check{g}(x)=\dot\alpha".$$

  Since this holds in $\mathcal{K}^M[K]$, there is $p^*\leq p\restriction \eta_\omega$ in $K$ forcing it. Thus,  $p^*{}^\smallfrown \dot{q}$ and a name $\dot{g}$ such that $\dot{g}_K=g$ satisfy the properties of Claim~\ref{claim: bounding ordinals}.
   \end{proof}
Work now in $M$ (i.e., in $\mathcal{K}^M[H]$). To conclude the proof we prove:
\begin{claim}
   $\langle \eta_n\mid n<\omega\rangle$ is contained in $b_{\eta_\omega}$. 
\end{claim}
\begin{proof}[Proof of claim]
Suppose for the sake of contradiction that $n<\omega$ is such that $\eta_n\notin b_{\eta_\omega}$.  Invoke the previous claim to find a function $g_n\colon \eta_\omega^{i_n}\rightarrow \eta_\omega$ in $\mathcal{K}^{M}[H\restriction\eta_\omega]$ and $x_n\in (b_{\eta_\omega}\cap \eta_n)^{i_n}$ such that $g_n(x_n)=\eta_n$. 
Therefore, $$\eta_n\in g_n[\eta_n^{i_n}].$$

\smallskip

Let $C_n:=\{\alpha<\eta_\omega\mid \forall x\in [\alpha]^{i_n}\, g_n(x)<\alpha\}$. Clearly, $C_n$ is a club on $\eta_\omega$ in the model  $\mathcal{K}^M[H\restriction \eta_\omega]$. Let $C:=\bigcap_{n<\omega} C_n$. By the properties of the iteration $\iota(\mathbb{P})_{\eta_\omega}$, there is a club $D\s \eta_\omega$ in $\mathcal{K}^M$ such that $D\s C$. Then, there is $E\in \mathrm{Cub}^{\mathcal{K}}_{\kappa}$  such that $\iota(E)\s^* D\s C$: To see this, let $h \colon \eta_\omega \to \eta_\omega$ be a function such that its closure points are contained in $D$. As the iteration $\iota=j\restriction \mathcal{K}$ is normal, $h = \iota(g)(\zeta_0,\dots, \zeta_{n-1})$ for $\zeta_0, ..., \zeta_{n-1} < \eta_\omega$. Let $E$ be the closure points of $g$. It follows that  $\iota(E) \setminus \zeta_{n-1} \subseteq D\s C$.

Since each $\eta_n$ is a strong generator,  $\eta_n\in \iota(E)\s C_n$. But notice that this is a contradiction with the conclusion of the previous paragraph, where we showed that $\eta_n=g_n(x_n)$ for some $x_n\in [\eta_n]^{i_n}.$ 
\end{proof}
The above completes the proof of the theorem.
\end{proof}

\section{Open problems}\label{sec: open problems}
We close  the paper with a couple of open problems. The most accessible question,  given our current techniques and understanding, seems to be:
\begin{question}\label{que: gp via ult}
    Does $\kappa^+$-$\gp$ imply  $\kappa^+$-$\gp$ via ultrafilters? 
\end{question}
We conjecture that the answer is negative. 

\smallskip

Assuming  $2^\kappa=\kappa^{+}$,  Kunen  \cite[Theorem~2.3]{Ketonen} (or Comfort–Negrepontis \cite[Theorem~4.3]{ComfortNegrepontis}) in combination with Theorem~\ref{thm: EquivalencewithRK} show that $\kappa$-compact cardinals have the $\kappa^+$-$\gp$ via ultrafilters. However the optimal consistency strength is plausibly much weaker than that. Therefore, we ask:
\begin{question}
   What is the  consistency strength of the following, modulo $\mathrm{ZFC}$:
   \begin{enumerate}
       \item  ``There is a measurable cardinal $\kappa$ with the  $\kappa^+$-$\gp$".
       \item ``There is a measurable cardinal $\kappa$ with the  $\kappa^+$-$\gp$ via ultrafilters".
   \end{enumerate}
\end{question}

\begin{question}\label{que: full gluing}
    What is the consistency strength of the theory $\mathrm{ZFC}+$``There exists a measurable cardinal $\kappa$ with the $\lambda$-$\gp$ for all cardinals $\lambda$"?
\end{question}

In Theorem~\ref{thm: full gluing} we justified (leveraging an argument of Gitik) that a lower bound for the consistency of the above configuration is the existence of a strong cardinal. We conjecture that this is the exact consistency strength.


\bibliographystyle{alpha}
\bibliography{biblio}

\end{document}